\documentclass[final]{siamltex} 

\title{Single-pass Nystr\"{o}m approximation in mixed precision 
}
\author{Erin Carson\thanks{Faculty of Mathematics and Physics, Charles University. Both authors were supported by ERC Starting Grant No. 101075632, Charles University PRIMUS project no. PRIMUS/19/SCI/11, and by the Exascale Computing Project (17-SC-20-SC), a collaborative effort of the U.S. Department of Energy Office of Science and the National Nuclear Security Administration. The first author was additionally supported by Charles University Research program no. UNCE/SCI/023. } 
\and 
Ieva Dau\v{z}ickait\.{e}\footnotemark[1]}

\usepackage{amsmath}
\usepackage{amssymb}
\usepackage{graphicx}
\usepackage{epstopdf}
\usepackage{color}
\usepackage{graphicx}
\usepackage{lipsum}
\usepackage{graphicx}
\usepackage{subcaption}
\usepackage{hyperref}
\usepackage{algpseudocode}
\usepackage{algorithmicx}
\usepackage{mathtools}
\usepackage{multirow}
\usepackage[T1]{fontenc}

\usepackage[section]{algorithm}

\newcommand{\hA}{\widehat{A}}
\newcommand{\hB}{\widehat{B}}
\newcommand{\hC}{\widehat{C}}
\newcommand{\Amu}{A + \mu I}
\newcommand{\hAmu}{\widehat{A}_N + \mu I}
\newcommand{\hY}{\widehat{Y}}
\newcommand{\hy}{\widehat{y}}
\newcommand{\hF}{\widehat{F}}
\newcommand{\hf}{\widehat{f}}
\newcommand{\E}{\mathcal{E}}
\newcommand{\hlambda}{\widehat{\lambda}}

\newcommand{\hPisqr}{\widehat{P}^{-1/2}} 
\newcommand{\hPi}{\widehat{P}^{-1}}
\newcommand{\hP}{\widehat{P}}
\newcommand{\hU}{\widehat{U}}
\newcommand{\hnu}{\widehat{\nu}}
\newcommand{\lambmax}{\lambda_{max}}
\newcommand{\lambmin}{\lambda_{min}}
\newcommand{\sigmax}{\sigma_{max}}

\begin{document}

\maketitle

\renewcommand{\thefootnote}{\fnsymbol{footnote}}

\begin{abstract}
Low rank matrix approximations appear in a number of scientific computing applications. We consider the Nystr\"{o}m method for approximating a positive semidefinite matrix $A$. In the case that $A$ is very large or its entries can only be accessed once, a single-pass version may be necessary. In this work, we perform a complete rounding error analysis of the single-pass Nystr\"{o}m method in two precisions, where the computation of the expensive matrix product with $A$ is assumed to be performed in the lower of the two precisions. Our analysis gives insight into how the sketching matrix and shift should be chosen to ensure stability, implementation aspects which have been commented on in the literature but not yet rigorously justified. 

We further develop a heuristic to determine how to pick the lower precision, which confirms the general intuition that the lower the desired rank of the approximation, the lower the precision we can use without detriment.  We also demonstrate that our mixed precision Nystr\"{o}m method can be used to inexpensively construct limited memory preconditioners for the conjugate gradient method and derive a bound the condition number of the resulting preconditioned coefficient matrix. 
We present numerical experiments on a set of matrices with various spectral decays and demonstrate the utility of our mixed precision approach.  
\end{abstract}



\section{Introduction}

We consider the construction of a rank-$k$ approximation $A_N$ to a positive semidefinite matrix $A \in \mathbb{R}^{n \times n}$ of the form
\begin{equation}
     A_N = (A \Omega) (\Omega^T A \Omega)^{\dagger} (A \Omega)^T, 
     \label{eq:Nystromapprox}
\end{equation}
where $\Omega \in \mathbb{R}^{n \times k}$ is a sampling matrix and $\dagger$ denotes the Moore-Penrose pseudoinverse. 
When the matrix $A$ is symmetric positive semidefinite, then a high quality approximation can be obtained using the Nystr\"{o}m method as shown theoretically and experimentally in \cite{GittensMahoney2016}. The Nystr\"{o}m method, a randomised approach, arises in two forms based on column-sampling and general random projections. 
The column-sampling approach is often analysed and used in machine learning settings \cite{WilliamsSeeger,DrineasMahoney05} and the general projection version has been explored for, e.g., approximating matrices in a streaming model \cite{Tropp2017fixed,Tropp2019streaming} and preconditioning linear systems of equations \cite{AlDaasNystprec,Dauzickaite2021forcing,Frangella2021}. The choice of the approach depends on the application; see the discussion in, e.g., \cite{GittensMahoney2016,Frangella2021}. In general, randomised methods are powerful tools for obtaining low-rank matrix approximations and are discussed in extensive reviews \cite{Halko11,MahoneyRndAlgs,Martinsson20,tropp2023randomized}.  

In this paper, we focus on the projection-based approach and the case when products with $A$ are very expensive and are thus the bottleneck of the randomised method. Such a setting motivates the use of a single-pass variant that requires only one matrix-matrix product with $A$ to reduce the overall cost. These are also employed in the streaming model in which $A$ can only be accessed once. 

The increasing commercial availability of hardware with low and mixed precision capabilities has inspired much recent work in developing mixed precision algorithms that can exploit this hardware to improve performance \cite{abdelfattah2021}. For instance, the latest NVIDIA H100 GPUs offer IEEE double (64 bit), IEEE single (32 bit), IEEE half (16 bit), and even quarter (8 bit) floating point storage and computation. When using specialized tensor core instructions, quarter precision can offer up to 4 petaflops/s and half precision (fp16) up to 2 petaflops/s performance on a single H100 GPU, compared to 60 teraflops/s for double precision (fp64) \cite{h100}. See Table~\ref{tab:ieee_param} for the unit roundoff and range for some IEEE floating-point arithmetics.

\begin{table}[htbp]
    \centering
    \footnotesize{
\resizebox{\textwidth}{!}{   
    \begin{tabular}{l|c|c|c|c}
        Arithmetic &  $u$ & \multicolumn{3}{c}{Range}\\
        \cline{3-5}
                  &     & $x^s_{min}$ & $x_{min}$ & $x_{max}$ \\
         \hline
        fp16 (half)  &  $2^{-11} \approx 4.88 \times 10^{-4}$ & $5.96 \times 10^{-8}$ & $6.10 \times 10^{-5}$ & $6.55 \times 10^4$\\
        fp32 (single) & $2^{-24} \approx 5.96 \times 10^{-8}$ & $1.40 \times 10^{-45}$ & $1.18 \times 10^{-38}$ & $3.40 \times 10^{38}$\\
        fp64 (double) & $2^{-53} \approx 1.11 \times 10^{-16}$ & $4.94 \times 10^{-324}$ & $2.22 \times 10^{-308}$ & $1.80 \times 10^{308}$\\
    \end{tabular}}}
    \caption{Unit roundoff $u$ for IEEE floating point arithmetics and the smallest positive subnormal number $x^s_{min}$, smallest positive number $x_{min}$, and the largest positive number $x_{max}$.}
    \label{tab:ieee_param}
\end{table}

Given that in our particular setting, the matrix-matrix products with $A$ are overwhelmingly the dominant cost, we thus seek to further reduce this cost through the use of low precision. We develop a mixed precision single-pass Nystr\"{o}m approach in which storage of and computation with $A$ is performed at a precision lower than the working precision. 

The natural question to ask is how does using the lower precision affect the quality of the approximation. We prove a bound on the error of $\Vert A-\hA_N\Vert_F$, where $\hA_N$ is the approximation computed by the mixed precision algorithm. Intuitively, this can be bounded in terms of the deviation of the exact (infinite precision) Nystr\"{o}m approximation $A_N$ from $A$ and the deviation of $\hA_N$ from $A_N$. These bounds are a large overestimate, but allow us to develop a practical heuristic to determine how low a precision can safely be used so that the error of the exact Nystr\"{o}m approximation remains dominant. Numerical experiments show that the heuristic is useful for a range of problems with various spectral decay curves.

We also consider the impact of the low-precision computations in preconditioning. Our focus is on limited memory preconditioners that can be constructed using the Nystr\"{o}m approximation and have been used in \cite{Dauzickaite2021forcing,Frangella2021}. This is an area that can particularly benefit from using the single pass mixed precision approach by reducing the cost of constructing the preconditioner. We extend the bounds on the condition number of the preconditioned system given in \cite{Frangella2021} to account for finite precision error.

In summary, our contributions are as follows. We provide a full finite precision analysis for computing the Nystr\"{o}m approximation via Algorithm~\ref{alg:nystrom_reg_id}, where two different precisions can be used and the computation of the matrix-matrix product can be performed in a lower precision. This analysis results in a deterministic bound for the Frobenius norm of the total finite precision error and allows us to formulate a practical heuristic for how to choose the lower precision. Insights into preserving stability when selecting the sketching matrix and a shift for the Cholesky decomposition are presented. We also extend bounds for the condition number of a matrix preconditoned by a limited memory preconditioner constructed using the Nystr\"{o}m approximation to account for the finite precision error. The theoretical results are illustrated with numerical experiments on synthetic and application problems.

The paper is structured as follows. In Section~\ref{sec:Nystrom_approx}, we discuss known bounds for the exact single-pass Nystr\"{o}m approximation, and derive and analyze a mixed precision variant. Then in Section~\ref{sec:preconditioning} we consider the application of our mixed precision approach to constructing limited memory preconditioners. Numerical examples are presented in Section~\ref{sec:numerics} and we conclude the paper in Section~\ref{sec:conclusions}.

\section{Nystr\"{o}m approximation}\label{sec:Nystrom_approx}
 
The approximation $A_N$ in \eqref{eq:Nystromapprox} can be written using an orthogonal projector $P_{A^{1/2}\Omega} = (A^{1/2} \Omega) (\Omega^T A \Omega)^{\dagger} (A^{1/2} \Omega)^T$ as $A_N = A^{1/2} P_{A^{1/2}\Omega} A^{1/2}$ and thus the quality of the approximation depends on the range of $\Omega$. This motivates using $\Omega$ that depend on $A$, for example, $\Omega = A G$, where $G$ is a random test matrix. Such an approach may be infeasible when the products with $A$ are expensive or in the streaming model, and thus we do not consider it in this paper. 

Structured sampling matrices that are suitable for fast products with $A$ can be used and experiments show that they can give a good quality approximation; see, for example, Section 9 in \cite{Martinsson20}. However most of the theoretical results are for Gaussian matrices. These results can also be applied to an orthonormal model, where the test matrix $\Omega$ is chosen to be the $Q$ factor from the QR decomposition of a Gaussian matrix $G$. The resulting approximation is the same as when using $G$ in exact arithmetic, but orthogonal matrices are preferred for stability in finite precision when $k$ is large \cite{Tropp2017fixed}. We comment on this further in Section~\ref{sec:choosing_shift_sketching}.  Algorithm~\ref{alg:nystrom_reg_id} is based on the stable implementation that appears in \cite{LiLinderman2017,Tropp2017fixed} (although we note that there is no specification of precision in \cite{LiLinderman2017,Tropp2017fixed}). The complexity of Algorithm~\ref{alg:nystrom_reg_id} is $\mathcal{O}(Tk+k^2n)$, where $T$ is the time needed to compute a matrix-vector product with $A$ \cite{Frangella2021}.

Various deterministic and probabilistic bounds for the exact approximation error 
\begin{equation}\label{eq:Nystrom_exact_error}
    E = A - A_N
\end{equation}
appear in the literature, e.g.,  \cite{GittensMahoney2016,Tropp2017fixed,Frangella2021}. Let $ W \Lambda W^T$ be an eignevalue decomposition of $A$ such that
\begin{equation}\label{eq:A_eigendecomp}
    A = W \Lambda W^T = \begin{bmatrix}
        W_1 & W_2
    \end{bmatrix} \begin{bmatrix}
        \Lambda_1 & \\ & \Lambda_2
    \end{bmatrix} \begin{bmatrix}
        W_1^T \\ W_2^T,
    \end{bmatrix}
\end{equation}
where $\Lambda = \textrm{diag}(\lambda_1, \lambda_2, \dots, \lambda_n)$ with $\lambda_i \geq \lambda_{i+1}$ and $\lambda_n \geq 0$,  $\Lambda_1 = \textrm{diag}(\lambda_1, \lambda_2, \dots, \lambda_k)$ and $\Lambda_2 = \textrm{diag}(\lambda_{k+1}, \lambda_{k+2}, \dots, \lambda_n)$. Then if $W_1^T \Omega$ is full rank and $A$ is accessed only once, \cite[Theorem 3]{GittensMahoney2016} shows that
\begin{equation}\label{eq:exact_error_determ_frob}
    \Vert E \Vert_F \leq \Vert \Lambda_2 \Vert_F + \Vert \Lambda^{1/2}_2 (W_2^T \Omega) (W_1^T \Omega)^{\dagger} \Vert_2 \left(\sqrt{2 \textrm{trace}(\Lambda_2)} + \Vert \Lambda^{1/2}_2 (W_2^T \Omega) (W_1^T \Omega)^{\dagger} \Vert_F \right).
\end{equation}
The deterministic bounds are pessimistic and the error estimate can be improved by considering the expected error or bounds that hold with high probability \cite{GittensThesis}. For example, a recent result in 
\cite{Frangella2021} bounds the expected exact approximation error of a rank $k \geq 4$ approximation $A_N$ obtained by Algorithm~\ref{alg:nystrom_reg_id} by
\begin{equation}\label{eq:bound_expected_error_FTU}
    \mathbb{E}\, \| A - A_N \|_2 \leq \min_{2 \leq p \leq k-2} \left( \left( 1 + \frac{2(k-p)}{p-1} \right) \lambda_{k-p+1} + \frac{2e^2k}{p^2-1} \sum_{j=k-p+1}^n \lambda_j \right),
\end{equation}
where $\lambda_i \geq \lambda_{i+1}$ are eigenvalues of $A$ and $e$ is the exponential constant.

The finite precision error analysis of randomised methods is usually missing from the literature, with the notable exceptions of \cite{Nakatsukasa2020} which shows that a stabilised generalised Nystr\"{o}m approximation is stable and \cite{connolly2022randomized}, which provides a general finite precision analysis of the randomised SVD. Note that analysis in \cite{Nakatsukasa2020} requires using two independent random sampling matrices, which is not the case in Algorithm~\ref{alg:nystrom_reg_id}. 
In the following subsection, we introduce a mixed precision variant of the stabilised single-pass Nystr\"{o}m approximation and analyze the error in finite precision.

\begin{algorithm}
\caption{Stabilised single-pass Nystr\"{o}m approximation for symmetric positive semidefinite $A$ in precisions $u$ and $u_p$}\label{alg:nystrom_reg_id}
\hspace*{\algorithmicindent} \textbf{Input:} symmetric positive semidefinite matrix $A \in \mathbb{R}^{n \times n}$ stored in precision $u_p$, a sketching matrix $\Omega \in \mathbb{R}^{n \times k}$ stored in  precision $u$\\
 \hspace*{\algorithmicindent} \textbf{Output:} $U \in \mathbb{R}^{n \times k}$ with orthonormal columns approximating eigenvectors of $A$, and diagonal $ \Theta \in \mathbb{R}^{k \times k}$ with approximations to the largest $k$ eigenvalues of $A$ on the diagonal
\begin{algorithmic}[1]
\State  $Y=A \Omega $ \Comment{compute in $u_p$, store in $u$} \label{step:nystrom_AQ}
\State Compute a shift $\nu$ \Comment{$u$} 
\State Shift $Y$: $Y_{\nu} = Y + \nu \Omega $ \Comment{$u$} \label{step:nystrom_shift}
\State $B = \Omega^T Y_{\nu}$\Comment{$u$} \label{step:nystrom_QTY}
\State Compute the upper triangular Cholesky factor $C = \textrm{chol}((B + B')/2)$ \Comment{$u$} \label{step:nystrom_chol}
\State Solve $F = Y_{\nu}/C$ \Comment{$u$} \label{step:nystrom_triangular}
\State Compute the economy size SVD $[U,\Sigma, \sim] = \textrm{svd}(F,0)$ \Comment{$u$} \label{step:nystrom_svd}
\State Remove the shift: $\Theta = \textrm{max}(0,\Sigma^2 - \nu I)$ \Comment{$u$} \label{step:remove_shift}
\end{algorithmic}
\end{algorithm}

\subsection{Finite precision analysis}
We analyze Algorithm~\ref{alg:nystrom_reg_id} which uses two precisions, a potentially lower precision with unit roundoff $u_p$ for computing the expensive product with $A$, and a working precision $u$ for all other steps of the algorithm. Throughout the paper, hats are used to denote the computed versions of quantities and thus the computed approximation is $\hA_N$. The finite precision error is denoted 
\begin{equation}\label{eq:Nystrom_finite_error}
    \E = A_N - \hA_N.
\end{equation}
We aim to bound the Frobenius norm of  $\E$ in order to give a bound on the total error of the approximation via
\begin{equation*}
\Vert A - \hA_N \Vert_F = \Vert A-A_N +A_N -\hA_N \Vert_F \leq \Vert E \Vert_F + \Vert \E \Vert_F. 
\end{equation*}

We use a standard model of floating point arithmetic; see, e.g., \cite[Section 2.2]{HighamBook}. For matrix $A$, $\kappa(A)$ will be used to denote the condition number $\kappa(A)=\Vert A^{\dagger} \Vert \Vert A \Vert$ in an indicated norm. Note that in our analysis, we do not account for oversampling. We consider a general sketching matrix $\Omega \in \mathbb{R}^{n \times k}$ and some small shift $\hnu$. Discussion on how these should be chosen based on the finite precision analysis is presented in Section~\ref{sec:choosing_shift_sketching}. 

Our analysis requires the following assumptions:
\begin{enumerate}
    \item the SVD in step~\ref{step:nystrom_svd} of Algorithm~\ref{alg:nystrom_reg_id} is computed exactly; \label{assumpt_first}
    \item $u \leq u_p$;
    \item no overflow or underflow occurs;
    \item $\kappa_2(\widetilde{B}) \ll u^{-1}$, where $\widetilde{B} = fl \left( \frac{1}{2} (\hB + \hB^T)\right)$;
    \item $\hnu$ and $\Omega$ are set so that 
    \begin{gather}
        \Omega^T (A+ \hnu I )\Omega \textrm{ is symmetric positive definite (SPD) }; \label{eq:assumption_shifted_spd} \\
        \Vert \left( \Omega^T (A+ \hnu I )\Omega \right)^{-1} \Delta_2 \Vert_F \leq  \Vert \left( \Omega^T (A+ \hnu I )\Omega \right)^{-1} \Vert_F \Vert \Delta_2 \Vert_F  < 1,   \label{eq:assumption_1st_order_is_OK}
\end{gather} \label{assumpt_last}
where $\Delta_2$ is the finite precision error accumulated in steps~\ref{step:nystrom_AQ} to \ref{step:nystrom_chol};
\item $\hnu \leq c(n,k) u_p \Vert A \Vert_F  \Vert \Omega \Vert_F^2$, where $c(n,k)$ is a constant that depends on $k$ and $n$; and \label{assumption_shift_bound}
\item the precisions $u_p$ and $u$ are chosen so that
\begin{align}
    \kappa_2(A_k + \hnu I_k) \widetilde{\kappa}( \Omega)^2 & \ll u_p^{-2} \quad \textrm{and} \label{cond:u_p_kappaAk}\\
   \kappa_2(A_k + \hnu I_k) \widetilde{\kappa}( \Omega)^2 & \ll u^{-1} \label{cond:u_kappaAk},
\end{align}
where $A_k$ is the best rank-$k$ approximation of $A$ and
\begin{equation}\label{eq:kappa_tld_omega}
\widetilde{\kappa}(\Omega) \coloneqq \Vert \Omega \Vert_F \Vert \left( W_1^T \Omega \right)^{\dagger} \Vert_2. 
\end{equation}
\label{assumpt_simplify_u_up_kappaAk}
\end{enumerate}
The first assumption does not have a significant effect on the final bound when a numerically stable algorithm is used to compute the SVD, but it simplifies the analysis; note that this same assumption is used in \cite{connolly2022randomized}. Assuming that the sketching is performed in a lower or the same precision as other computations allows us to make the analysis easier to read by ignoring terms of order $u^2$. 
Note that since the sketching step is the only time $A$ is accessed, $A$ may also be stored in precision $u_p$. We assume all other quantities are stored in precision $u$. Assuming no overflow or underflow is a standard assumption, although we note that overflow and underflow can become increasingly common when very low precisions are used. The assumption on the condition number of $\widetilde{B}$ is standard. 
Assumption~\ref{assumpt_last} is essential to our analysis and we investigate how $\Omega$ and $\hnu$ should be chosen to satisfy \eqref{eq:assumption_shifted_spd} and \eqref{eq:assumption_1st_order_is_OK} in Section~\ref{sec:satisfying_assumptions}. The bound for $\hnu$ ensures that the shift does not increase the approximation error significantly. The final assumption requires that the ideal rank-$k$ approximation is well conditioned in the precisions used in Algorithm~\ref{alg:nystrom_reg_id}, and
allows us to simplify the presentation of the bounds.

Following \cite{HighamBook}, we define 
\begin{equation*}
    \gamma_n^{(p)} = \frac{nu_p}{1-nu_p}, \quad
     \widetilde{\gamma}_n^{(p)} =  \frac{c n u_p}{1 - c n u_p},
\end{equation*}
where $c$ is a small constant independent of $n$. The superscript $p$ is omitted when referring to the terms with $u$ instead of $u_p$. The notations $\approx$ and $\lesssim$ are used when dropping  second order terms that are insignificant in comparison to other terms in the expression.

\subsection{Preliminary results}\label{sec:prel_lemma}
Before we delve into the finite precision analysis, we first explore the weighted pseudoinverse 
\begin{equation}\label{eq:weighted_pseudoinv_def}
   X_A^{\dagger} \coloneqq A X \left( X^T A X \right)^{\dagger} 
\end{equation}
of a matrix $X \in \mathbb{R}^{n \times k}$, where $A \in \mathbb{R}^{n \times n}$ is symmetric positive semidefinite. Bounds on $\Vert X_A^{\dagger} \Vert_2$ independent of $A$ are available \cite{stewart1989scaled,forsgren1996linear}, but these can be arbitrarily large. Alternatively, we can obtain a bound that depends on $A$. We do this by noting that 
\begin{equation*}
    A X \left( X^T A X \right)^{\dagger} = A^{1/2} \left( X^T A^{1/2} \right)^{\dagger}
\end{equation*} 
and considering the eigenvalue decomposition as in \eqref{eq:A_eigendecomp}. Then
\begin{equation}\label{eq:psinv_interim}
     \left\lVert A^{1/2} \left( X^T A^{1/2} \right)^{\dagger} \right\rVert_2 =  \left\lVert   \Lambda^{1/2} W^T \left( X^T W \Lambda^{1/2} W^T \right)^{\dagger}  \right\rVert_2 \leq \Vert  \Lambda^{1/2} \Vert_2 \left\lVert \left( X^T W \Lambda^{1/2} \right)^{\dagger}  \right\rVert_2.
\end{equation}
Recall that
 $   \left\lVert \left( X^T W \Lambda^{1/2} \right)^{\dagger}  \right\rVert_2 = 1/\sigma_{min}(X^T W \Lambda^{1/2})$,
where $\sigma_{min}(X^T W \Lambda^{1/2})$ denotes the smallest singular value of $X^T W \Lambda^{1/2}$. We can write \linebreak
 $   X^T W \Lambda^{1/2} = \begin{bmatrix}
       X^T W_1  \Lambda_1^{1/2} & X^T W_2 \Lambda_2^{1/2}
    \end{bmatrix} $
and using the singular value interlacing property and the inequality on the singular values of a product of two square matrices, we obtain
\begin{multline}\label{eq:sigma_min_XT_W_lambda}
    \sigma_{min} (X^T W \Lambda^{1/2}) \geq \sigma_{min} \left( X^T W_1 \Lambda_1^{1/2} \right) \geq  \sigma_{min} ( X^T W_1 ) \sigma_{min} ( \Lambda_1^{1/2} ) \\ 
    = \Vert ( X^T W_1 )^{\dagger} \Vert_2^{-1} \Vert \Lambda_1^{-1/2} \Vert_2^{-1}.
\end{multline}
Combining this with \eqref{eq:psinv_interim} and using $A_k = W_1 \Lambda_1 W_1^T$, we obtain
\begin{equation}\label{eq:weighted_pseudinv_bound}
    \Vert X_A^{\dagger} \Vert_2 \leq \kappa_2(A_k)^{1/2} \Vert ( X^T W_1 )^{\dagger} \Vert_2.
\end{equation}

A bound on $\Vert  \left( X^T A X \right)^{-1}  \Vert_2 $ is required in our further analysis. Since \linebreak $\sigma_i(X^T A X) = \sigma_i(X^T A^{1/2})^2$, we use \eqref{eq:sigma_min_XT_W_lambda} to obtain
\begin{equation}\label{eq:bound_inverse_XTAX}
    \Vert  \left( X^T A X \right)^{-1}  \Vert_2 \leq \Vert ( X^T W_1 )^{\dagger} \Vert_2^2 \Vert \Lambda_1^{-1/2} \Vert_2^2 = \Vert ( X^T W_1 )^{\dagger} \Vert_2^2 \Vert A_k^{\dagger} \Vert_2.
\end{equation}

\subsection{Finite precision error bound}
Our main result in this section is the  following theorem.

\begin{theorem}\label{th:finite_precision_error}
Let $A \in \mathbb{R}^{n \times n}$ be a symmetric positive semidefinite matrix stored in precision $u_p$, and $\hA_N$ be its approximation computed by Algorithm~\ref{alg:nystrom_reg_id} using precision $u_p$ in step~\ref{step:nystrom_AQ} and precision $u \leq u_p$ in other steps. If the assumptions \ref{assumpt_first} - \ref{assumpt_simplify_u_up_kappaAk} are satisfied, then the total approximation error is bounded as
\begin{equation}\label{eq:total_bound_simplified}
    \Vert A - \hA_N \Vert_F \lesssim   \Vert (A + \hnu I) - (A+ \hnu I)_N \Vert_F  +  k^{1/2} \widetilde{\gamma}_n^{(p)} \kappa_2( A_k + \hnu I_k)  \widetilde{\kappa}(\Omega)^2 \Vert A \Vert_F,
\end{equation} 
where $A_k$ is the best rank-$k$ approximation of $A$, $(A+ \hnu I)_N$ is the exact Nystr\"{o}m approximation of $A+ \hnu I$ and $\widetilde{\kappa}(\Omega)$ is defined in \eqref{eq:kappa_tld_omega}.
\end{theorem}

We briefly comment on the bounds before stating the proof. The first term on the right-hand side bounds the error of the exact Nystr\"{o}m approximation of the shifted matrix $A + \hnu I$. The bounds for the exact Nystr\"{o}m approximation error of $A$ that depend on the eigenvalues of $A$ can be easily adapted to bound the error of approximating $A + \hnu I$ as the eigenvalues are shifted by $\hnu$ and the extra error due to the shift is absorbed by the second term in \eqref{eq:total_bound_simplified}. We now continue with the proof.

\begin{proof}
The roadmap of the proof is as follows. First, we account for the finite precision error terms in every step of the algorithm, bound their norm and track their influence on the subsequent computations. We then backtrack all the steps of the algorithm to find out how the finite precision error influences the computed approximation $\hA_N$. In order to express $\hA_N$ via $A_N$ we approximate a perturbed inverse of $\Omega^T (A + \hnu I) \Omega$ to the first order. The final bound for the norm of the total finite precision error makes use of the bounds for the norms of the weighted pseudoinverse and the finite precision error terms obtained in the first part of the proof. We simplify the bound under reasonable assumptions. The proof is provided here with some extra details and cumbersome expressions deferred to Appendix \ref{sec:app_fp}.

\emph{Part 1. Step-by-step analysis.}
In step~\ref{step:nystrom_AQ}, we compute
\begin{gather}
    \hY = A \Omega + \Delta, \textrm{ where } \label{eq:Yhat} \\
    \Vert \Delta \Vert_F \leq \gamma_n^{(p)} \Vert A \Vert_F \Vert \Omega \Vert_F. \label{eq:delta}
\end{gather}
Applying the shift $\hnu$ gives
\begin{gather}
    \hY_{\nu} = \hY + \hnu \Omega + \Delta_{\nu}, \textrm{ where } \label{eq:Yshifted_compputed}\\
    \Vert \Delta_{\nu} \Vert_F \leq u \Vert \hY + \hnu \Omega \Vert_F \leq \gamma_n (1+ \gamma_n^{(p)}) \Vert A \Vert_F \Vert \Omega \Vert_F + \gamma_n \hnu \Vert \Omega \Vert_F, \label{eq:delta_nu}
\end{gather}
and the second inequality is due to \eqref{eq:Yhat} and \eqref{eq:delta}. We also require a bound on $\Vert \hY_{\nu} \Vert_F$ in the following analysis, which can be given as 
\begin{equation}\label{eq:Yhat_nu_bound}
    \Vert \hY_{\nu} \Vert_F \leq  (1 + \gamma_n^{(p)} + \gamma_n + \gamma_n \gamma_n^{(p)})\Vert A \Vert_F \Vert \Omega \Vert_F + (1+\gamma_n)\hnu \Vert \Omega \Vert_F. 
\end{equation}
In step~\ref{step:nystrom_QTY}, we obtain a $k \times k$ matrix
\begin{gather}
    \hB = \Omega^T \hY_{\nu} + \Delta_B, \textrm{ where } \label{eq:B=QTY}\\
    \Vert \Delta_B \Vert_F \leq  \gamma_n \Vert \Omega \Vert_F \Vert \hY_{\nu} \Vert_F  \leq  \gamma_n (1 + \gamma_n + \gamma_n^{(p)}  )\Vert A \Vert_F \Vert \Omega \Vert_F^2 + \gamma_n (1 + \gamma_n )\hnu \Vert \Omega \Vert_F^2 \nonumber \\
    \lesssim  (\gamma_n  + \gamma_n \gamma_n^{(p)}  )\Vert A \Vert_F \Vert \Omega \Vert_F^2 + \gamma_n \hnu \Vert \Omega \Vert_F^2. \label{eq:delta_B}
\end{gather}
The algorithm continues by forming a symmetric matrix
\begin{gather}
    \widetilde{B} = \frac{1}{2} (\hB + \hB^T) + \Delta_{s}, \textrm{ where } \label{eq:forming_Btld} \\
    \Vert  \Delta_{s} \Vert_F \leq u \Vert \hB \Vert_F \lesssim  (\gamma_n  + \gamma_n \gamma_n^{(p)}  )\Vert A \Vert_F \Vert \Omega \Vert_F^2 + \gamma_n \hnu \Vert \Omega \Vert_F^2, \label{eq:delta_s}
\end{gather}
and computing its Cholesky decomposition. From \cite[Theorem 10.3]{HighamBook} we have that
\begin{gather}
     \widetilde{B} + \Delta_{Ch} = \hC^T \hC, \textrm{ where } \label{eq:chol}\\
     \Vert \Delta_{Ch} \Vert_F \leq \gamma_{k+1} \Vert \hC  \Vert_F^2 \label{eq:delta_ch}
\end{gather}
and $\hC$ is the computed upper triangular factor of the Cholesky decomposition. We note that to successfully compute the Cholesky decomposition, we need to ensure that $ \widetilde{B}$ is positive definite. This is related to satisfying the condition \eqref{eq:assumption_1st_order_is_OK} and we discuss it in Section~\ref{sec:successful_cholesky}. We bound $\Vert \hC \Vert_F$ in terms of $A$ and $\Omega$ as (the reader is referred to Appendix~\ref{append:cholesky_factor_bound} for details)
\begin{align}
    \Vert \hC \Vert_F^2 \leq & \, k \left(1  + 2^{3/2} k (3k+1)u \kappa(\widetilde{B})   \right)^2    \left(1 + \gamma_n^{(p)} + 3 \gamma_n +   3 \gamma_n \gamma_n^{(p)} \right) \Vert A \Vert_F \Vert \Omega \Vert_F^2 \nonumber \\
     & \, +k \left(1  + 2^{3/2} k (k+1)u \kappa(\widetilde{B})   \right)^2  \left(1+  3 \gamma_n \right) \hnu \Vert \Omega \Vert_F^2. \label{eq:C_hat_frob_bound}
\end{align}
The algorithm continues with step~\ref{step:nystrom_triangular}, where we solve $n$ triangular systems of the form
\begin{equation*}
    \hC^T f_j^T = \hy_j^T,
\end{equation*}
where $f_j$ and $\hy_j$ are $j$th rows of $F$ and $\hY_{\nu}$, respectively. The computed $\hf_j$ then satisfies \cite[Theorem 8.5]{HighamBook}
\begin{gather*}
    ( \hC^T + \Delta \hC^T_j ) \hf_j^T = \hy_j^T, \textrm{ where} \\
    \vert \Delta \hC_j \vert \leq \gamma_k \vert \hC \vert. 
\end{gather*}
In the following part of the proof, where we reconstruct $\hA_N$, we make use of the error term
\begin{equation*}
     \Delta_{FC}\coloneqq \hF \hC - \hY_{\nu}.
\end{equation*}
We bound its Frobenius norm in terms of $\hY_{\nu}$. Note that
\begin{equation*}\label{eq:triangsolv_bckw_err}
   \Vert \Delta_{FC} \Vert_F =  \Vert \hF \hC - \hY_{\nu} \Vert_F = \left\lVert \begin{bmatrix}
       \hf_1 \Delta \hC_1 \\ \dots \\ \hf_n \Delta \hC_n 
   \end{bmatrix}\right\rVert_F = \left( \sum_{j=1}^n \Vert \hf_j \Delta \hC_j \Vert_F^2 \right)^{1/2} \leq   \gamma_k \Vert \hF \Vert_F \Vert \hC \Vert_F.
\end{equation*}
Using \cite[Eq. 8.2]{HighamBook}, we can write
\begin{gather*}
        \hF = F + \Delta_F, \textrm{ where} \label{eq:Fhat}\\
        F = \hY_{\nu} \hC^{-1}, \label{eq:F=YnuCinv} \\
    \Vert \Delta_F \Vert_F \leq  \frac{\gamma_k k^{1/2} \kappa_2(\hC)}{1 - \gamma_k k^{1/2} \kappa_2(\hC)} \Vert F \Vert_F, \label{eq:delta_F}
\end{gather*}
and this gives
\begin{equation}\label{eq:Delta_FC_bound}
    \Vert \Delta_{FC} \Vert_F \leq \gamma_k \kappa_F(\hC ) \left( 1 + \frac{\gamma_k k^{1/2} \kappa_2(\hC)}{1 - \gamma_k k^{1/2} \kappa_2(\hC)} \right) \Vert \hY_{\nu} \Vert_F.
\end{equation}
We continue by considering the penultimate step of the algorithm. The SVD is assumed to be computed exactly, thus
\begin{equation}\label{eq:F=svd}
    \hF = U \Sigma V^T,
\end{equation}
where $\Sigma = \textrm{diag}(\sigma_1, \sigma_2, \dots, \sigma_k)$. 
In the final step of the algorithm we remove the shift $\hnu$ and set all the computed eigenvalues that are smaller than the shift to zero. Then the computed approximation is 
\begin{equation*}
    \hA_N = U \Theta U^T =  U \left( \Sigma^2 - \hnu I \right) U^T  - \Delta_r = U \Sigma^2 U^T - \hnu U U^T - \Delta_r,
\end{equation*}
where $\Delta_r$ accounts for the eigenvalues set to zero. If we have $\sigma_1^2 \geq \dots \geq \sigma_{k-j}^2 \geq \hnu$ and $ \hnu > \sigma_{k-j+1}^2 \geq \dots \geq \sigma_k^2$, that is, we set the smallest $j$ eigenvalues to zero, then 
\begin{gather}
    \Delta_r = U \textrm{diag}(\underbrace{0, \dots, 0}_\text{k-j}, \sigma_{k-j+1}^2 - \hnu, \dots, \sigma_k^2 - \hnu )  U^T \textrm{ and } \nonumber \\
    \Vert \Delta_r \Vert_F \leq j^{1/2} \hnu \label{eq:delta_r_bound},
\end{gather}
where we use the fact that $0 \leq \sigma_{k-j+i}^2 < \hnu$ for $i=1,\dots,j$. 

\emph{Part 2. The computed approximation.}
We can now backtrack all the computations and obtain
\begin{equation}
    \hA_N =  \left( (A+ \hnu I) \Omega + \Delta_1 \right) \left( \Omega^T (A+ \hnu I)\Omega  + \Delta_2 \right)^{-1} \left( (A+ \hnu I) \Omega + \Delta_1 \right)^T - \hnu U U^T - \Delta_r, \label{eq:backtrack_AhatN} 
\end{equation}
where
\begin{gather*}
   \Delta_1 =  \Delta  + \Delta_{\nu}  + \Delta_{FC}, \quad\text{and}\\
   \Delta_2 = \frac{1}{2} \left(  \Omega^T (  \Delta  + \Delta_{\nu}) + 
           ( \Delta + \Delta_{\nu})^T \Omega + \Delta_B + \Delta_B^T \right) + \Delta_{s} + \Delta_{Ch}.
\end{gather*}
The step-by-step derivation is supplied in Appendix~\ref{append:backtracking_ANhat}. A crucial step in our analysis is proceeding with a first order approximation of the inverse. This can be done when the assumptions \eqref{eq:assumption_shifted_spd} and \eqref{eq:assumption_1st_order_is_OK} hold.
The approximation is
\begin{equation*}
\resizebox{.99\hsize}{!}{$\left( \Omega^T (A+ \hnu I )\Omega  + \Delta_2 \right)^{-1} \approx \left( \Omega^T (A+ \hnu I )\Omega \right)^{-1}  - \left( \Omega^T (A+ \hnu I )\Omega \right)^{-1} \Delta_2 \left( \Omega^T (A+ \hnu I )\Omega \right)^{-1}.$}
\end{equation*}
Combining this with \eqref{eq:backtrack_AhatN} gives
\begin{align*}
\hA_N \approx  & \, (A+ \hnu I ) \Omega \left( \Omega^T (A+ \hnu I )\Omega \right)^{-1} \Omega^T (A+ \hnu I )^T  +  (A+ \hnu I ) \Omega \left( \Omega^T (A+ \hnu I )\Omega \right)^{-1} \Delta_1^T \\
						& \, +  \Delta_1 \left( \Omega^T (A+ \hnu I )\Omega \right)^{-1} \Omega^T (A+ \hnu I )   +  \Delta_1 \left( \Omega^T (A+ \hnu I )\Omega \right)^{-1}  \Delta_1 \\
						& \, - (A+ \hnu I) \Omega \left( \Omega^T (A+ \hnu I )\Omega \right)^{-1} \Delta_2 \left( \Omega^T (A+ \hnu I )\Omega \right)^{-1} \Omega^T (A+ \hnu I)^T \\
						& \, -  (A+ \hnu I ) \Omega \left( \Omega^T (A+ \hnu I ) \Omega \right)^{-1} \Delta_2 \left( \Omega^T (A+ \hnu I )\Omega \right)^{-1}\Delta_1^T \\
						& \, - \Delta_1 \left( \Omega^T (A+ \hnu I )\Omega \right)^{-1} \Delta_2 \left( \Omega^T (A+ \hnu I )\Omega \right)^{-1} \Omega^T (A+ \hnu I)  \\
						& \, -  \Delta_1 \left( \Omega^T (A+ \hnu I)\Omega \right)^{-1} \Delta_2 \left( \Omega^T (A+ \hnu I)\Omega \right)^{-1}  \Delta_1 \\
						& \, - \hnu U U^T - \Delta_r.
\end{align*}
We notice that the approximation involves terms with the weighted pseudoinverse $\Omega^{\dagger}_{A+ \hnu I}$ and the inverse of $\Omega^T (A+ \hnu I)\Omega$, and a Nystr\"{o}m approximation of the shifted matrix $A+ \hnu I$, that is,
\begin{equation*}
(A+ \hnu I)_N \coloneqq (A+ \hnu I ) \Omega \left( \Omega^T (A+ \hnu I )\Omega \right)^{-1} \Omega^T (A+ \hnu I )^T.
\end{equation*}
Moving $(A+ \hnu I)_N$ to the left hand side,  taking norms and using  \eqref{eq:assumption_1st_order_is_OK} gives
\begin{align}
\Vert \hA_N - (A+ \hnu I)_N \Vert_F
        \leq  & \,  4 \Vert  (A+ \hnu I ) \Omega \left( \Omega^T (A+ \hnu I )\Omega \right)^{-1}\Vert_F  \Vert  \Delta_1 \Vert_F  \nonumber \\
						 & \, + 2 \Vert  \left( \Omega^T (A+ \hnu I )\Omega \right)^{-1}  \Vert_F \Vert \Delta_1 \Vert^2_F  \nonumber \\
						 & \, + \Vert  (A+ \hnu I) \Omega \left( \Omega^T (A+ \hnu I )\Omega \right)^{-1} \Vert^2_F \Vert\Delta_2\Vert_F  \nonumber \\
						  & \, + \Vert  \hnu U U^T \Vert_F  + \Vert  \Delta_r \Vert_F.
         \label{eq:computedNystr_exactShiftedN}
\end{align}
We now focus on $\Delta_2$ and $\Delta_1$. We use \eqref{eq:delta}, \eqref{eq:delta_nu}, \eqref{eq:delta_B}, \eqref{eq:delta_ch}, and \eqref{eq:C_hat_frob_bound} to obtain
\begin{align}
    \Vert \Delta_2 \Vert_F = &\,  \Vert \frac{1}{2} \left(  \Omega^T (  \Delta  + \Delta_{\nu}) + 
           ( \Delta + \Delta_{\nu})^T \Omega + \Delta_B + \Delta_B^T \right) + \Delta_{s} + \Delta_{Ch}  \Vert_F \nonumber \\
           \leq &\, \Vert \Omega \Vert_F \left(  \Vert \Delta \Vert_F  + \Vert \Delta_{\nu} \Vert_F \right) + 
             \Vert \Delta_B \Vert_F + \Vert \Delta_s  \Vert_F + \Vert \Delta_{Ch}  \Vert_F \nonumber \\
             \lesssim &\, \left( \gamma_n^{(p)} + 3 \gamma_n + k \gamma_{k+1} + 3 \gamma_n \gamma_n^{(p)} + k \gamma_{k+1} \gamma_n^{(p)}  \right) \Vert A \Vert_F \Vert \Omega \Vert_F^2 \nonumber \\
             &\, + \left( 3 \gamma_n + k \gamma_{k+1} \right) \hnu \Vert \Omega \Vert_F^2, \label{eq:delta2_bound}
\end{align}
where we have ignored the $\mathcal{O}(u^2)$ terms (which include all the terms with $\kappa(\widetilde{B})$). Considering $\Delta_1$, using \eqref{eq:delta}, \eqref{eq:delta_nu}, and \eqref{eq:Delta_FC_bound} for the first inequality, and \eqref{eq:Yhat_nu_bound} and ignoring $\gamma_n \gamma_k$ terms for the second inequality, we have
\begin{align}
   \Vert \Delta_1 \Vert_F = & \, \Vert \Delta  + \Delta_{\nu}  + \Delta_{FC}  \Vert_F \nonumber \\
   \leq & \, \left( \gamma_n + \gamma_n^{(p)} + \gamma_n \gamma_n^{(p)} \right) \Vert A \Vert_F \Vert \Omega \Vert_F \nonumber \\
   & \, +  \gamma_n \hnu \Vert \Omega \Vert_F + \gamma_k \kappa_F(\hC) \left( 1 + \frac{\gamma_k k^{1/2} \kappa_2(\hC)}{1 - \gamma_k k^{1/2} \kappa_2(\hC)} \right) \Vert \hY_{\nu} \Vert_F  \nonumber \\
    \lesssim  & \, \left( \gamma_n + \gamma_n^{(p)} + \gamma_n \gamma_n^{(p)} +  \left( 1  + \gamma_n^{(p)}\right) \gamma_k  \kappa_F(\hC)  \right) \Vert A \Vert_F \Vert \Omega \Vert_F \nonumber \\
   & \, + \left( 1 +  \gamma_n^{(p)}\right)\gamma_k  \kappa_F(\hC) \frac{\gamma_k k^{1/2} \kappa_2(\hC)}{1 - \gamma_k k^{1/2} \kappa_2(\hC)}  \Vert A \Vert_F \Vert \Omega \Vert_F \nonumber \\
  & \, + \left(  \gamma_n  + \gamma_k \kappa_F(\hC)  \left( 1 + \frac{\gamma_k k^{1/2} \kappa_2(\hC)}{1 - \gamma_k k^{1/2} \kappa_2(\hC)} \right)  \right)\hnu \Vert \Omega \Vert_F. \label{eq:delta1_bound}
\end{align}
As we detail in Appendix~\ref{append:cholesky_factor_condition_number}, $\kappa_2(\hC)$ can be expressed via quantities depending on $A$ and $\Omega$ as
\begin{align}
    \kappa_2(\hC) \leq & \ \left( \frac{1+ \epsilon}{1 - \epsilon \kappa_2(\Omega^T \left(A +\hnu I \right) \Omega)}\right)^{1/2} \kappa_2( A_k + \hnu I_k)^{1/2}  \widetilde{\kappa}(\Omega), \label{eq:kappahC_bound}
\end{align}
where $\widetilde{\kappa}(\Omega)$ is defined in \eqref{eq:kappa_tld_omega} and
\begin{equation}\label{eq:epsilon_kappaC}
    \epsilon = \frac{\Vert  \Delta_2 \Vert_2}{\Vert \Omega^T \left(A + \hnu I \right) \Omega \Vert_2}. 
\end{equation}

Combining \eqref{eq:weighted_pseudinv_bound},  \eqref{eq:bound_inverse_XTAX}, \eqref{eq:delta_r_bound}, \eqref{eq:computedNystr_exactShiftedN}, \eqref{eq:delta2_bound}, and \eqref{eq:delta1_bound}
gives
\begin{align}
\quad &
\begin{multlined}[t]
    \Vert \hA_N - (A+ \hnu I)_N \Vert_F \leq \\  4 k^{1/2} \left( \gamma_n + \gamma_n^{(p)} + \gamma_n \gamma_n^{(p)} +  \left( 1  + \gamma_n^{(p)}\right) \gamma_k  \kappa_F(\hC) \left( 1 +  \frac{\gamma_k k^{1/2} \kappa_2(\hC)}{1 - \gamma_k k^{1/2} \kappa_2(\hC)} \right) \right) \\
     \times \kappa_2(A_k + \hnu I)^{1/2} \widetilde{\kappa}( \Omega) \Vert A \Vert_F  \end{multlined} \nonumber  \\
						 & + k^{1/2} \left( \gamma_n^{(p)} + 3 \gamma_n + k \gamma_{k+1} + 3 \gamma_n \gamma_n^{(p)} + k \gamma_{k+1} \gamma_n^{(p)}  \right) \kappa_2(A_k + \hnu I) \widetilde{\kappa}( \Omega)^2 \Vert A \Vert_F  \nonumber \\
 &  + \begin{multlined}[t]  
 2 k^{1/2}\left( \gamma_n + \gamma_n^{(p)} + \gamma_n \gamma_n^{(p)} +  \left( 1  + \gamma_n^{(p)}\right) \gamma_k  \kappa_F(\hC) \left( 1 +  \frac{\gamma_k k^{1/2} \kappa_2(\hC)}{1 - \gamma_k k^{1/2} \kappa_2(\hC)} \right) \right)^2 \\
  \times \Vert (\Lambda_k + \hnu I_k)^{-1} \Vert_2 \widetilde{\kappa}( \Omega)^2 \Vert A \Vert_F^2   \end{multlined} \nonumber \\ 
       &  + k^{1/2}\left(  \gamma_n  + \gamma_k \kappa_F(\hC)  \left( 1 + \frac{\gamma_k k^{1/2} \kappa_2(\hC)}{1 - \gamma_k k^{1/2} \kappa_2(\hC)} \right)  \right)\hnu \kappa_2(A_k + \hnu I)^{1/2} \widetilde{\kappa}( \Omega) \nonumber \\
      &   + 2 k^{1/2} \left(  \gamma_n  + \gamma_k \kappa_F(\hC)  \left( 1 + \frac{\gamma_k k^{1/2} \kappa_2(\hC)}{1 - \gamma_k k^{1/2} \kappa_2(\hC)} \right)  \right)^2 \hnu^2 \Vert (\Lambda_k + \hnu I_k)^{-1} \Vert_2 \widetilde{\kappa}( \Omega)^2 \nonumber \\
       &     +  k^{1/2} \left( 3 \gamma_n + k \gamma_{k+1} \right) \hnu  \kappa_2(A_k + \hnu I) \widetilde{\kappa}( \Omega)^2  \nonumber \\
		& + 2 k^{1/2} \hnu. \label{eq:finite_prec_error_full_bound} 
\end{align}
Using assumption~\ref{assumpt_simplify_u_up_kappaAk}, we can simplify \eqref{eq:finite_prec_error_full_bound} as detailed in Appendix~\ref{append:simplifying_bound}. We then have 
\begin{equation}\label{eq:finite_precision_bound_simplified}
    \Vert \hA_N - (A+ \hnu I)_N \Vert_F \lesssim  k^{1/2} \widetilde{\gamma}_n^{(p)} \kappa_2( A_k + \hnu I_k)  \widetilde{\kappa}(\Omega)^2 \Vert A \Vert_F.
\end{equation}
Notice that \eqref{eq:finite_precision_bound_simplified} gives us a bound on the error of the computed Nystr\"{o}m approximation of $A$ and the exact Nystr\"{o}m approximation of the shifted matrix. In this case, our ultimate goal to bound the total approximation error can be achieved via
\begin{align*}
\Vert A - \hA_N \Vert = & \, \Vert A - \hnu I + \hnu I- (A+ \hnu I)_N + (A+ \hnu I)_N - \hA_N \Vert \\
\leq & \,  \Vert (A + \hnu I) - (A+ \hnu I)_N \Vert + \Vert \hA_N - (A+ \hnu I)_N  \Vert + \hnu \Vert I \Vert.
\end{align*}
Thus, using this with \eqref{eq:finite_precision_bound_simplified} and assumption~\ref{assumption_shift_bound} we achieve the required result.

\end{proof}

\subsection{A practical heuristic}\label{sec:heuristic}
We note that the bound in Theorem~\ref{th:finite_precision_error} overestimates the total error. However its structure, that is, it being a sum of the exact approximation error and an additional term for the finite precision error, gives us insight into when setting $u_p$ to lower than the working precision may be appropriate. The finite precision error may be essentially ignored if $k^{1/2} \widetilde{\gamma}_n^{(p)} \kappa_2( A_k + \hnu I_k)  \widetilde{\kappa}(\Omega)^2 \Vert A \Vert_F \lesssim \Vert (A + \hnu I) - (A+ \hnu I)_N \Vert_F$. The exact approximation error $\Vert (A + \hnu I) - (A+ \hnu I)_N \Vert_F$ is expected to decrease when the rank $k$ of the approximation is increased, and hence the effect of the finite precision error may be important for large-rank approximations.

We can roughly estimate for which values of $u_p$ $\Vert \E \Vert_F$ stays smaller than $\Vert E \Vert_F$. 
We assume that the shift $\hnu$ is small enough to be ignored and replace $\Vert E \Vert_F$ and $\Vert \E \Vert_F$ with the bounds \eqref{eq:exact_error_determ_frob} and \eqref{eq:finite_precision_bound_simplified}. We further simplify \eqref{eq:exact_error_determ_frob} using $\Vert \Lambda^{1/2}_2 (W_2^T \Omega) (W_1^T \Omega)^{\dagger} \Vert_2 \leq \Vert \Lambda^{1/2}_2 \Vert_2  \Vert (W_2^T \Omega) \Vert_2  \Vert (W_1^T \Omega)^{\dagger} \Vert_2 $ and thus bound
\begin{equation*}
    \Vert E \Vert_F \leq c k^{1/2} \Vert \Lambda^{1/2}_2 \Vert_F^2 \Vert \Omega \Vert_F^2 \Vert (W_1^T \Omega)^{\dagger} \Vert_2^2,
\end{equation*}
where $c$ is a small constant. Requiring \eqref{eq:finite_precision_bound_simplified} to be smaller than this, using the rule of thumb that $n$ can be replaced by $n^{1/2}$ \cite{HighamBook}, and ignoring $\hnu$ gives the constraint
\begin{equation}\label{eq:frob_heuristic_full}
    u_p \ll n^{-1/2} \frac{\lambda_k}{\lambmax} \frac{\sum_{i=k+1}^n \lambda_i}{(\sum_{j=1}^n \lambda_j^2)^{1/2}}.
\end{equation}
Note that we used large overestimates of both $\Vert \E \Vert_F$ and $\Vert E \Vert_F$ to obtain \eqref{eq:frob_heuristic_full}, so this should not be interpreted as a mathematically rigorous condition. Also, while it may be possible to compute or estimate $\lambda_k$ and $\lambmax$, estimating the fraction $\sum_{i=k+1}^n \lambda_i/ (\sum_{j=1}^n \lambda_j^2)^{1/2}$ may be not achievable in practice. In this case, one may choose to omit this term and use the heuristic 
\begin{equation}
    \label{eq:frob_heuristic_no-eig-sum}
    u_p \ll n^{-1/2} \frac{\lambda_k}{\lambmax}.
\end{equation}
In Section \ref{sec:numerics} we demonstrate that, although not a rigorous constraint, this heuristic often gives a good indication of values of $u_p$ that can be chosen without significant affecting approximation quality. 

\subsection{Satisfying the assumptions}\label{sec:satisfying_assumptions}
The validity of our analysis depends on the conditions \eqref{eq:assumption_shifted_spd} and \eqref{eq:assumption_1st_order_is_OK}.
We tackle \eqref{eq:assumption_shifted_spd} by considering the smallest eigenvalue of $\Omega^T \left( A + \hnu I \right) \Omega$ and showing that it is larger than zero, and thus the matrix in question is SPD. 
We use Weyl's inequality and obtain
\begin{equation*}
    \lambda_{min}\left(\Omega^T \left( A + \hnu I \right) \Omega \right) \geq \lambda_{min} (\Omega^T A \Omega ) + \lambda_{min} ( \hnu \Omega^T \Omega  ).
\end{equation*}
Combining this with \eqref{eq:bound_inverse_XTAX} and the fact that $\lambda_i (\Omega^T \Omega  ) = \sigma_i (\Omega )^2$ gives
\begin{equation}\label{eq:mineig_QTAnuQ2}
    \lambda_{min}\left(\Omega^T \left( A + \hnu I \right) \Omega \right) \geq \sigma_{min} ( \Omega^T W_1 )^2 \lambda_k(A) + \hnu \sigma_{min} (\Omega)^2.
\end{equation}
All the quantities on the right hand side of the bound are non-negative and even in the case when $\lambda_k(A) = 0$ the distance between the smallest eigenvalue of $\Omega^T \left( A + \hnu I \right) \Omega$ and zero depends on the shift $\hnu$ and the smallest singular value of the sketching matrix $\Omega$.

We now consider condition \eqref{eq:assumption_1st_order_is_OK}. This is done by showing that
\begin{align}
    \left \lVert \left[ \Omega^T \left( A + \hnu I \right) \Omega \right]^{-1} \Delta_2 \right \rVert_2 \leq  & \, \left \lVert \left[ \Omega^T \left( A + \hnu I \right) \Omega \right]^{-1} \right \rVert_2  \left \lVert \Delta_2 \right \rVert_2 \nonumber \\
    \leq & \, \left \lVert \left[ \Omega^T \left( A + \hnu I \right) \Omega \right]^{-1} \right \rVert_2  \left \lVert \Delta_2 \right \rVert_F < 1. \nonumber
\end{align}
Using $\left\lVert \left( \Omega^T \left( A + \hnu I \right) \Omega \right)^{-1} \right\rVert_2  = \frac{1}{\lambda_{min}\left( \Omega^T \left( A + \hnu I \right) \Omega \right)}$, \eqref{eq:mineig_QTAnuQ2}, and \eqref{eq:delta2_bound} we have
\begin{multline}\label{eq:bound_shifted_inverse_delta2} 
    \left\lVert \left( \Omega^T \left( A + \hnu I \right) \Omega \right)^{-1} \right\rVert_2  \Vert \Delta_2 \Vert_F \leq  \\
    \frac{\left( \gamma_n^{(p)} + 3 \gamma_n +  k \gamma_{k+1} + 3 \gamma_n \gamma_n^{(p)} +  k \gamma_{k+1} \gamma_n^{(p)}  \right) \Vert A \Vert_F \Vert \Omega \Vert_F^2 }{ \lambda_k(A) \sigma_{min} ( W_1^T \Omega )^2 + \hnu \sigma_{min} (\Omega )^2} \\
    \, + \frac{ \left( 3 \gamma_n + k \gamma_{k+1} \right) \hnu \Vert \Omega \Vert_F^2 }{ \lambda_k(A) \sigma_{min} ( W_1^T \Omega)^2 + \hnu \sigma_{min} (\Omega )^2}.   
\end{multline}   
Thus to satisfy \eqref{eq:assumption_1st_order_is_OK} we require that
\begin{multline*}
    \left( \gamma_n^{(p)} + 3 \gamma_n +  k \gamma_{k+1} + 3 \gamma_n \gamma_n^{(p)} +  k \gamma_{k+1} \gamma_n^{(p)}  \right) \Vert A \Vert_F \Vert \Omega \Vert_F^2 + \left( 3 \gamma_n + k \gamma_{k+1} \right) \hnu \Vert \Omega \Vert_F^2 \\
    < \lambda_k(A) \sigma_{min} ( W_1^T \Omega )^2 + \hnu \sigma_{min} (\Omega )^2.
\end{multline*}
Moving the shift to the right-hand side gives
\begin{multline*}
    \left( \gamma_n^{(p)} + 3 \gamma_n + k \gamma_{k+1} + 3 \gamma_n \gamma_n^{(p)} +  k \gamma_{k+1} \gamma_n^{(p)}  \right) \Vert A \Vert_F \Vert \Omega \Vert_F^2 \\
    < \lambda_k(A) \sigma_{min} ( W_1^T \Omega)^2 + \hnu \left( \sigma_{min} (\Omega )^2 - \left( 3 \gamma_n + k \gamma_{k+1} \right)  \Vert \Omega \Vert_F^2  \right).
\end{multline*}
Ideally, this condition should be satisfied independently of $\lambda_k$ and $W_1$, and hence we require
\begin{multline}\label{eq:condition_shift_sketch}
    \left( \gamma_n^{(p)} + 3 \gamma_n +  k \gamma_{k+1} + 3 \gamma_n \gamma_n^{(p)} +  k \gamma_{k+1} \gamma_n^{(p)}  \right) \Vert A \Vert_F \Vert \Omega \Vert_F^2  \\
    < \hnu \left( \sigma_{min} (\Omega )^2 - \left( 3 \gamma_n + k \gamma_{k+1} \right)  \Vert \Omega \Vert_F^2  \right).
\end{multline}
We thus need to choose the sketching matrix $\Omega$ and the shift $\hnu$ accordingly. We address these points in Section~\ref{sec:choosing_shift_sketching}.

\subsubsection{Success of the Cholesky decomposition}\label{sec:successful_cholesky}
A potential breaking point of Algorithm~\ref{alg:nystrom_reg_id} is the Cholesky decomposition in step~\ref{step:nystrom_chol}. We assume that $\kappa(\widetilde{B})\ll u^{-1}$ and thus the Cholesky factorization runs successfully if $\widetilde{B}$ is SPD. The shift $\nu$ was introduced in \cite{LiLinderman2017} to ensure this. We further show that $\widetilde{B}$ is SPD and hence the success of the Cholesky decomposition is guaranteed if conditions \eqref{eq:assumption_shifted_spd} and \eqref{eq:assumption_1st_order_is_OK} are satisfied. In step~\ref{step:nystrom_chol}, we compute the Cholesky decomposition of
\begin{align*}
    \widetilde{B} = & \, \frac{1}{2}\left(\hB + \hB^T \right) + \Delta_s \\
                =   & \, \Omega^T \left( A + \hnu I \right) \Omega + \frac{1}{2} \left(  \Omega^T (  \Delta  + \Delta_{\nu}) + 
           ( \Delta + \Delta_{\nu})^T \Omega + \Delta_B + \Delta_B^T \right) + \Delta_{s} \\
            =   & \, \Omega^T \left( A + \hnu I \right) \Omega + ( \Delta_2 - \Delta_{Ch}). 
\end{align*}
If the SPD condition \eqref{eq:assumption_shifted_spd} holds, then $\Omega^T \left( A + \hnu I \right) \Omega$ has a Cholesky decomposition. We can thus use \cite[Theorem 1.4]{sun1991perturbation} (as in \eqref{eq:cholesky_condition}) to show that if 
    \begin{equation}\label{eq:cholesky_ex_condition}
        \left\lVert \left( \Omega^T \left( A + \hnu I \right) \Omega \right)^{-1} \right\rVert_2  \Vert \Delta_2 - \Delta_{Ch} \Vert_F < 1, 
    \end{equation}
then $\widetilde{B}$ has a Cholesky decomposition. Notice that
\begin{equation*}
    \left\lVert \left( \Omega^T \left( A + \hnu I \right) \Omega \right)^{-1} \right\rVert_2  \Vert \Delta_2 - \Delta_{Ch} \Vert_F \leq \left\lVert \left( \Omega^T \left( A + \hnu I \right) \Omega \right)^{-1} \right\rVert_2 \left( \Vert \Delta_2 \Vert_F + \Vert \Delta_{Ch} \Vert_F \right) 
\end{equation*}
which can be bounded as in \eqref{eq:bound_shifted_inverse_delta2} and thus \eqref{eq:cholesky_ex_condition} is satisfied if \eqref{eq:assumption_shifted_spd} and \eqref{eq:assumption_1st_order_is_OK} hold.

\subsubsection{Choosing the shift and the sketching matrix}\label{sec:choosing_shift_sketching}
Our analysis indicates that the shift and the sketching matrix have to be chosen such that
\begin{itemize}
    \item \eqref{eq:condition_shift_sketch} is satisfied; 
    \item assumption~\ref{assumption_shift_bound} is satisfied, that is, $\hnu \leq c(n,k) u_p \Vert \Omega \Vert_F^2 \Vert A \Vert_F $, where $c(n,k)$ is a constant that depends on $n$ and $k$, so that the shift does not increase the total error bound. 
\end{itemize} 
From \eqref{eq:condition_shift_sketch}, we also require 
\begin{equation*}
    \sigma_{min} (\Omega )^2 > \left( 3 \gamma_n + k \gamma_{k+1} \right)  \Vert \Omega \Vert_F^2.
\end{equation*}
Rearranging the terms gives
\begin{equation*}
    \left( 3 \gamma_n + k \gamma_{k+1} \right)^{-1} > \frac{\Vert \Omega \Vert_F^2 }{\sigma_{min} (\Omega )^2} = \Vert \Omega \Vert_F^2 \Vert \Omega^{\dagger} \Vert_2^2,
\end{equation*}
which can be simplified to $\kappa(\Omega) \ll u^{-1/2}$. We note that sketching matrices are usually chosen to be well-conditioned and thus this condition is satisfied. We denote
\begin{equation}
    \beta \coloneqq \sigma_{min}(\Omega)^2 - \left( 3 \gamma_n + k \gamma_{k+1} \right)  \Vert \Omega \Vert_F^2
\end{equation}
and write \eqref{eq:condition_shift_sketch} as
\begin{equation}\label{eq:condition_shift}
   \left( \gamma_n^{(p)} + 3 \gamma_n + k \gamma_{k+1} + 3 \gamma_n \gamma_n^{(p)} + k \gamma_{k+1} \gamma_n^{(p)}  \right) \Vert A \Vert_F \Vert \Omega \Vert_F^2  < \hnu \beta.
\end{equation}
Note that the computed versions of the shifts defined in \cite{Tropp2017fixed} and \cite{tropp2023randomized} are
\begin{gather*}
\hnu = 2 u_p \Vert \hY \Vert_F + \delta = 2 u_p \Vert A \Omega + \Delta \Vert_F + \delta_1 , \textrm{ where } \\
   \vert \delta_1 \vert \leq 2 n u_p u \Vert \hY \Vert_F  
\end{gather*}
and
\begin{gather*}
    \hnu = 2 u_p (\textrm{trace}(A) + \delta_2) =  2 u_p  \Vert A^{1/2} \Vert_F^2 + 2u_p \delta_2, \textrm{ where } \\
     \vert \delta_2 \vert \leq \gamma_n^{(p)} \textrm{trace}(A).
\end{gather*}
Unfortunately we cannot show that \eqref{eq:condition_shift} strictly holds with these choices of the shift. They are however sufficient in most cases as the bound for $\Vert \Delta_2 \Vert_F$ that produces the terms on the left hand side of \eqref{eq:condition_shift} is a worst-case bound and is usually a large overestimate. If, however, it is important to ensure that \eqref{eq:condition_shift} holds even in the worst case, we provide the following guidance. The expression on the left-hand side of \eqref{eq:condition_shift} can be simplified using $\gamma_n \leq \gamma_n^{(p)}$, $\gamma_{k+1} \leq \gamma_n^{(p)} $, $3 \gamma_n \ll 1$, and $k \gamma_{k+1} \ll 1$ to
\begin{equation*}
   \left( \gamma_n^{(p)} + 3 \gamma_n + k \gamma_{k+1} + 3 \gamma_n \gamma_n^{(p)} + k \gamma_{k+1} \gamma_n^{(p)}  \right) \Vert A \Vert_F \Vert \Omega \Vert_F^2 \leq  (k+6) \gamma_n^{(p)} \Vert A \Vert_F \Vert \Omega \Vert_F^2 
\end{equation*}
and we can thus consider a slightly stricter condition 
\begin{equation}\label{eq:shift_sketch_singval_cond}
   (k+6) \gamma_n^{(p)} \Vert A \Vert_F \Vert \Omega \Vert_F^2  < \hnu \beta.
\end{equation}
This shows that the choice of $\nu$ depends on $\beta$, which is dominated by $\sigma_{min}(\Omega)$. If $\Omega$ is set to be a matrix with $k$ orthogonal columns, then $\Vert \Omega \Vert_F^2 = k$, $\sigma_{min}(\Omega) = 1$ and thus $\beta = 1 - k( 3 \gamma_n + k \gamma_{k+1})$. It is reasonable to assume that $k( 3 \gamma_n + k \gamma_{k+1}) \ll 0.5$. Then \eqref{eq:shift_sketch_singval_cond} is satisfied with the shift set to 
\begin{equation}\label{eq:shift_proposal}
    \nu = 2 k (k+7) \gamma_n^{(p)} \Vert A \Vert_F.
\end{equation}
We note that $\Vert A \Vert_F$ can be computed simultaneously when $A$ is accessed to compute $A \Omega$. Alternatively, $\Vert A \Vert_F$ can be replaced by $n^{1/2} \Vert A \Vert_{\infty}$ in \eqref{eq:shift_proposal} and $\Vert A \Vert_{\infty}$ can be computed at the same time as $A \Omega$ by multiplying $\begin{pmatrix} 1 & 1 & \dots & 1 \end{pmatrix}$ with each column of $\vert A \vert$.

In the case when $\Omega$ does not have orthogonal columns, we may require computing or estimating $\sigma_{min}(\Omega)$. For example, if we take $\Omega$ to be a random matrix with Gaussian entries, then $\sigma_{min}(\Omega)$ is bounded from below by $n^{1/2} - (k-1)^{1/2}$ with high probability \cite{rudelson2009smallest}. Thus for small $k$ we can expect $\sigma_{min}(\Omega) > 1$ and thus set
\begin{equation*}
    \nu = (k+7) \gamma_n^{(p)} \Vert A \Vert_F \Vert \Omega \Vert_F^2.
\end{equation*}
Notice that this is also the case when we are most interested in computing the Nystr\"{o}m approximation. 
If however $k$ is so large that $n^{1/2} - 1 < (k-1)^{1/2}$ holds then we may have $\sigma_{min}(\Omega)<1$ (see Figure~\ref{fig:smallest_singular_value_sketch} for an illustration) and thus $\sigma_{min}(\Omega)$ has to be incorporated into the shift. This is in line with the comment in \cite{Tropp2017fixed} that for large $k$ using orthogonal rather than Gaussian $\Omega$ improves the numerical stability.
\begin{figure}
    \centering
    \includegraphics[width=0.5\linewidth]{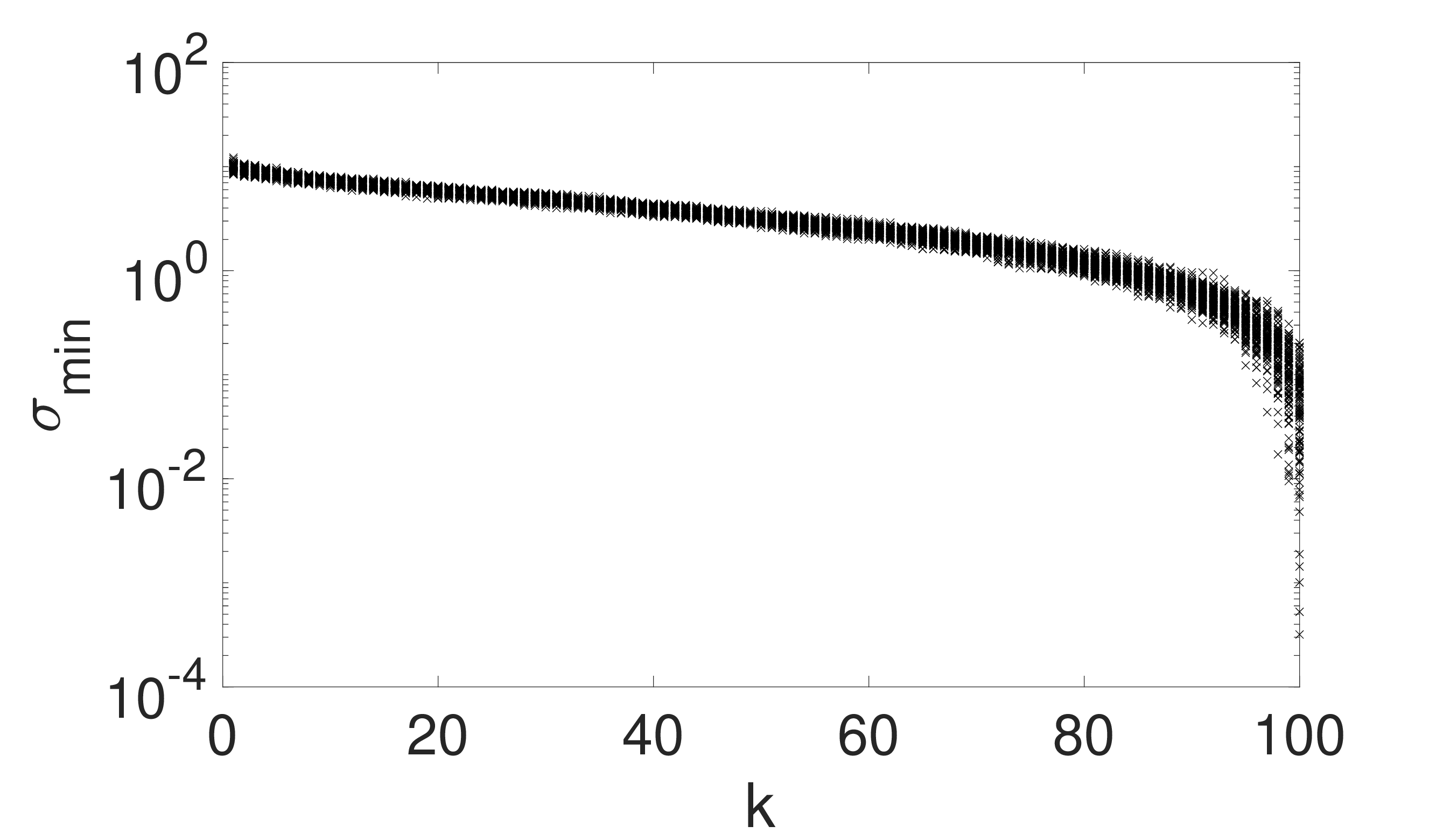}
    \caption{The smallest singular value of 100 realisations of a random $100 \times k$ matrix with standard Gaussian entries.}
    \label{fig:smallest_singular_value_sketch}
\end{figure}

\section{Preconditioning}\label{sec:preconditioning}

We now consider an important application area where low-rank matrix approximations are used, namely, preconditioning iterative solvers. A useful preconditioner has to be inexpensive to construct and apply while accelerating the convergence of an iterative solver. Using the mixed precision Algorithm~\ref{alg:nystrom_reg_id} to construct the preconditioner may thus result in computational savings.

Let $A\in\mathbb{R}^{n\times n}$ be a symmetric positive semidefinite matrix, so that it can be approximated via the  Nystr\"{o}m method, and consider a system of linear equations of the form
\begin{equation}\label{eq:shifted_system}
    (A + \mu I) x = b,
\end{equation}
where $\mu \geq 0$ so that $A + \mu I$ is positive definite, $I \in \mathbb{R}^{n \times n}$ is an identity matrix, and $x,b \in \mathbb{R}^n$. Preconditioned conjugate gradient (PCG) is a popular method for systems with symmetric positive definite coefficient matrices; see, e.g., \cite{Saadbook}.

We focus on the case where $A$ has rapidly decreasing eigenvalues or a cluster of large eigenvalues; notice that the spectrum of $A + \mu I$ has the same structure. In such settings, finite precision PCG convergence can be slow \cite{CarsonStrakos2020} and removing the largest eigenvalues with a preconditioner can accelerate convergence. This can be achieved using a spectral limited memory preconditioner (LMP), defined as 
\begin{gather}\label{eq:prec:LMP_nystrom_shifted_A+muI}
    P = I - U U^T + \frac{1}{\alpha + \mu} U (\Theta +\mu I) U^T, \\
    P^{-1} = I - U U^T + (\alpha + \mu) U (\Theta +\mu I)^{-1} U^T, \label{eq:prec:LMP_nystrom_shifted_A+muI_inverse}
\end{gather}
where the columns of $U \in \mathbb{R}^{n \times k}$ are approximate eigenvectors of $A$ and $U^T U = I$, $\Theta$ is diagonal with approximations to the eigenvalues of $A$, and $\alpha \geq 0$. 

The expression \eqref{eq:prec:LMP_nystrom_shifted_A+muI_inverse} is an instance of a general class of limited memory preconditioners studied in \cite{Tshimanga08,Gratton11,Tshimanga_thesis} and used in data assimilation \cite{Moore11,Nemovar_ecmwf,Laloyaux18}. Algorithm~\ref{alg:nystrom_reg_id} returns an eigendecomposition of $A_N$ and thus can be used to construct \eqref{eq:prec:LMP_nystrom_shifted_A+muI}. The randomised version of the preconditioner is mentioned in \cite{Martinsson20} (Section 17) and is analysed in \cite{Frangella2021}, where it is called a randomised Nystr\"{o}m preconditioner. It is also explored in a data assimilation setting under the name randomised LMP in \cite{Dauzickaite2021forcing}.

It is easy to show that if \eqref{eq:prec:LMP_nystrom_shifted_A+muI_inverse} is constructed using exact eigenpairs, then the eigenvalues used to construct the preconditioner are mapped to $\alpha + \mu$ and the other eigenvalues remain unchanged. This cannot be guaranteed when constructing the preconditioner with approximations to the eigenpairs, but a study of the eigenvalues of the preconditioned matrix when \eqref{eq:prec:LMP_nystrom_shifted_A+muI} is constructed with eigenvalues of $A+tE$, where $\| E \|_2 = 1$, $t \in \mathbb{R}$ is small, and $\mu=0$, in \cite{Giraud06} by Giraud and Gratton show that if the preconditioner is constructed using high quality approximations of not clustered or not small isolated eigenpairs, then the eigenvalues in the inexact case will be close to the exact ones.

If we have information on existing eigenvalue clusters of $A$, we may use it to choose $\alpha$ and thus send the largest eigenvalues close to an already existing cluster. Martinsson and Tropp in \cite{Martinsson20} suggest setting $\alpha = \lambda_k$, where $\lambda_k$ is the smallest nonzero eigenvalue of $A_N$. This is done with the hope that the spectrum of $P^{-1}(A +\mu I)$ is then more clustered compared to $A$, and the condition number is reduced. Note that the condition number alone does not determine the PCG convergence behaviour and the same holds for the number of clusters of eigenvalues; see \cite{CarsonStrakos2020} and \cite{carson202270} for detailed commentary. A small condition number however does indicate fast convergence. In the following section, we consider the resulting condition number when the mixed precision Nystr\"{o}m approximation in Algorithm \ref{alg:nystrom_reg_id} is used to construct the preconditioner.  

\subsection{Bound on the condition number of the preconditioned coefficient matrix}

The work in \cite{Frangella2021} provides bounds for the condition number of the preconditioned coefficient matrix in exact arithmetic. We extend them to include the finite precision error in the preconditioner
\begin{equation}\label{eq:prec:LMP_nystrom_shifted_A+muI_inverse_finite-prec}
    \hPi = I - \hU \hU^T + (\hlambda_k + \mu) \hU (\widehat{\Theta} +\mu I)^{-1} \hU^T,
\end{equation}
where $\hA_N = \hU \widehat{\Theta} \hU^T$ is a finite precision rank-$k$ Nystr\"{o}m approximation of $A$ obtained via Algorithm~\ref{alg:nystrom_reg_id} and $\hlambda_k$ is the smallest eigenvalue of $\hA_N$. The columns of $\hU$ are the eigenvectors of $\hA_N$ and thus they are orthogonal. Then $\hPi$ is symmetric positive definite and denoting the nonzero eigenvalues of $\hA_N$ as $\hlambda_i > \hlambda_{i+1}$ we can write
\begin{align*}
    \hPi = & \, I - \hU \hU^T + (\hlambda_k + \mu) \hU \textrm{diag}\left(
        \frac{1}{\hlambda_1 + \mu }, \frac{1}{\hlambda_2 + \mu }, \dots, \frac{1}{\hlambda_k + \mu }
    \right) \hU^T \\
    = & \, I + \hU \textrm{diag}\left(
        \frac{\hlambda_k + \mu}{\hlambda_1 + \mu } - 1,  \frac{\hlambda_k + \mu}{\hlambda_2 + \mu } -1,\dots, \frac{\hlambda_k + \mu}{\hlambda_k + \mu } -1
    \right) \hU^T. 
\end{align*}
$\hPi$ thus has $n-k+1$ eigenvalues equal to one and the rest are equal to $(\hlambda_k +\mu)/(\hlambda_i +\mu)$ for $i=1,2,\dots,k-1$.

\begin{theorem}\label{prop:prec_condition_no_deterministic}
Let $(A + \mu I) \in \mathbb{R}^{n \times n}$ be symmetric positive definite, $\hPi$ as defined in \eqref{eq:prec:LMP_nystrom_shifted_A+muI_inverse_finite-prec}, and $E$ and $\E$ as defined in \eqref{eq:Nystrom_exact_error} and \eqref{eq:Nystrom_finite_error}, respectively. The condition number $\kappa(\hPisqr (A + \mu I) \hPisqr)$ can then be bounded as
\begin{equation}\label{eq:condNobounds-lower-upper}
       \max \left\{ 1, \frac{\hlambda_k + \mu - \| \E \|_2}{\mu + \lambda_{min} (A) } \right\} \leq \kappa(\hPisqr (A + \mu I) \hPisqr) \leq 1 + \frac{\hlambda_k + \| E \|_2 + 2\| \E \|_2}{\mu - \| \E \|_2 },
\end{equation}
where the upper bound holds if $\mu > \| \E \|_2$. Regardless of this constraint,
\begin{equation}\label{eq:condNobounds-spd}
\kappa(\hPisqr (A + \mu I) \hPisqr) \leq  \left(  \hlambda_k + \mu + \| E \|_2 + \| \E \|_2 \right) \left( \frac{1}{\hlambda_k + \mu} + \frac{\| \E \|_2 +1 }{\lambmin(A) + \mu} \right).  
\end{equation}
\end{theorem}
The proof closely follows the argument in \cite{Frangella2021} and is supplied in Appendix~\ref{append_condition_number_finite_precision}.

If $\| \E \|_2 = 0$, then the bounds in Theorem~\ref{prop:prec_condition_no_deterministic} coincide with the bounds in \cite[Proposition 5.3]{Frangella2021}. The lower bound is useful when $\hlambda_k + \mu - \| \E \|_2 > \mu + \lambda_{min} (A)$, that is, when  $\hlambda_k >  \| \E \|_2 + \lambda_{min} (A)$, which can be expected to hold when large eigenvalues are approximated. The finite precision error has an additive effect on the bounds and expands them, and a multiplicative effect appears in \eqref{eq:condNobounds-spd}.

\section{Numerical examples}\label{sec:numerics}

We illustrate the theory developed in the previous sections with simple numerical experiments in MATLAB R2021a\footnote{Our code can be found at \url{https://github.com/dauzickaite/mpNystrom}}. The Nystr\"{o}m approximation is constructed setting $u_p$ to double, single, and half precision. Half precision is simulated using the \textit{chop} function \cite{HighamChop}. The working precision $u$ is set to double. An `exact' Nystr\"{o}m approximation is computed using the Advanpix Multiprecision Computing Toolbox \cite{advanpix} using 64 decimal digits precision for all computations in Algorithm~\ref{alg:nystrom_reg_id}. The same extended precision is used to compute the total approximation error and $\E$, condition numbers, and all the bounds. Each experiment is performed with ten initializations of the sketching matrix $\Omega$ and we report the means. There is no oversampling. 

The sketching matrix $\Omega$ is obtained by generating an $n \times k$ matrix $G$ with Gaussian entries, computing its economical size QR decomposition $G=QR$ in double precision and setting $\Omega = Q$. We note that just setting $\Omega = G$ sometimes results in overflow when $u_p$ is set to half. 
The shift $\nu$ is set to be $2 u_p \Vert Y \Vert_F$ as proposed in \cite{Tropp2017fixed}.

We explore the approximation problem without preconditioning and consider synthetic and application problems in Section~\ref{section:lowrank_approx}. The approximations of $A$ are used to construct the LMP and the preconditioned shifted systems are solved via PCG in Section~\ref{section:precond_numerics}. Experiments for a kernel ridge regression problem are presented in Section~\ref{section:krr}.

\subsection{Low-rank approximation}\label{section:lowrank_approx}
We compute the Frobenius norms of the finite precision error $\E$ and the total approximation error $A - \hA_N$. 
We notice that the terms $\kappa(A_k +\hnu I) \widetilde{\kappa}(\Omega)^2$ do not contribute meaningful information to the bounds as they come from a loose bound \eqref{eq:weighted_pseudinv_bound}. The finite precision error is hence compared to
\begin{equation}\label{eq:fin_prec_error}
  k^{1/2} \gamma_n^{(p)} \|A\|_F.
\end{equation}

\subsubsection{Synthetic problems}
We perform experiments with synthetic matrices described in \cite{Tropp2017fixed}. $A \in \mathbb{R}^{n \times n}$ is a real matrix with effective rank $r$ and is constructed in the following ways.  
\begin{itemize}
    \item \emph{Exponential decay}: 
    \begin{equation*}
        A = diag(\beta_1,\beta_2,\dots,\beta_r,10^{-q},10^{-2q},\dots,10^{-(n-r)q}),
    \end{equation*} 
    where $q$ is set to values $0.1$, $0.25$, and $1$.
    \item \emph{Polynomial decay}:
    \begin{equation*}
        A = diag(\beta_1,\beta_2,\dots,\beta_r,2^{-p},3^{-p},\dots,(n-r+1)^{-p}),
    \end{equation*}
     where $p$ is set to $ 0.5$, $1,$ and $2$.
    \item \emph{PSD noise}: 
    \begin{equation*}
        A = diag(\beta_1,\beta_2,\dots,\beta_r,0, \dots, 0) + \xi n^{-1}(G G^T), 
    \end{equation*}
    where $G \in \mathbb{R}^{n \times n}$ is a random Gaussian matrix, and $\xi$ is set to $10^{-4}$, $10^{-2}$, and $10^{-1}$ with higher values corresponding to greater noise.
\end{itemize}

We set $n=10^2$, $r=10$, $\beta_1=\beta_2=\dots = \beta_r \eqqcolon \beta$. The values of $\beta$ span $1, 10, 10^2, \dots, 10^{16}$. Note that $\| A \|_F \approx r^{1/2} \beta$. 
Experiments with half precision are performed when $\beta < 10^5$. We compute rank $k \in \{1,2,\dots,10\}$ approximations.

We report the computed total and finite precision error in Figure~\ref{fig:synthetic_error_pol_decay} for the polynomial decay problem with $p=1$; the results are similar for all the synthetic problems. 
When $k<10$, $E$ is dominated by $\beta_{k+1}$, $\beta_{k+2}$, $\dots$, $\beta_{r}$ and the finite precision error $\E$ stays significantly smaller. When $k=10$, all the large eigenvalues are being approximated and the exact approximation error $E$ depends on the small eigenvalues. The finite precision error thus becomes important and affects the approximation quality detrimentally for large $\|A\|_F$.

\begin{figure}
\begin{center}
\begin{subfigure}[b]{0.49\linewidth}
  \centering
 \includegraphics[width=\linewidth]{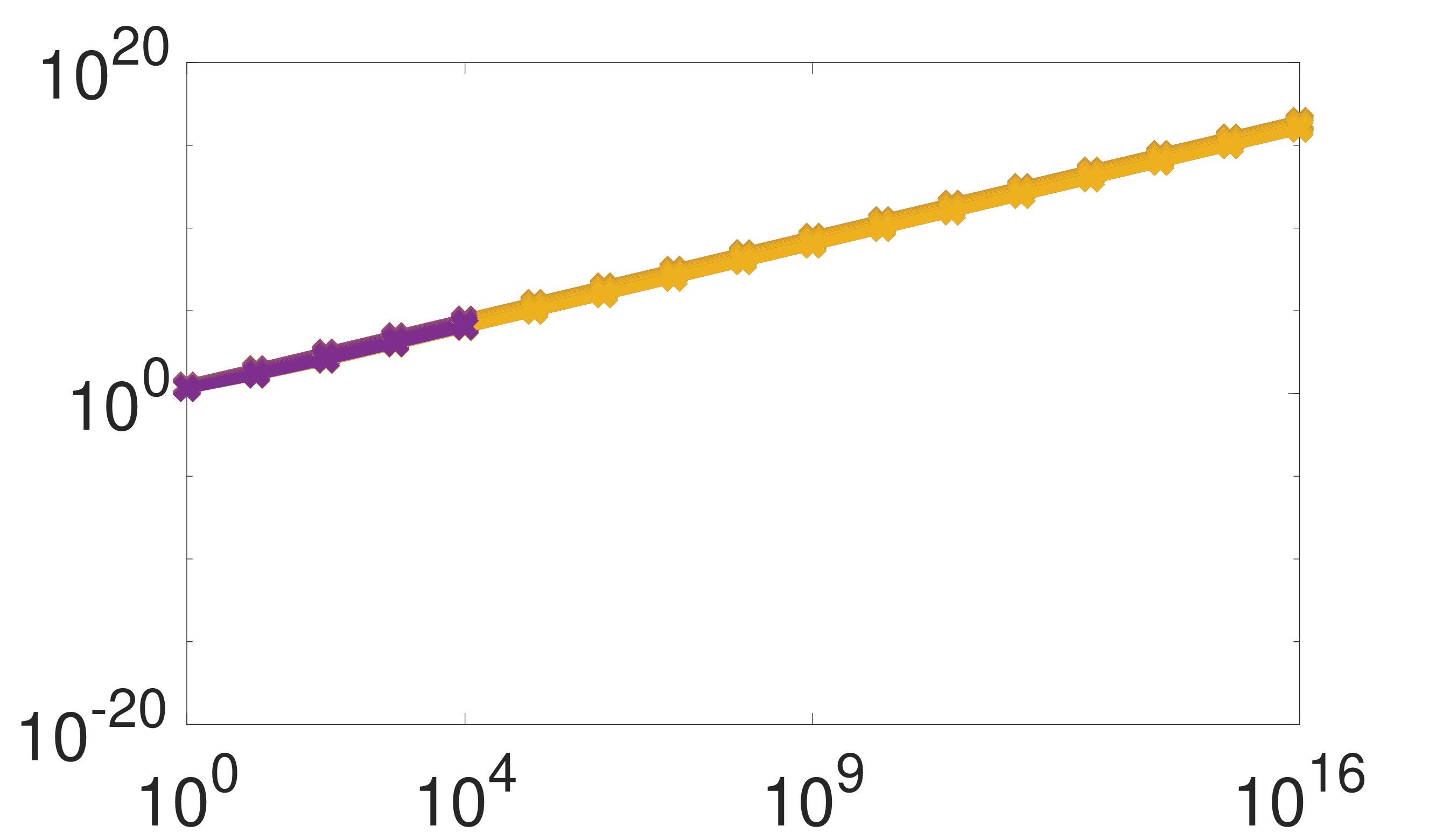}
  \caption{$\| A - \hA_N\|_F$, $k \in \{1,2,\dots,9\}$}
\end{subfigure}
\begin{subfigure}[b]{0.49\linewidth}
  \centering
 \includegraphics[width=\linewidth]{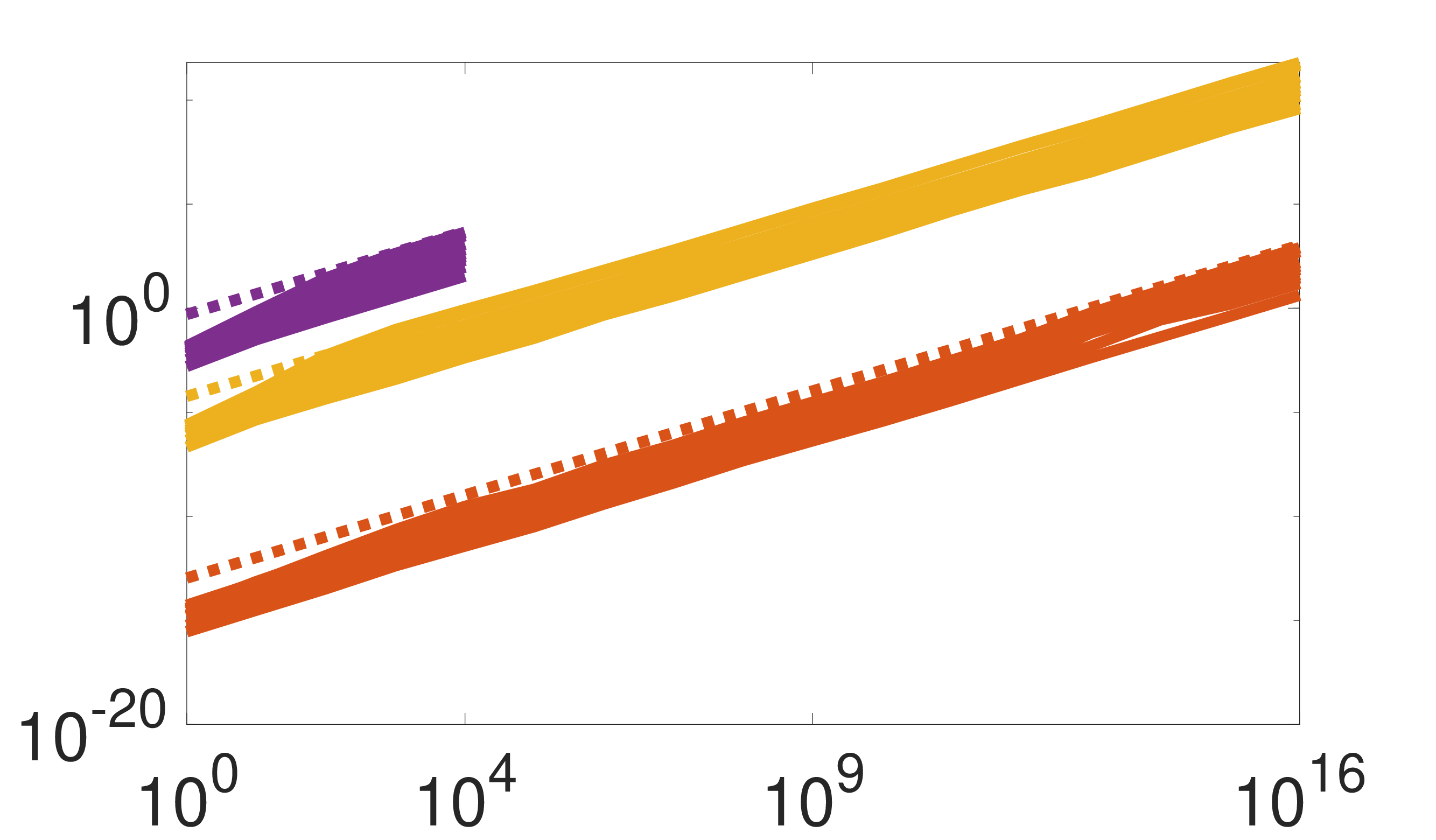}
  \caption{$\| A_N - \hA_N\|_F$, $k \in \{1,2,\dots,9\}$}
\end{subfigure}
\begin{subfigure}[b]{0.49\linewidth}
  \centering
 \includegraphics[width=\linewidth]{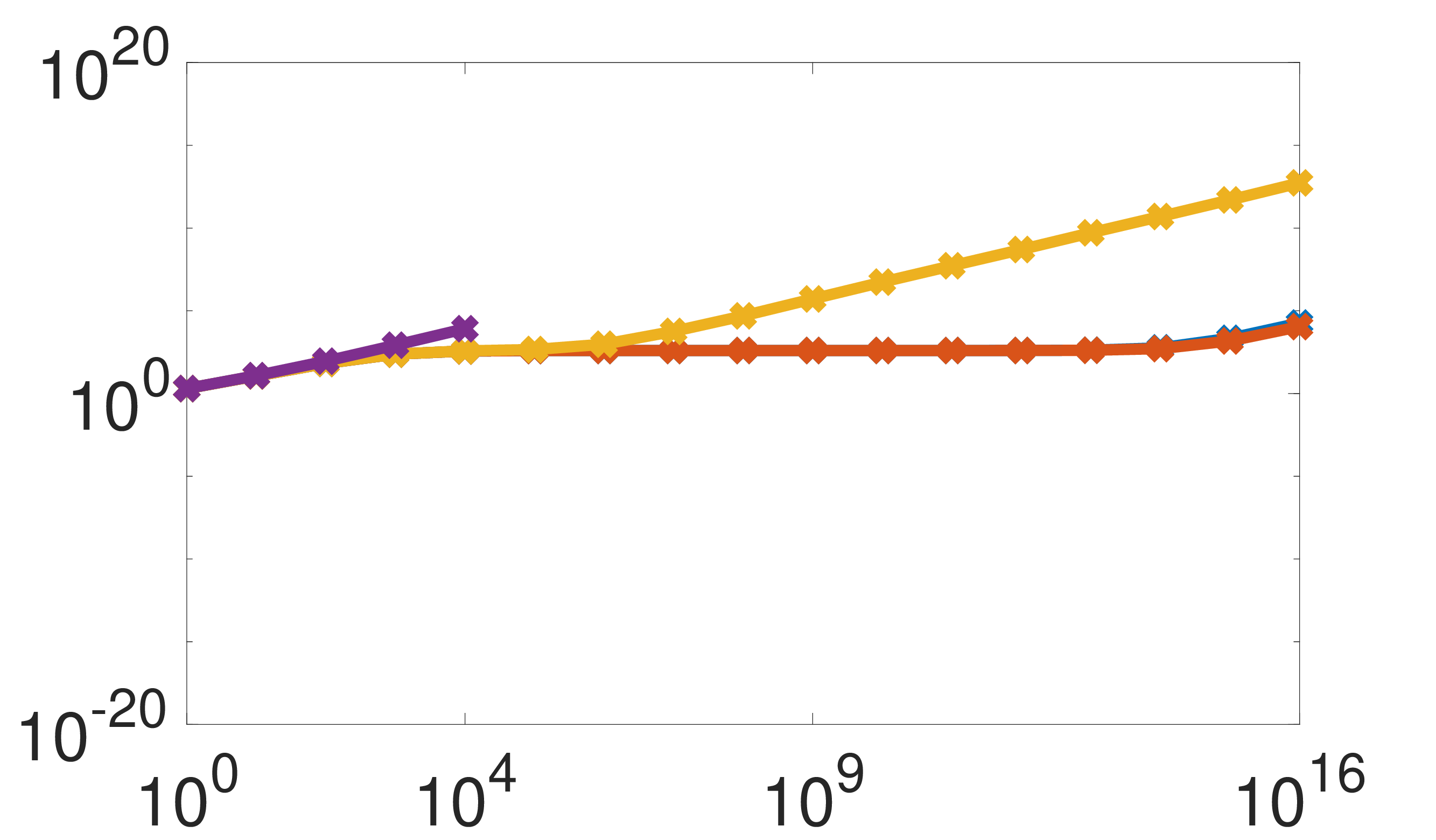}
  \caption{$\| A - \hA_N\|_F$, $k =10$}
\end{subfigure}
\begin{subfigure}[b]{0.49\linewidth}
  \centering
 \includegraphics[width=\linewidth]{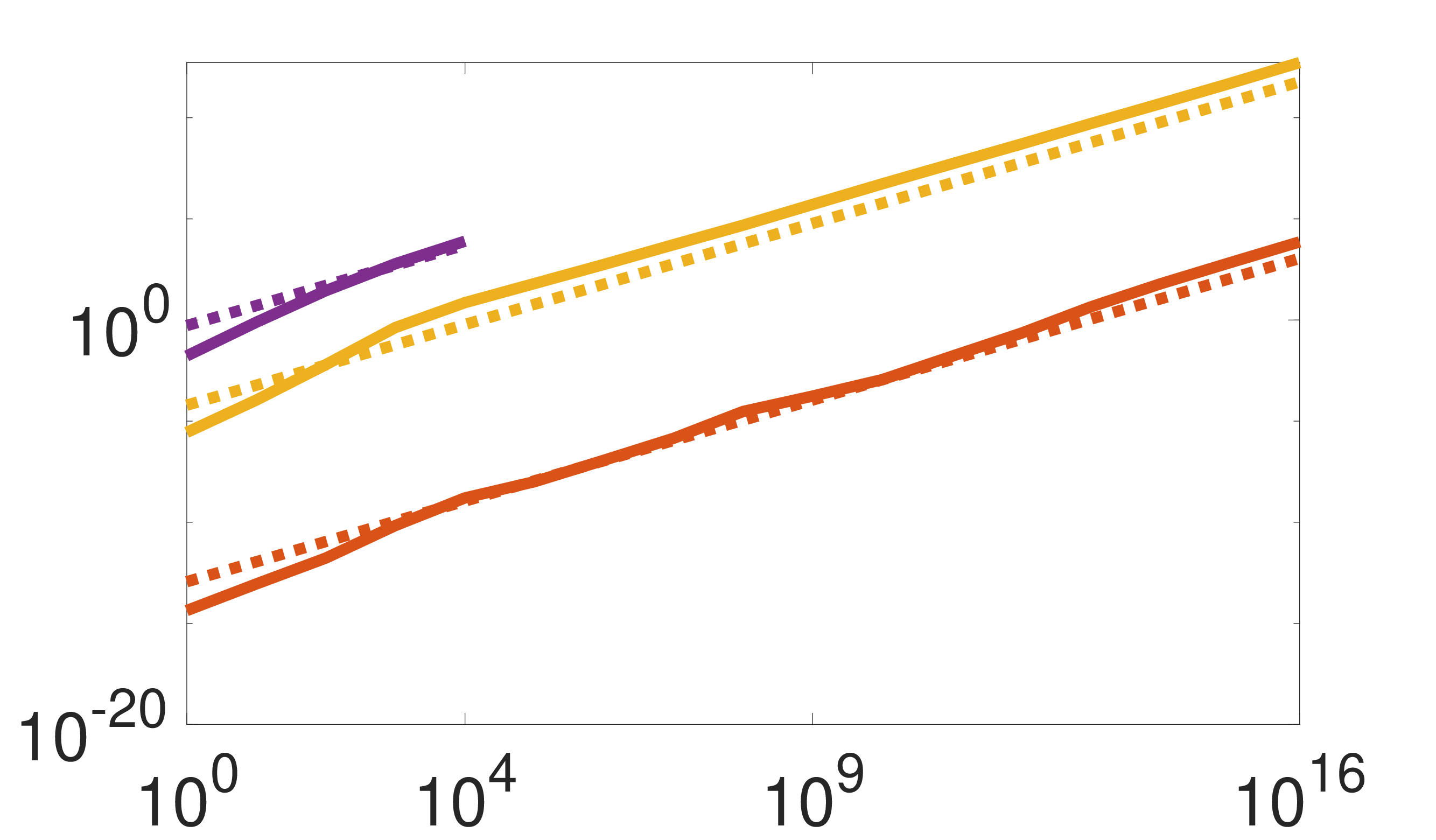}
 
  \caption{$\| A_N - \hA_N\|_F$, $k =10$}
\end{subfigure}
\end{center}
    \caption{Polynomial decay problem with $p=1$. The left panels show the Frobenius norms of the mean total error $A - \hA_N$ (crossed), and the right panels show the Frobenius norms of the mean finite precision error $A_N - \hA_N$ (solid) for every $k$ value versus $\beta$ and the finite precision error estimated by \eqref{eq:fin_prec_error} (dotted) with $k=9$ (top panel) and $k=10$ (bottom panel). For all plots, the colors correspond to $u_p$ set to half (purple), single (yellow), double (red), and `exact' (blue). When $k<10$ the total error is indistinguishable for all precisions.}
    \label{fig:synthetic_error_pol_decay}
\end{figure}

\subsubsection{SuiteSparse problems}
We now consider three symmetric positive definite problems from the SuiteSparse matrix collection \cite{SuiteSparse}. Their properties are summarised in Table~\ref{tbl:sparse_suit_problems_description}. The problems have different spectral properties including the decay of the largest eigenvalues and spectral gaps shown in Figure~\ref{fig:spectra}, and are of size $\mathcal{O}(10^2) - \mathcal{O}(10^3)$. The right hand-sides of the heuristics \eqref{eq:frob_heuristic_full} and \eqref{eq:frob_heuristic_no-eig-sum}, and the quantity $u_p$, which are used to estimate when the finite precision error becomes significant, are shown in Figure~\ref{fig:heuristics} for $u_p$ set to single and half precision. We note that for small $k$ value both \eqref{eq:frob_heuristic_full} and \eqref{eq:frob_heuristic_no-eig-sum} give similar estimates.

We report the total and finite precision errors in Figure~\ref{fig:suitsparse_error} for various $k$ values. The $k$ values are chosen so that we approximate eigenvalues throughout different parts of the spectrum; if there is a relatively large gap between eigenvalues $\lambda_j$ and $\lambda_{j+1}$, then we test $k=j$ and $k=j+1$. Note that this results in the nonuniform spacing of the $x$-axes in Figure~\ref{fig:suitsparse_error}.

The finite precision error $\| \E\|_F$ is approximated by \eqref{eq:fin_prec_error}. 
The error due to low precision affects the quality of the approximation when $\| E \|_F \approx \| \E \|_F$, which is the case for relatively large $k$ values. The heuristic \eqref{eq:frob_heuristic_full} is too pessimistic, whereas \eqref{eq:frob_heuristic_no-eig-sum} gives a good estimate of when this happens (Figure~\ref{fig:heuristics}). The finite precision error increases when eigenvalues close to a spectral gap are approximated; this can be attributed to the quantity $\| A \Omega \left(\Omega^T A \Omega \right)^{\dagger} \|_2$ (see Section~\ref{sec:prel_lemma}).

\begin{table}
\begin{center}
\begin{tabular}{ l| c| c|c }
Problem & $n$ & $\| A \|_2$ & half precision \\
\hline
bcsstm07 & 420 & $2.51 \times 10^3$ & yes\\
1138\_bus & 1138 & $3.01 \times 10^4$ & yes \\
nos7 & 729 & $9.86 \times 10^6$ & no \\
\end{tabular}
\caption{Problems from the SuiteSparse collection \cite{SuiteSparse}, where $A$ is an $n \times n$ positive definite matrix. Half precision is used for problems where the largest eigenvalues belong to the range of half precision, see Table~\ref{tab:ieee_param}.}
\label{tbl:sparse_suit_problems_description}
\end{center}
\end{table}

\begin{figure}
\begin{center}
\begin{subfigure}[b]{0.3\linewidth}
  \centering
 \includegraphics[width=\linewidth]{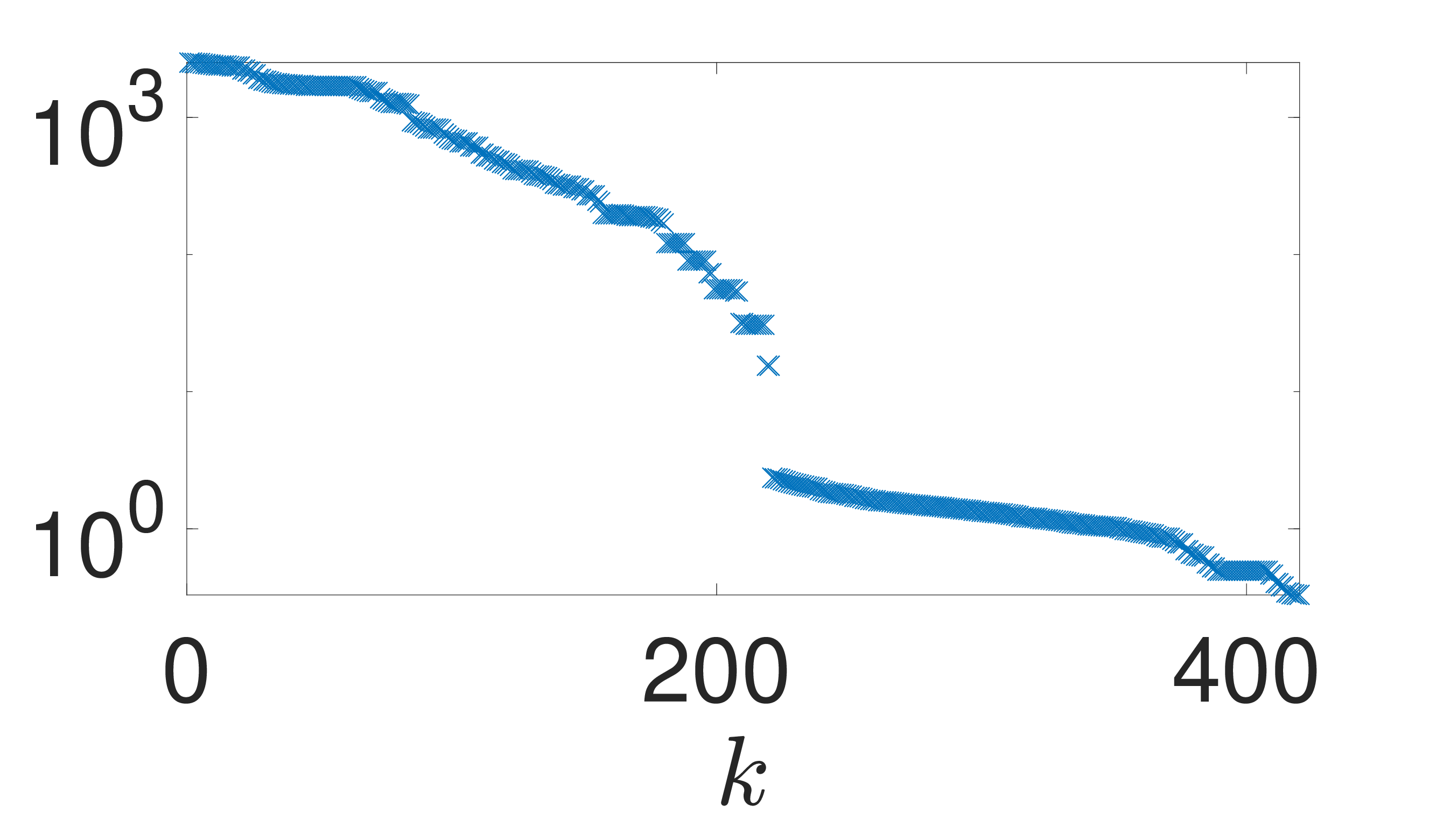}
  \caption{bcsstm07}
\end{subfigure}
\begin{subfigure}[b]{0.3\linewidth}
  \centering
 \includegraphics[width=\linewidth]{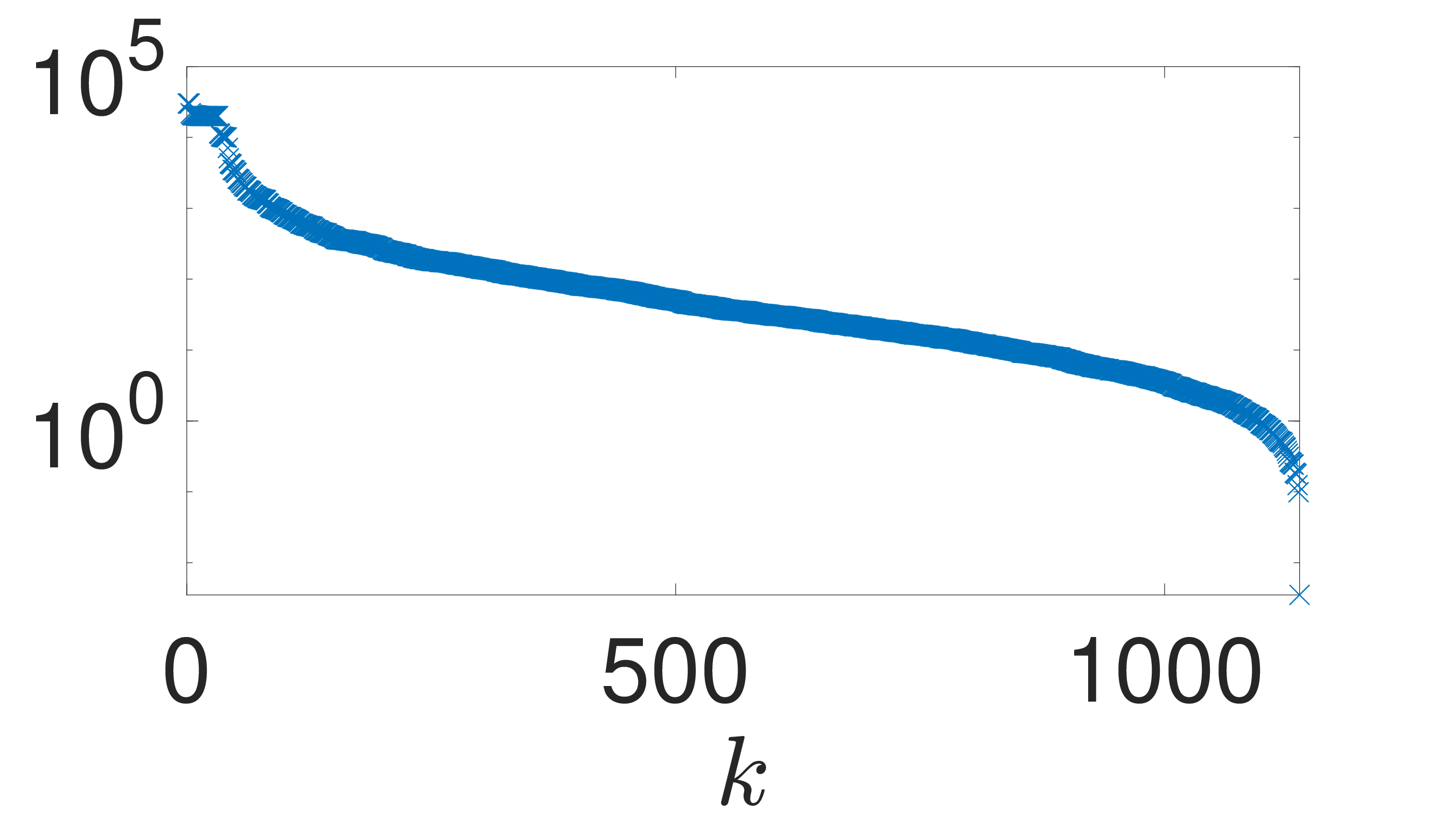}
  \caption{1138\_bus}
\end{subfigure}
\begin{subfigure}[b]{0.3\linewidth}
  \centering
 \includegraphics[width=\linewidth]{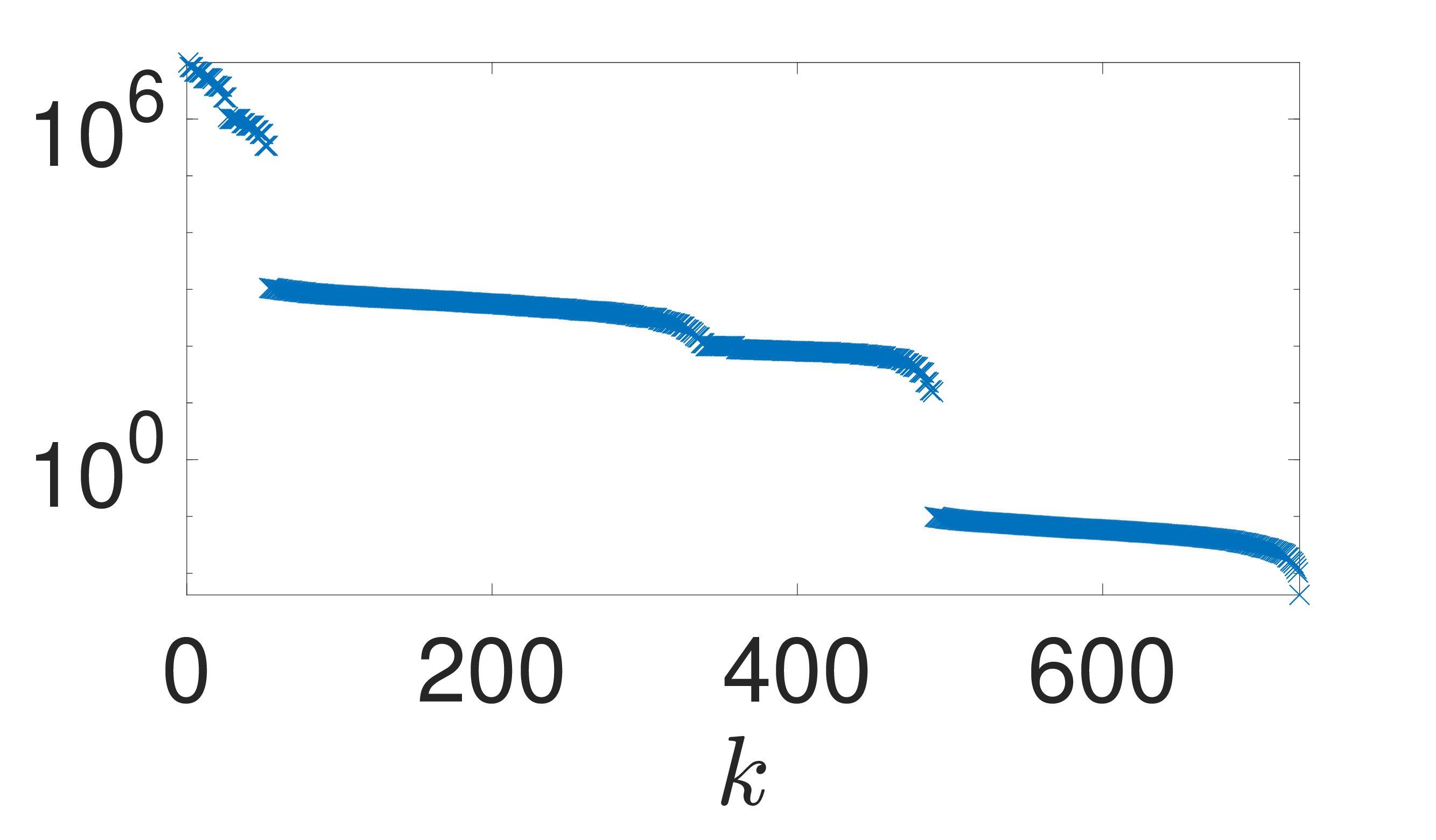}
  \caption{nos7}
\end{subfigure}
 \end{center}
    \caption{Spectra of the problems in Table~\ref{tbl:sparse_suit_problems_description}.}
    \label{fig:spectra}
\end{figure}

\begin{figure}
\begin{center}
\begin{subfigure}[b]{0.3\linewidth}
  \centering
 \includegraphics[width=\linewidth]{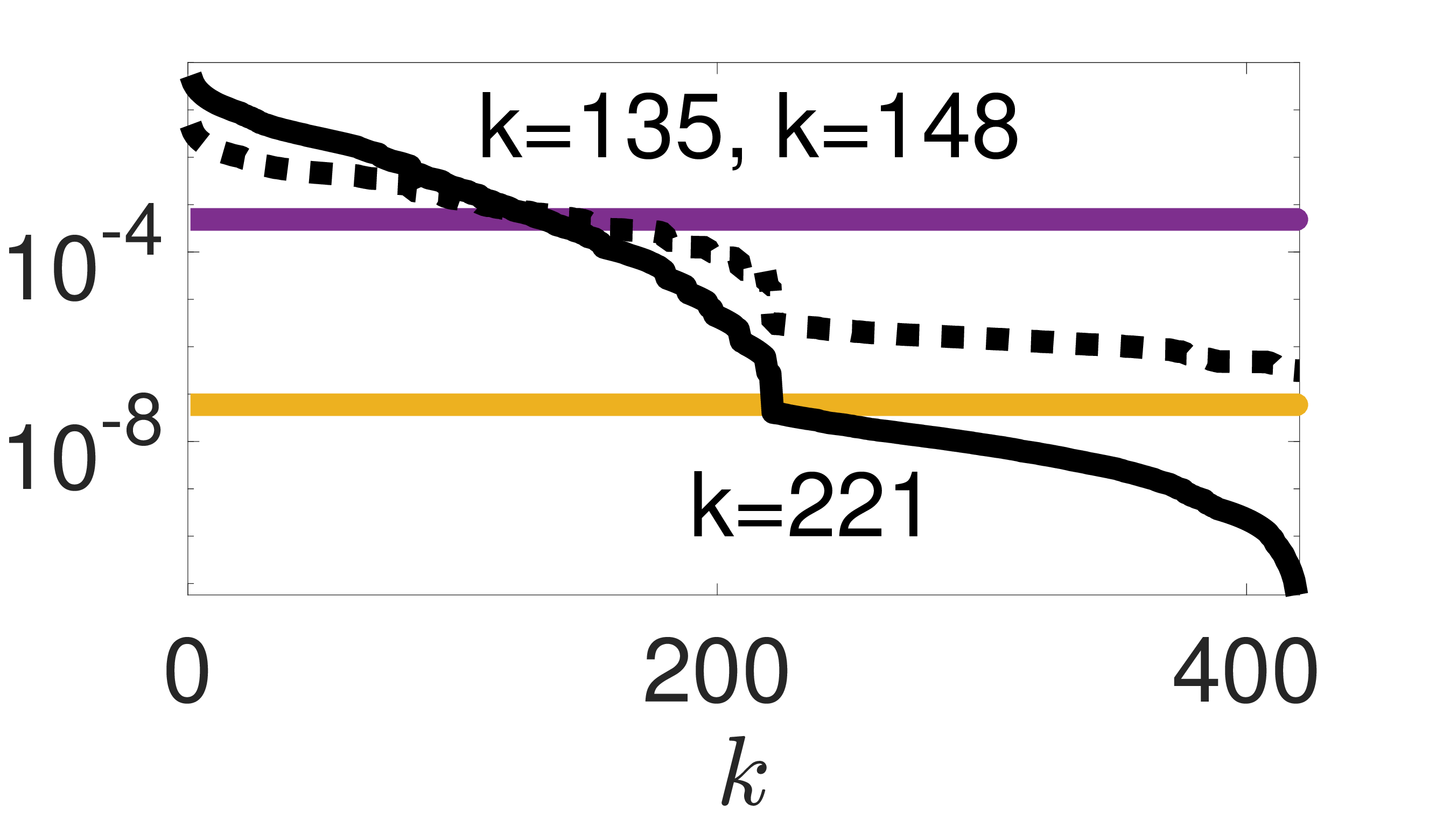}
  \caption{bcsstm07}
\end{subfigure}
\begin{subfigure}[b]{0.3\linewidth}
  \centering
 \includegraphics[width=\linewidth]{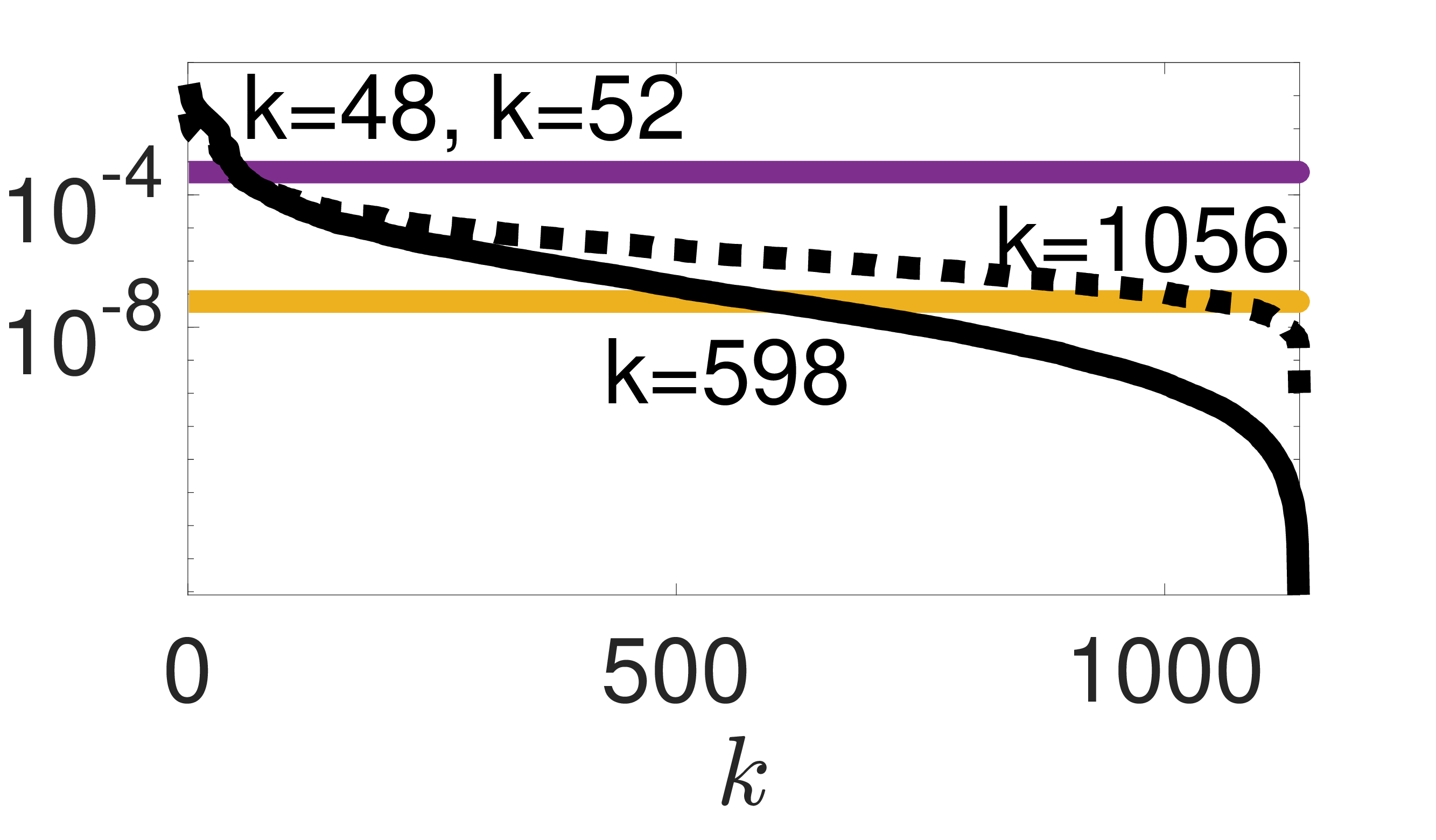}
  \caption{1138\_bus}
\end{subfigure}
\begin{subfigure}[b]{0.3\linewidth}
  \centering
 \includegraphics[width=\linewidth]{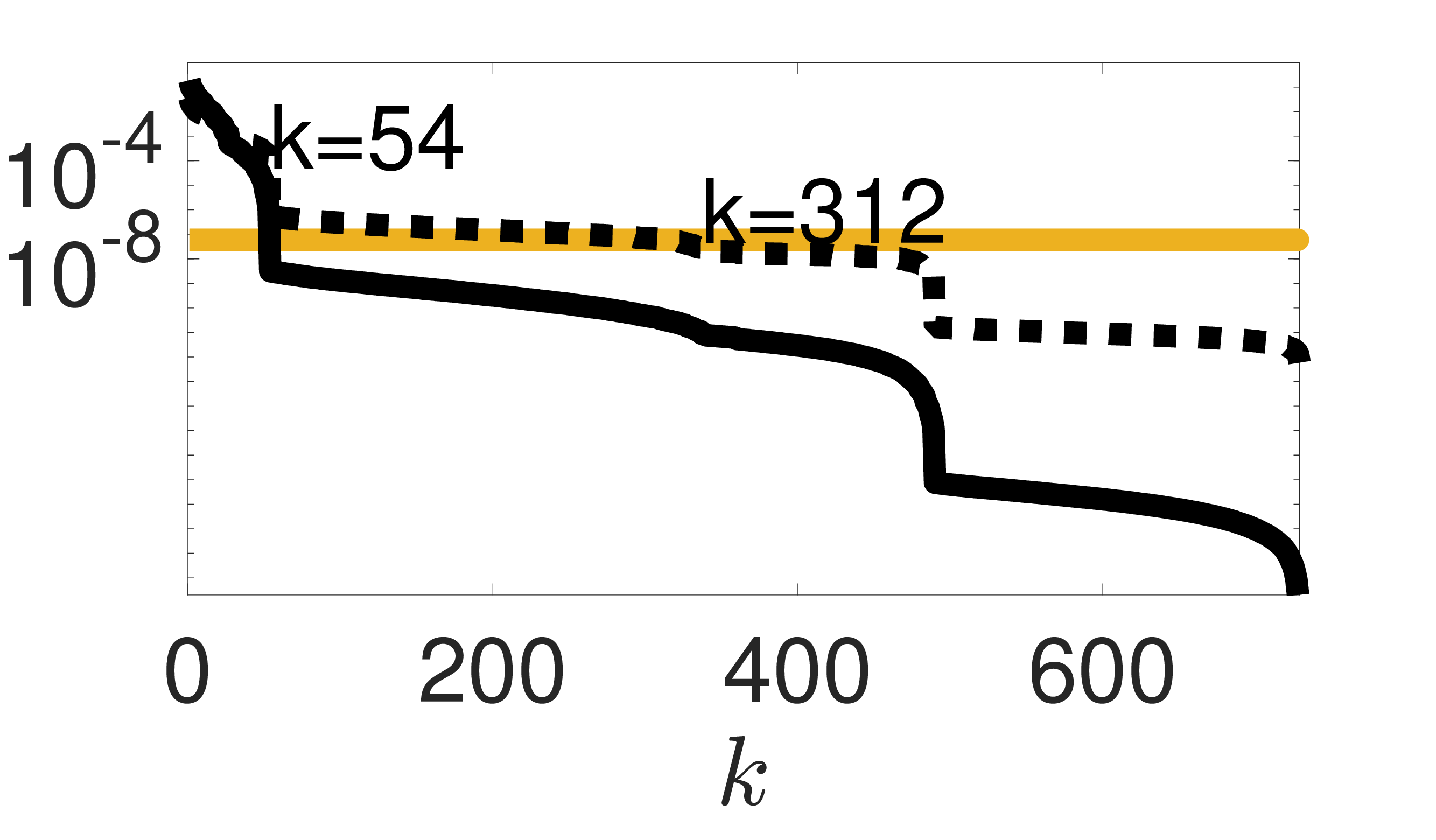}
  \caption{nos7}
\end{subfigure}
 \end{center}
    \caption{The right hand side of heuristic \eqref{eq:frob_heuristic_full} (black solid) and \eqref{eq:frob_heuristic_no-eig-sum} (black dotted) versus $k$, and $u_p$ for the problems in Table~\ref{tbl:sparse_suit_problems_description} with text indicating the point of intersection. For the  lines showing $u_p$, colours indicating single (yellow) and half (purple) precision are the same as in Figure~\ref{fig:synthetic_error_pol_decay}.}
    \label{fig:heuristics}
\end{figure}

\begin{figure}
\begin{center}
\begin{subfigure}[b]{0.49\linewidth}
  \centering
 \includegraphics[width=\linewidth]{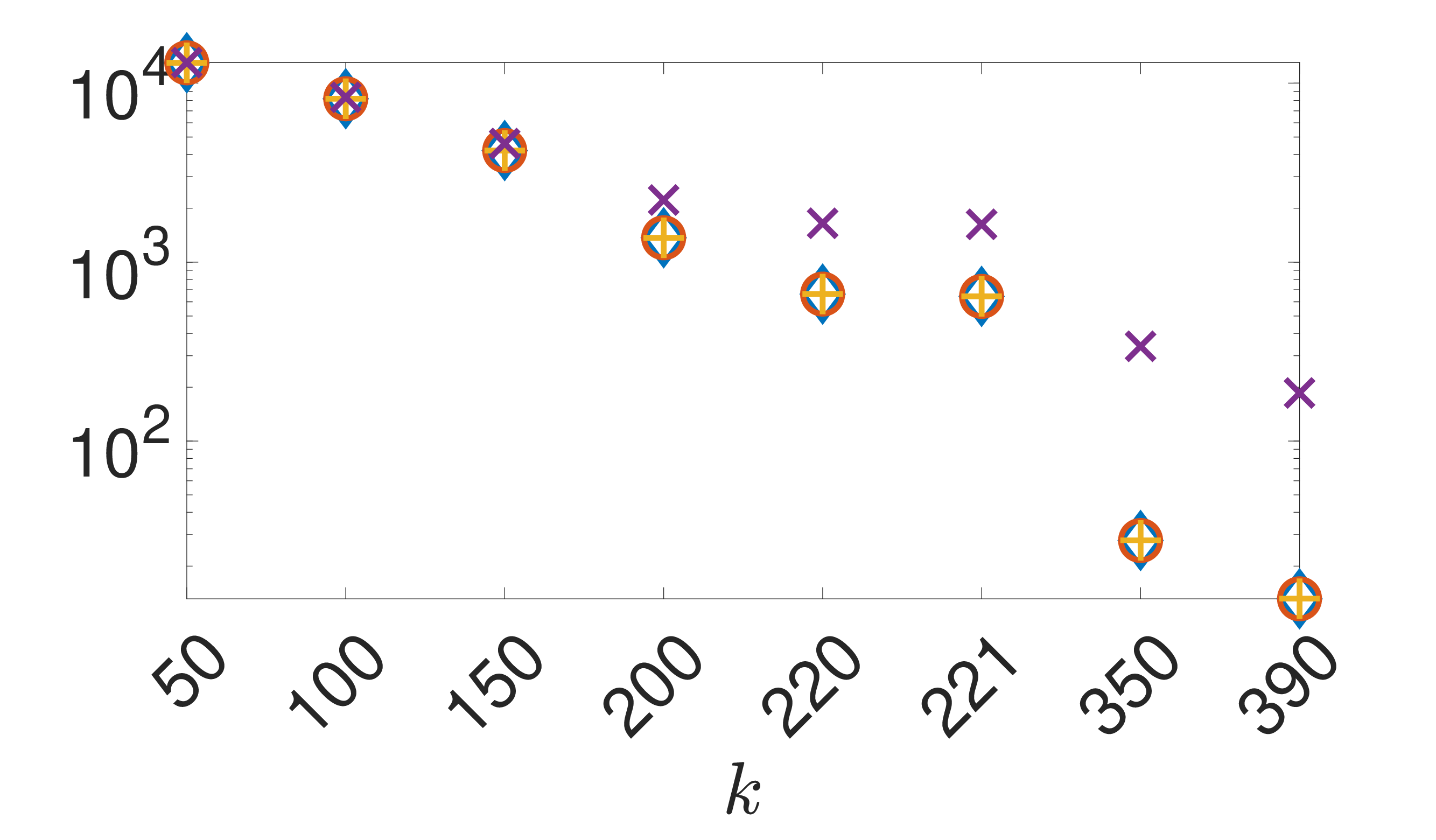}
  \caption{bcsstm07}
\end{subfigure}
\begin{subfigure}[b]{0.49\linewidth}
  \centering
 \includegraphics[width=\linewidth]{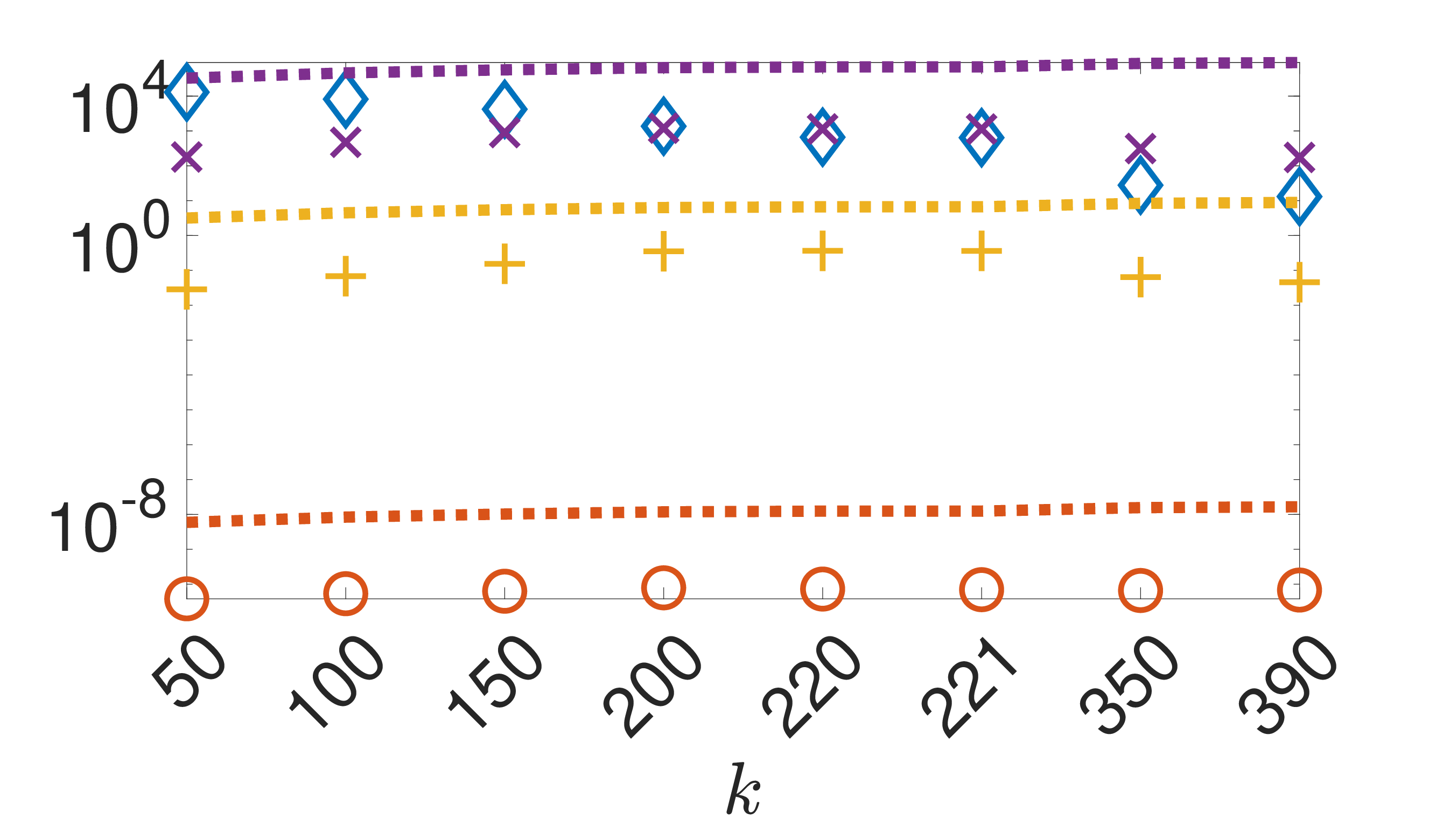}
  \caption{bcsstm07}
\end{subfigure}
\vspace{-2mm}
\begin{subfigure}[b]{0.49\linewidth}
  \centering
 \includegraphics[width=\linewidth]{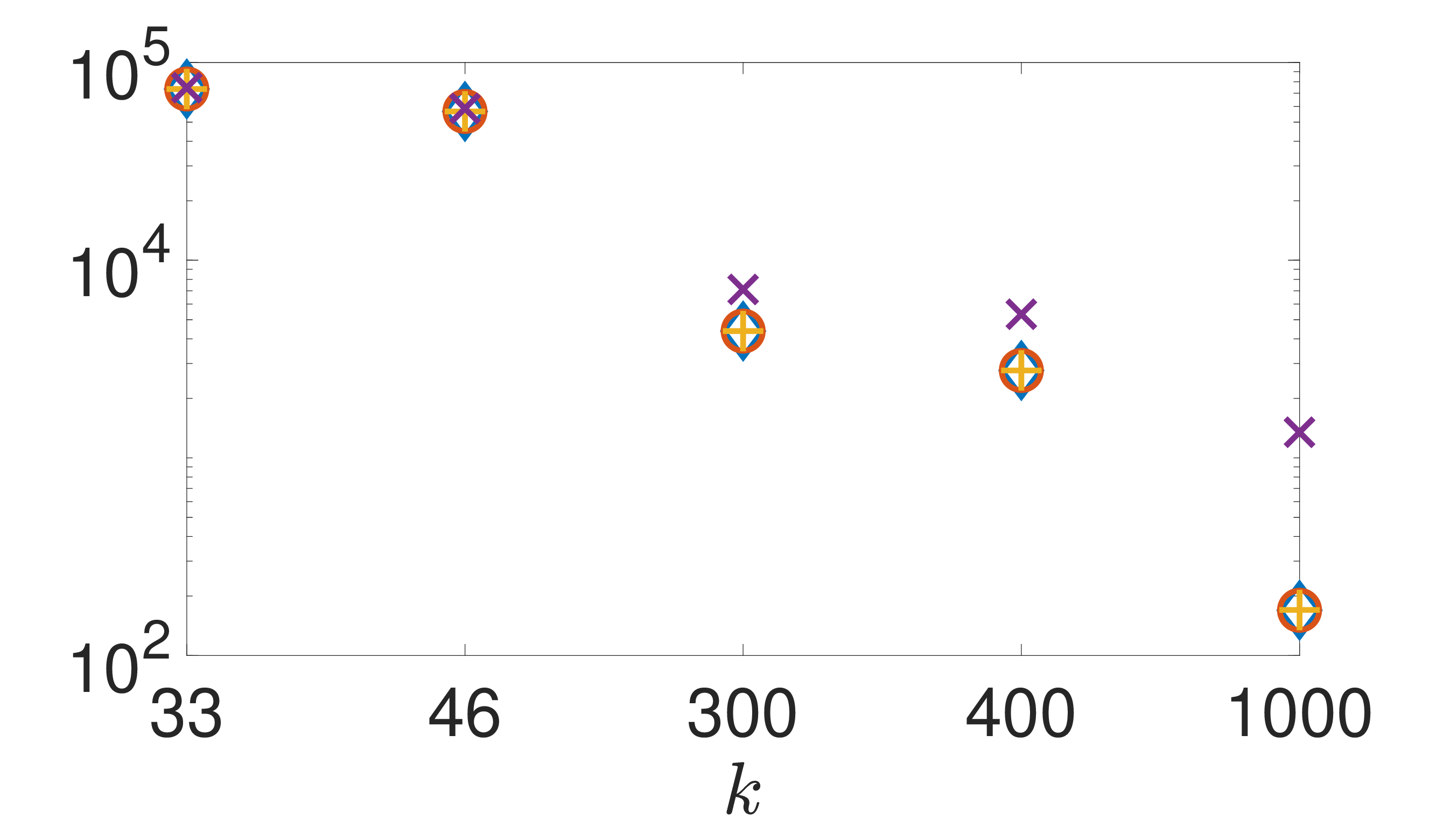}
  \caption{1138\_bus}
\end{subfigure}
\begin{subfigure}[b]{0.49\linewidth}
  \centering
 \includegraphics[width=\linewidth]{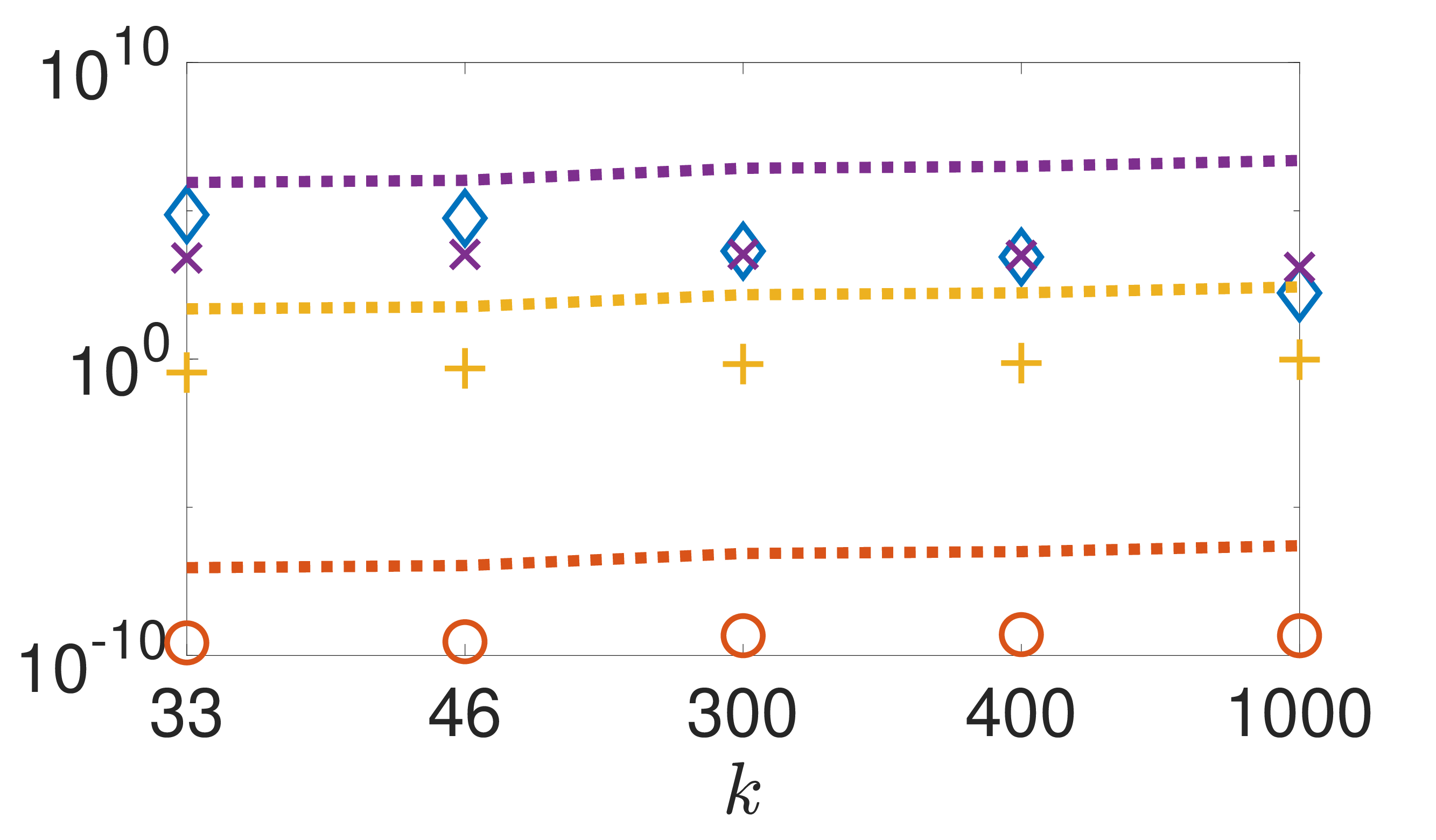}
  \caption{1138\_bus}
\end{subfigure}
\begin{subfigure}[b]{0.49\linewidth}
  \centering
 \includegraphics[width=\linewidth]{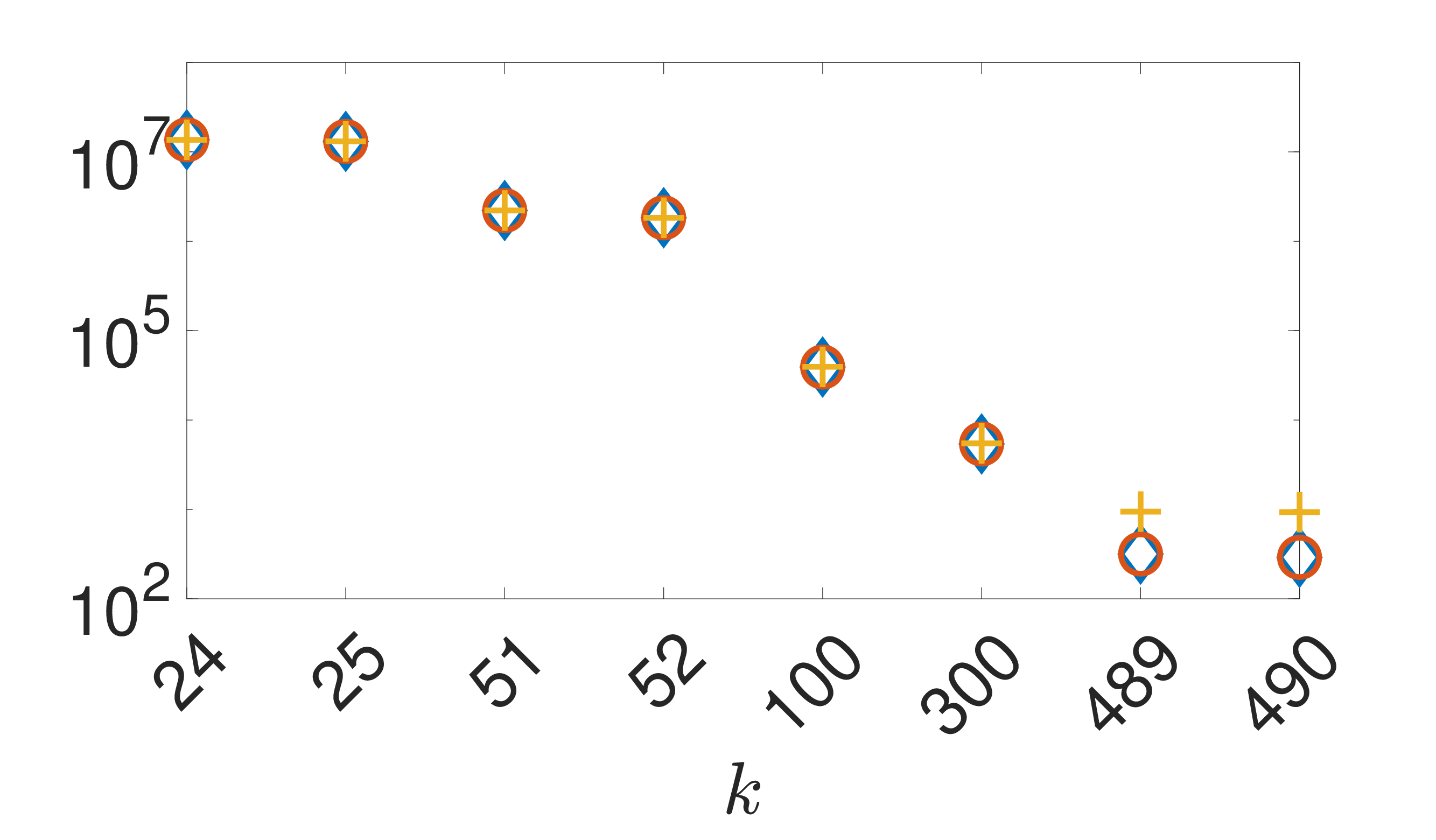} 
  \caption{nos7}
\end{subfigure}
\begin{subfigure}[b]{0.49\linewidth}
  \centering
 \includegraphics[width=\linewidth]{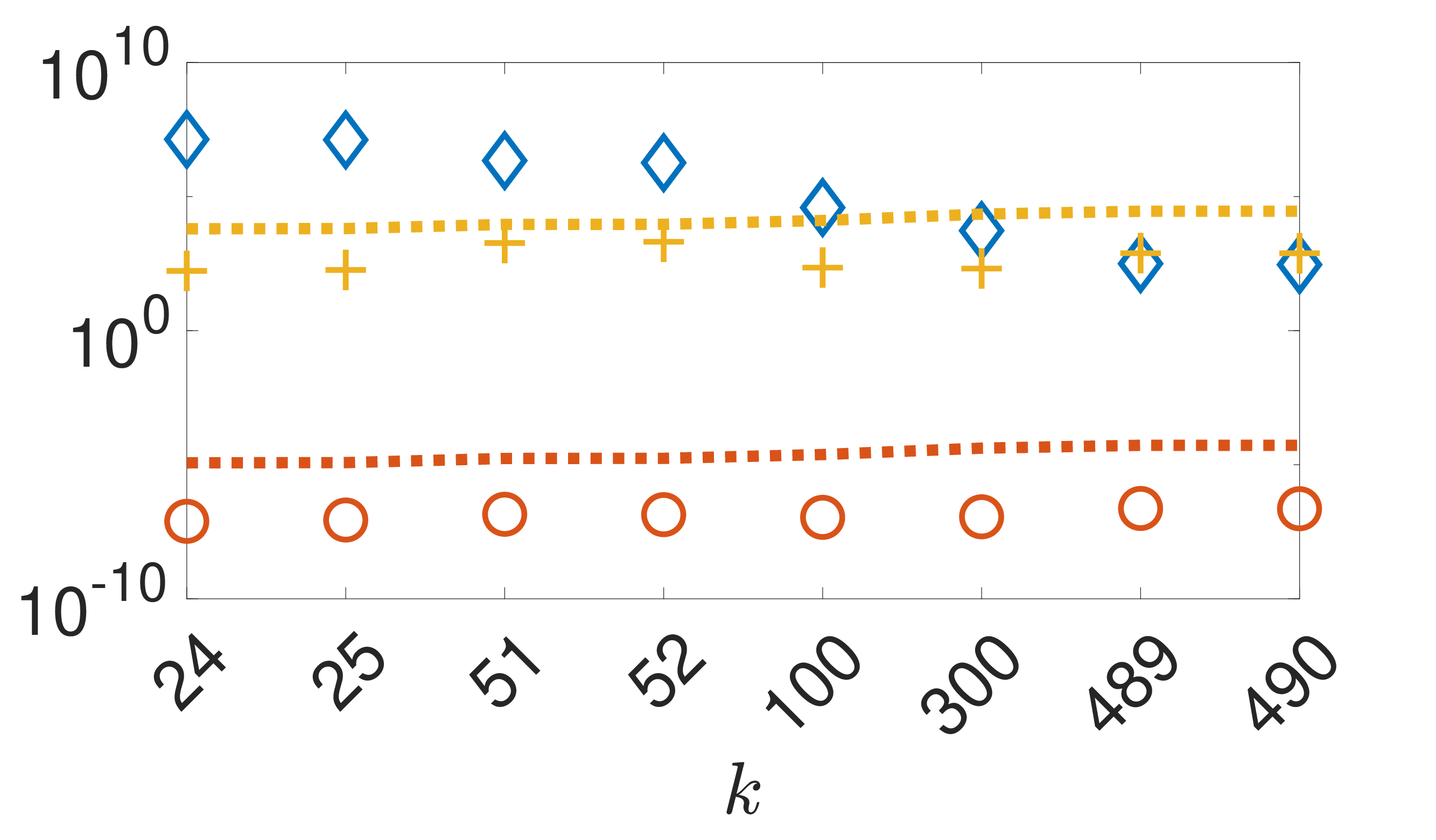}
  \caption{nos7}
\end{subfigure}
\end{center}
    \caption{SuiteSparse problems. The left panels show Frobenius norms of the mean total error $\|A - \hA_N\|_F$ (markers). The right panels show the mean finite precision error $\|A_N - \hA_N\|_F$ (markers), the mean exact approximation error $\|A-A_N\|_F$ (blue diamonds), and estimates of the finite precision error (dashed lines) versus the rank of approximation $k$.
    In all panels, the precision $u_p$ is indicated by the colour and the type of the markers: purple crosses denote half, yellow pluses denote single, red circles denote double, and blue diamonds denote `exact'.
    }
    \label{fig:suitsparse_error}
\end{figure}

\subsection{Preconditioned systems}\label{section:precond_numerics}
We are interested in comparing the preconditioning performance for shifted systems \eqref{eq:shifted_system} when the preconditioner \eqref{eq:prec:LMP_nystrom_shifted_A+muI_inverse_finite-prec} is constructed using an approximation computed via Algorithm~\ref{alg:nystrom_reg_id} with different precisions $u_p$. 
This is done by computing the condition number of preconditioned systems and solving them via PCG in double precision. We consider problems from the SuiteSparse collection described in the previous section. 
We consider the same $k$ values as in the previous section and refer the reader to \cite{Frangella2021} for strategies on choosing an optimal $k$. Note that the computational cost of generating the preconditioner depends on the precision $u_p$, but the cost of applying the preconditioner does not.

\subsubsection{Condition number}
We compute the condition number of split-\linebreak[4] preconditioned matrices 
\begin{equation*}
    \hPisqr (A + \mu I) \hPisqr,
\end{equation*}
where $\hPisqr$ is constructed as defined in \eqref{eq:prec:LMP_nystrom_shifted_A+muI_inverse_finite-prec}. 

To simplify notation, we denote the quantities in the condition number bounds in Theorem~\ref{prop:prec_condition_no_deterministic} as follows:
\begin{align*}
    b_\text{low} &=  \max \left\{ 1, \frac{\hlambda_k + \mu - \| \E \|_2}{\mu + \lambda_{min} (A) } \right\}, \\
    b_\text{upp} &=  1 + \frac{\hlambda_k + \| E \|_2 + 2\| \E \|_2}{\mu - \| \E \|_2 }, \\
    b_\text{uppspd} &=  \left(  \hlambda_k + \mu + \| E \|_2 + \| \E \|_2 \right) \left( \frac{1}{\hlambda_k + \mu} + \frac{\| \E \|_2 +1 }{\lambmin(A) + \mu} \right).
\end{align*}
We compute approximations of these quantities by replacing $\| E \|_2$ with the expected error $\mathbb{E}\, \| E \|_2$ in \eqref{eq:bound_expected_error_FTU} and $\| \E \|_2$ with the approximation \eqref{eq:fin_prec_error} for $\| \E \|_F$. We set $\mu$ to $0.1$, $0.5$ and $1$. Thus $b_\text{upp}$ is only computed when $u_p$ is set to double precision for all problems.

We report results for $\mu=0.5$ in Figure~\ref{fig:mean-condition-number-estimated-bounds-all-problems-mu-05}; the results for different $\mu$ values are similar. The mean condition numbers are bounded by the mean estimated bounds and the preconditioning reduces the condition number when an appropriate $k$ is chosen. The difference in the condition number with different $u_p$ arises when there is a difference in the total approximation error. The estimated bounds get worse with lower precision. The lower bound for half precision for all the problems and single precision for large $k$ values is equal to one and thus not useful. $b_\text{upp}$ and $b_\text{uppspd}$ either coincide or are very similar.
 
\begin{figure}
\begin{center}
\begin{subfigure}[b]{0.49\linewidth}
  \centering
 \includegraphics[width=\linewidth]{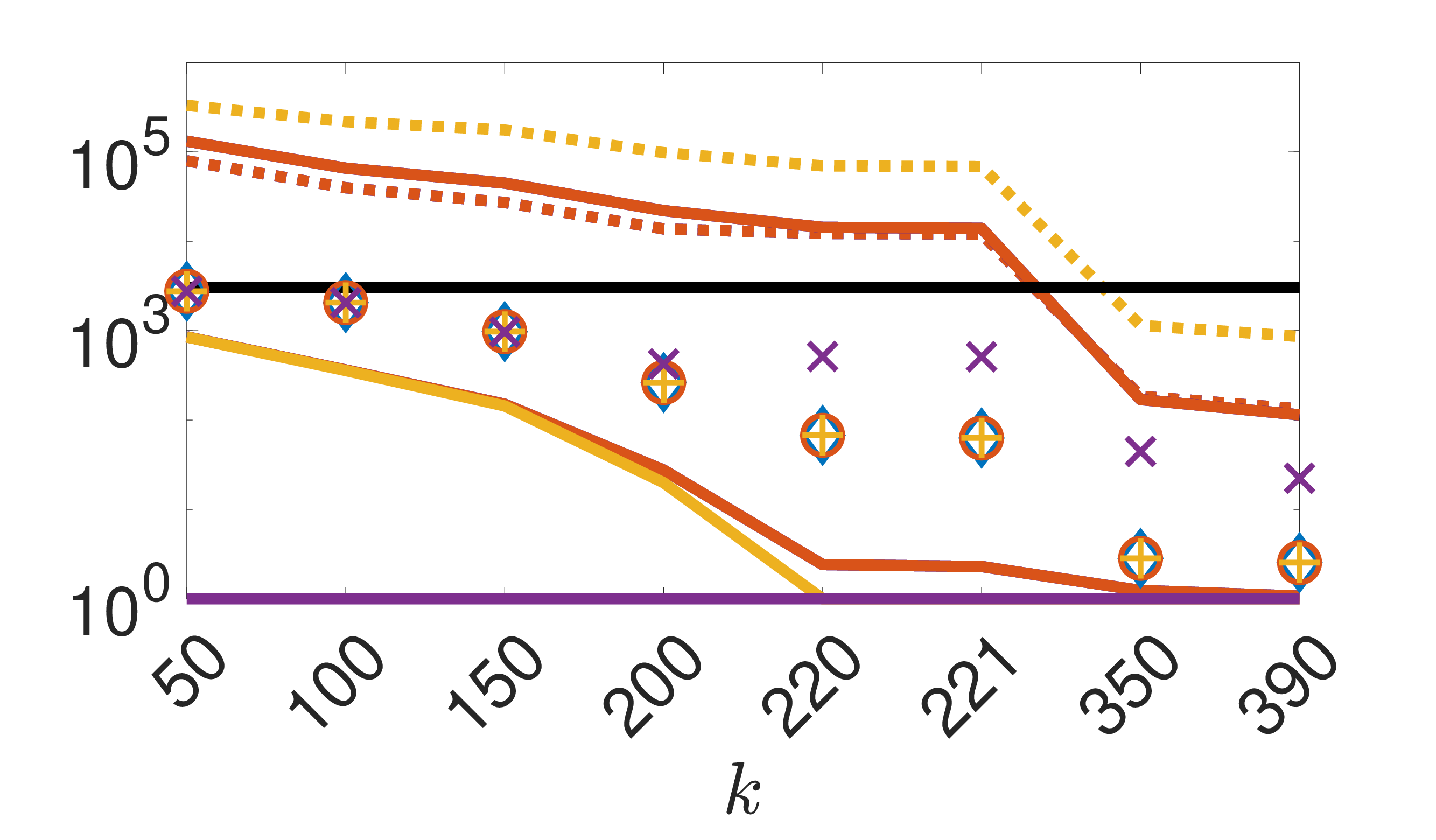}
  \caption{bcsstm07}
\end{subfigure}
\begin{subfigure}[b]{0.49\linewidth}
  \centering
 \includegraphics[width=\linewidth]{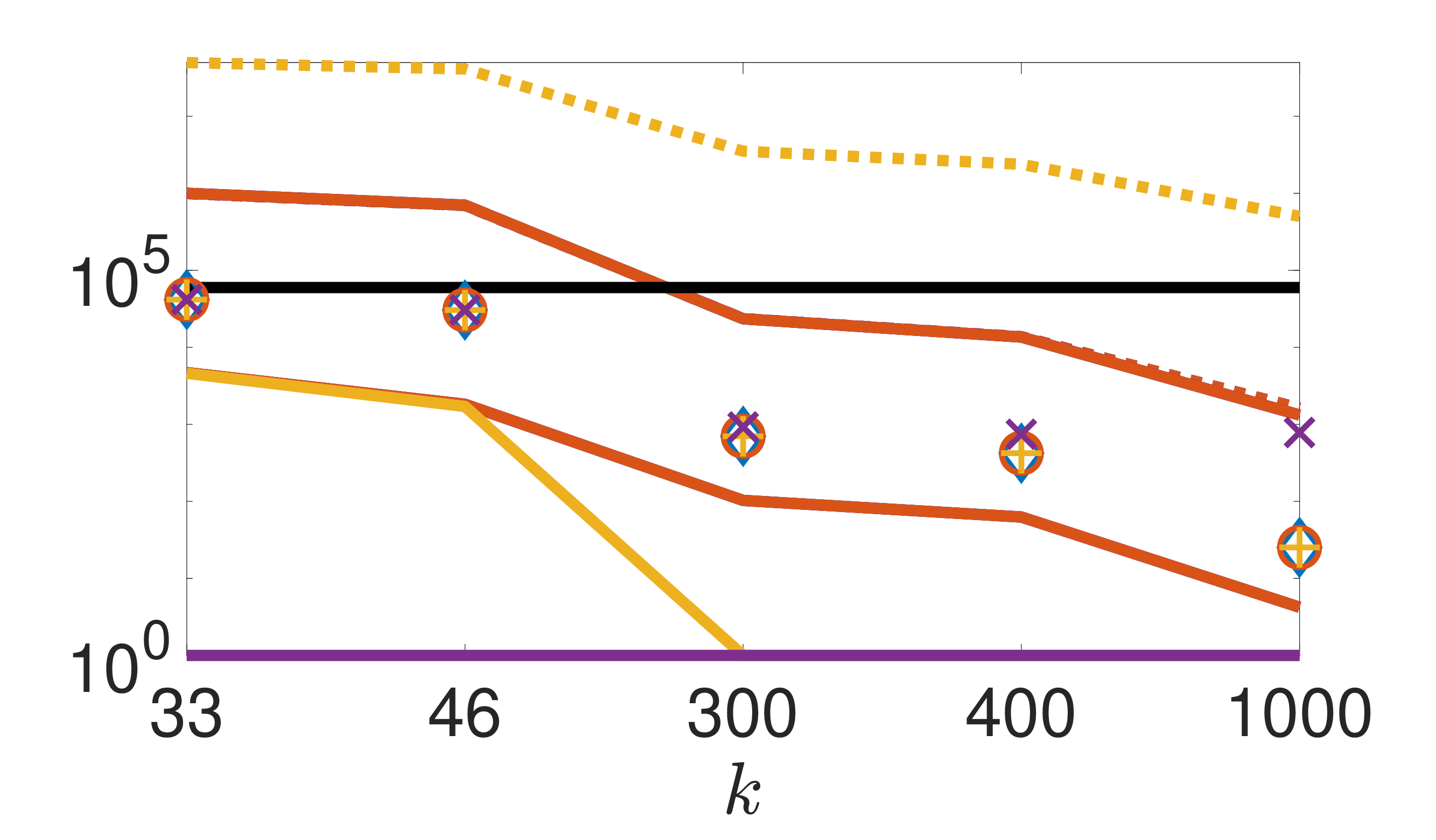}
  \caption{1138\_bus}
\end{subfigure}
\begin{subfigure}[b]{0.49\linewidth}
  \centering
 \includegraphics[width=\linewidth]{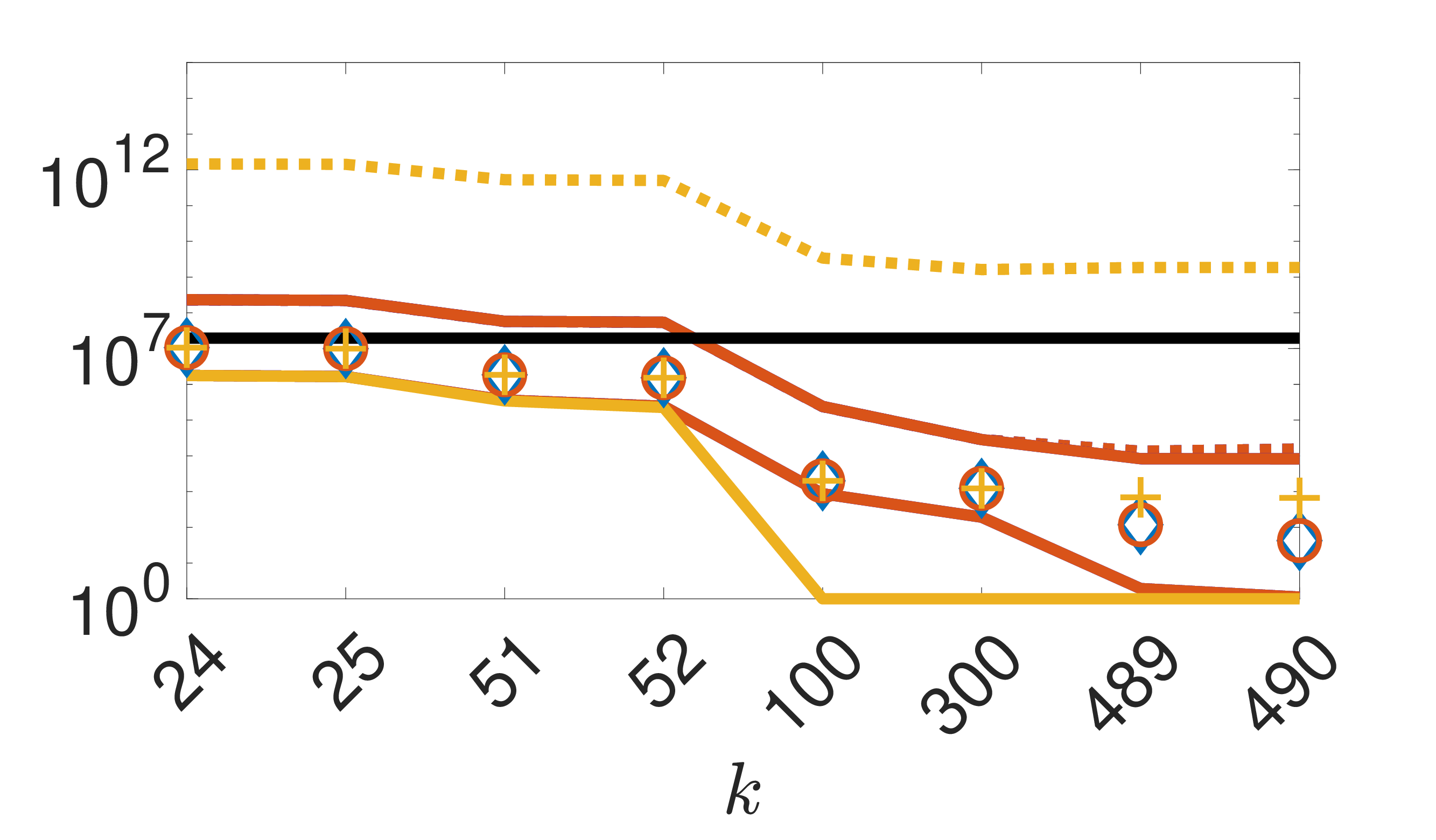}
  \caption{nos7}
\end{subfigure}
\end{center}
    \caption{Condition number of $A +0.5 I$ (black solid line) and the mean condition number of $\hPisqr (A + 0.5 I) \hPisqr$ (markers) with preconditioner \eqref{eq:prec:LMP_nystrom_shifted_A+muI_inverse} constructed using a rank-$k$ Nystr\"{o}m approximation. The mean estimates of the bounds $b_\text{low}$ (solid), $b_\text{upp}$ (solid; double only) and $b_{uppspd}$ (dotted). The bounds $b_{uppspd}$ with half precision are $\mathcal{O}(10^{9})-\mathcal{O}(10^{10})$ for the bcsstm07 problem and $\mathcal{O}(10^{12})-\mathcal{O}(10^{13})$ for 1138\_bus (not shown). The precision $u_p$ is indicated by the colour and the type of the markers: purple crosses denote half, yellow pluses denote single, red circles denote double, and blue diamonds denote `exact'.
    }
    \label{fig:mean-condition-number-estimated-bounds-all-problems-mu-05}
\end{figure}

\subsubsection{Solving the systems}
We now solve the systems in \eqref{eq:shifted_system} with MATLAB's built-in PCG with left-preconditioning, that is,
\begin{equation*}
    \hPi (A + \mu I) x = \hPi b.
\end{equation*}
$\hPi$ is constructed as in the previous section. The entries of $b$ are uniformly distributed random numbers where the random number generator seed is set to $1234$. We set the stopping tolerance to $10^{-6}$.  

The mean iteration count results in Figure~\ref{fig:mean-iteration-count-all-problems-mu-05} correspond to the condition number results in Figure~\ref{fig:mean-condition-number-estimated-bounds-all-problems-mu-05}, although the preconditioner constructed by setting $u_p$ to single and $k \in \{489,490\}$ for the nos7 problem is not useful even if the condition number of the preconditioned system does not grow significantly. There is a modest increase in the number of iterations when the total error is influenced by the use of smaller precision. We note that using a smaller shift when computing the Nystr\"{o}m approximation, for example $\nu = 2u_p \Vert Y \Vert_2$, can give results less sensitive to $u_p$ (not shown). However, for small $k$ values independent of the shift the precision does not have a meaningful influence. This indicates that in many practical cases, the mixed precision Nystr\"{o}m method is likely suitable for use in generating preconditioners for PCG.

\begin{figure}
\begin{center}
\begin{subfigure}[b]{0.49\linewidth}
  \centering
 \includegraphics[width=\linewidth]{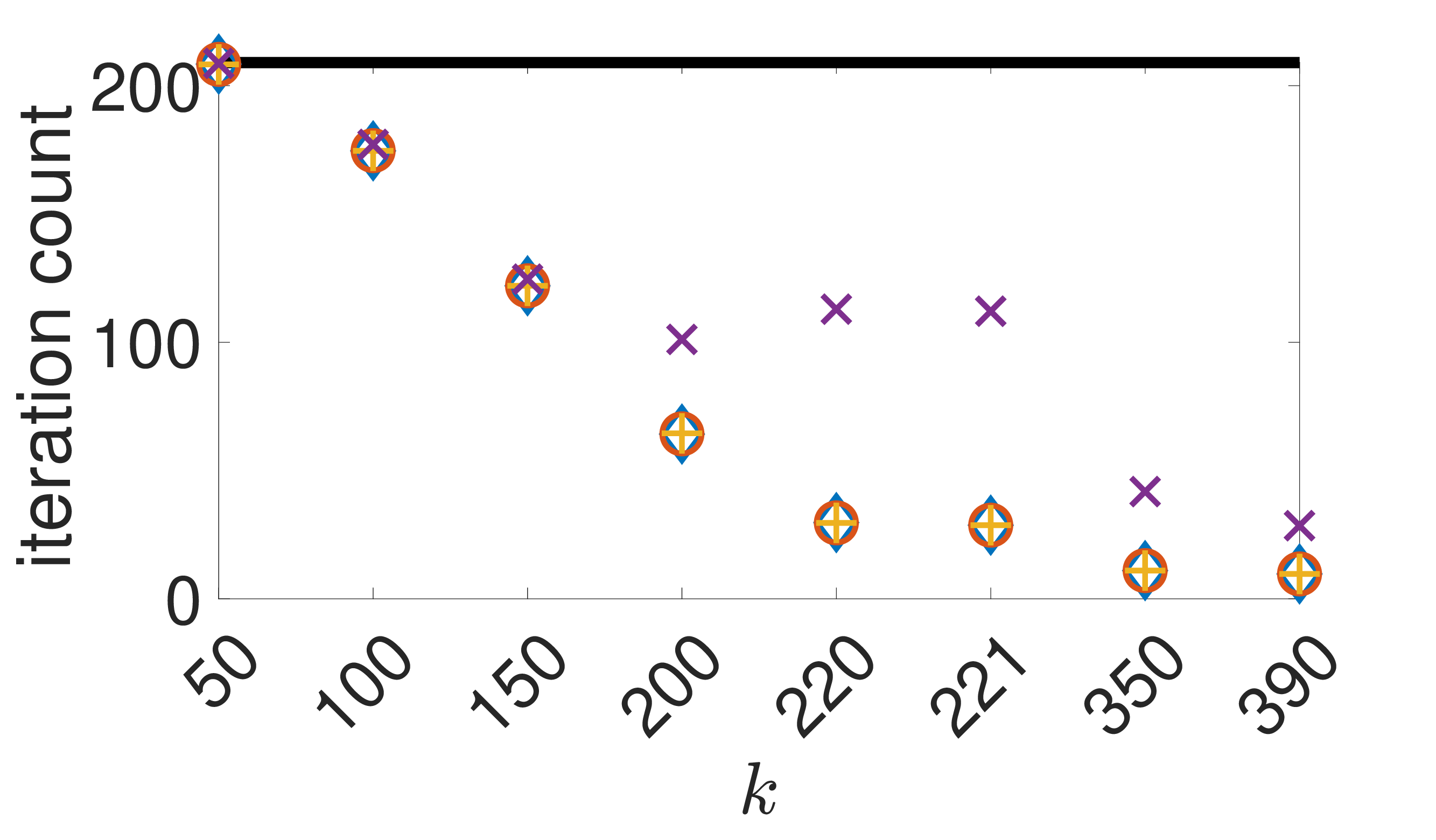}
  \caption{bcsstm07}
\end{subfigure}
\begin{subfigure}[b]{0.49\linewidth}
  \centering
 \includegraphics[width=\linewidth]{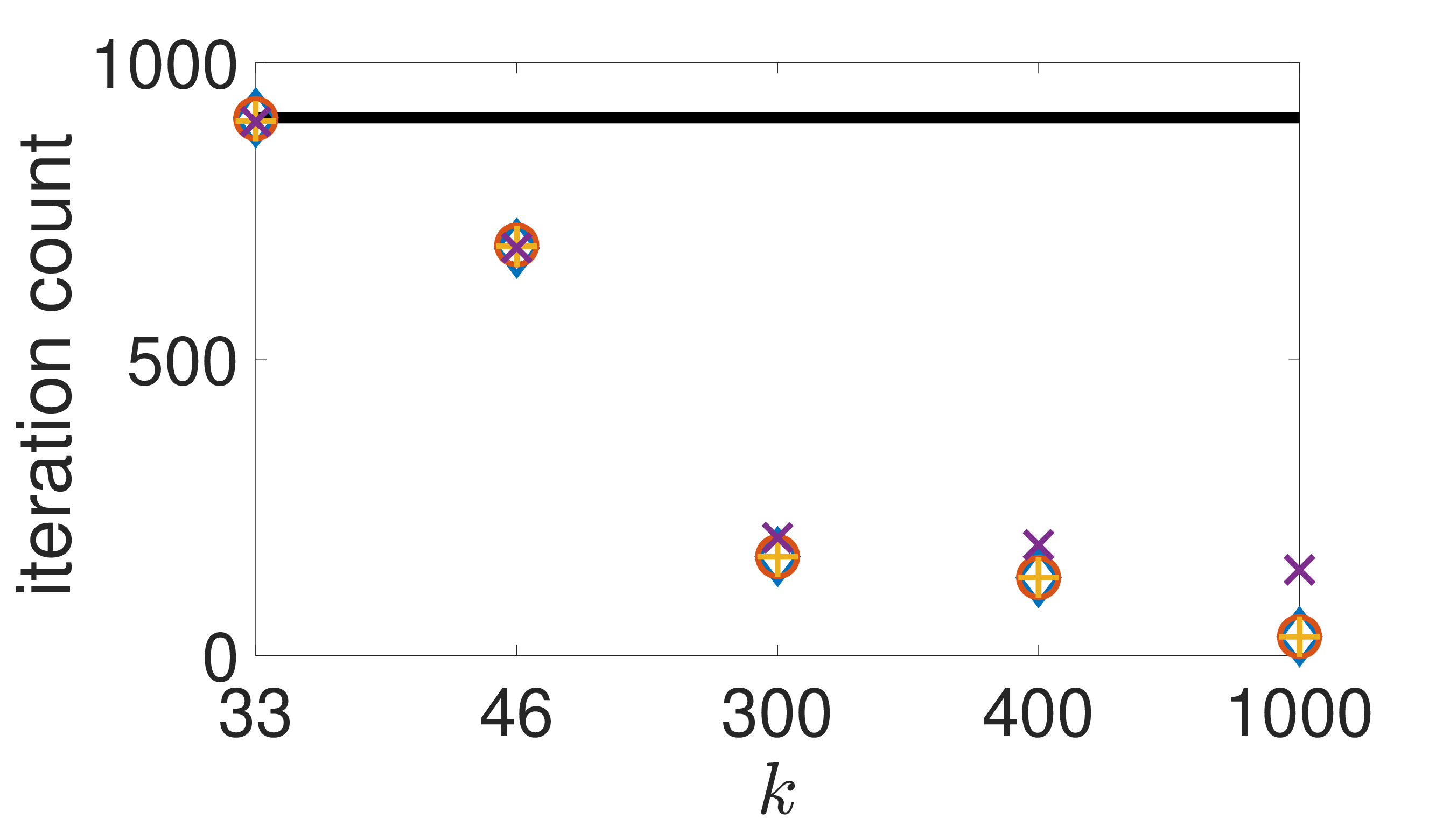}
  \caption{1138\_bus}
\end{subfigure}
\begin{subfigure}[b]{0.49\linewidth}
  \centering
 \includegraphics[width=\linewidth]{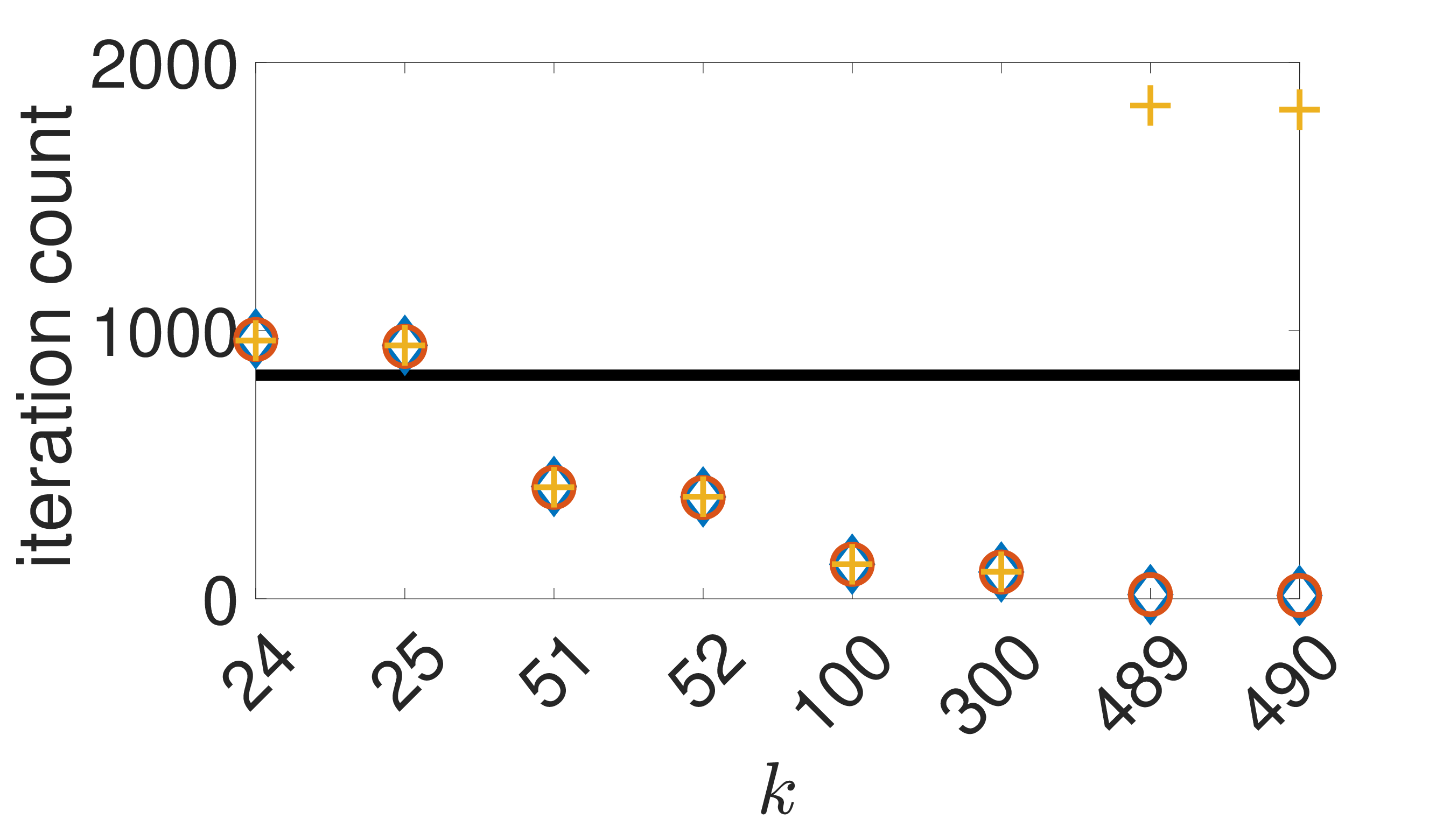}
  \caption{nos7}
\end{subfigure}
\end{center}
    \caption{Mean PCG iteration count when solving $(A +0.5 I)x = b$ without preconditioning (black) and with preconditioner \eqref{eq:prec:LMP_nystrom_shifted_A+muI_inverse} constructed using a rank-$k$ Nystr\"{o}m approximation (markers). For nos7 problem PCG did not converge in $3n$ iterations with one of the ten preconditioners when $k=490$. The precision $u_p$ is indicated by the colour and the type of the markers: purple crosses denote half, yellow pluses denote single, red circles denote double, and blue diamonds denote `exact'.
    }
    \label{fig:mean-iteration-count-all-problems-mu-05}
\end{figure}

\subsection{Kernel ridge regression problem}\label{section:krr}
We consider a linear system of equations arising in kernel ridge regression; see, for example \cite{scholkopf2002learning}. We randomly sample 1184 inputs $y_i \in \mathbb{R}^{22}$ and their corresponding outputs $b_i \in \{ -1, 1 \}$ from the ijcnn1 dataset from LIBSVM \cite{Chang2011LIBSVMAL} using \textit{libsvm2mat} function \cite{tpl_libsvm_func}. The $1184 \times 1184$ matrix $A$ is obtained as a Gaussian kernel, i.e., 
\begin{equation*}
    A_{ij} = \exp(- \|y_i - y_j \|^2_2/2 \sigma^2),
\end{equation*}
where we set $\sigma = 0.5$ as in \cite{Frangella2021}. 

The results are similar to those presented in previous sections and the heuristic \eqref{eq:frob_heuristic_no-eig-sum} gives a suitable estimate of when the finite precision error can be ignored. See Figure~\ref{fig:ijcnn1_approx} for the spectra, and means of the total and finite approximation errors when $k \in \{30,50,100,500,900\}$, and for the mean condition number of the preconditioned shifted systems and mean iteration count with $\mu = 10^{-2}$ (changing $u_p$ has similar effect with $\mu \in \{10^{-1},1\}$; results not shown).

\begin{figure}
\begin{center}
\begin{subfigure}[b]{0.49\linewidth}
  \centering
 \includegraphics[width=\linewidth]{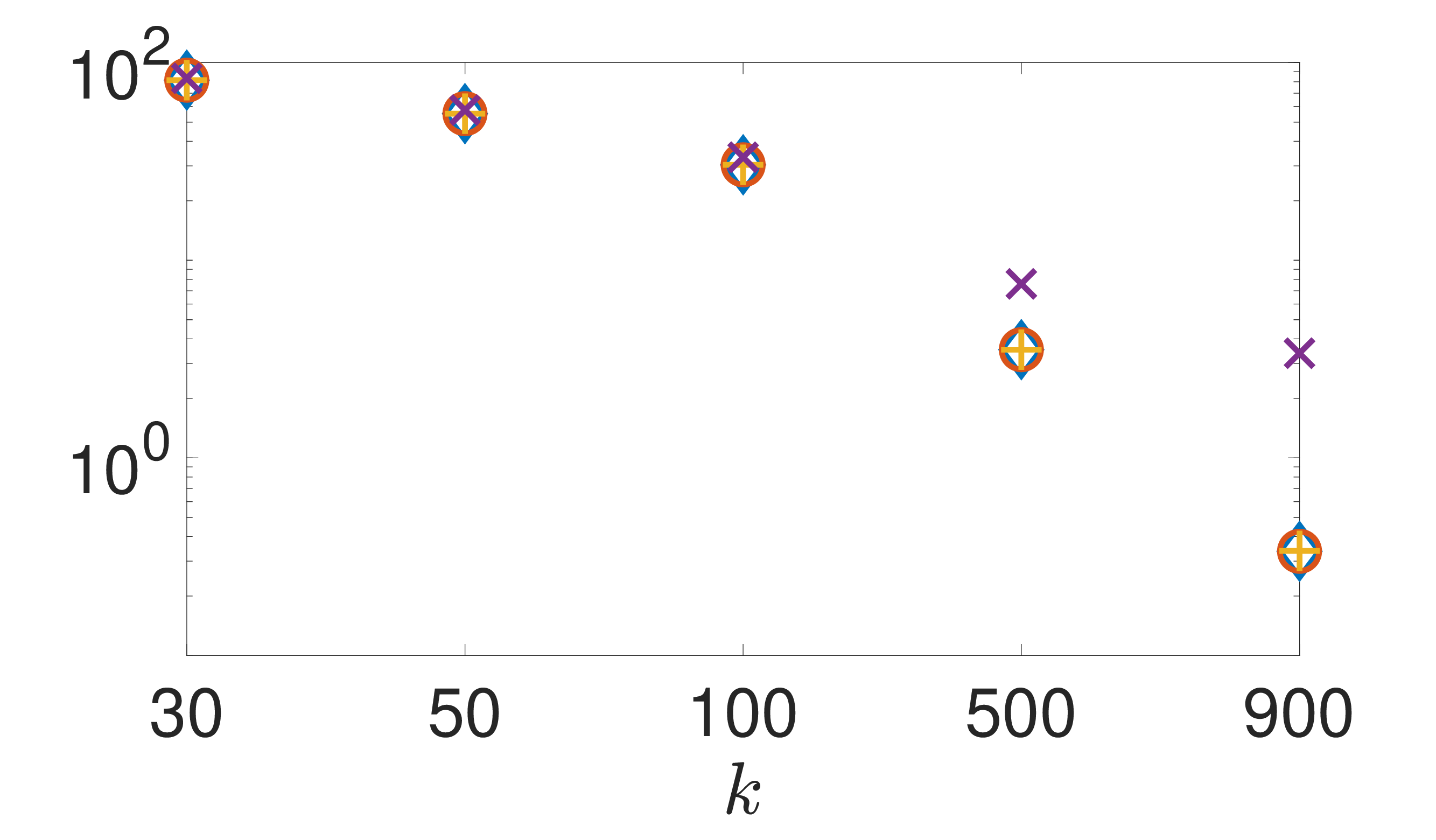}
  \caption{$\| A - \hA_N \|_F$}
\end{subfigure}
\begin{subfigure}[b]{0.49\linewidth}
  \centering
 \includegraphics[width=\linewidth]{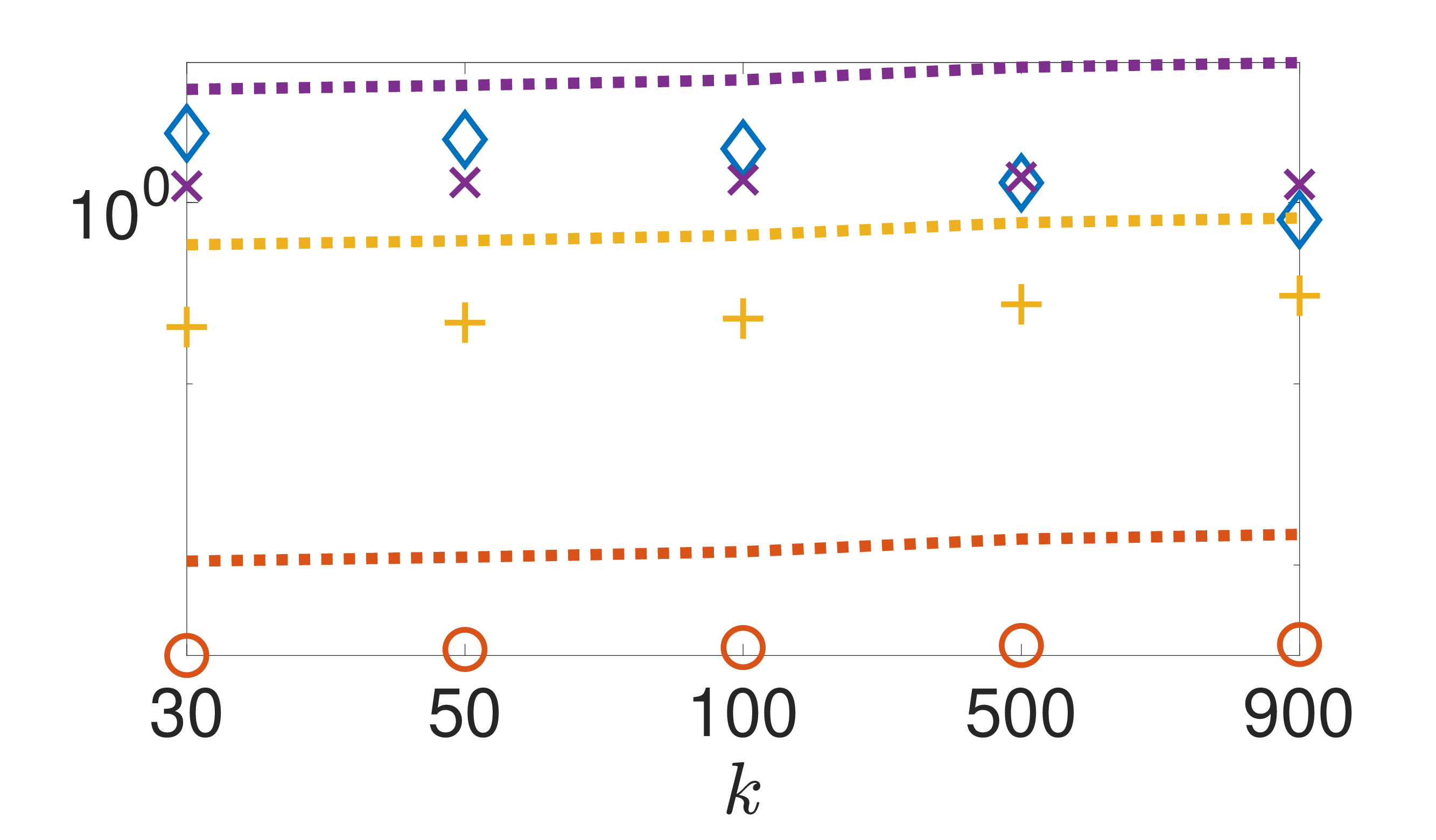}
  \caption{$\| A_N - \hA_N \|_F$}
\end{subfigure}
\begin{subfigure}[b]{0.49\linewidth}
  \centering
 \includegraphics[width=\linewidth]{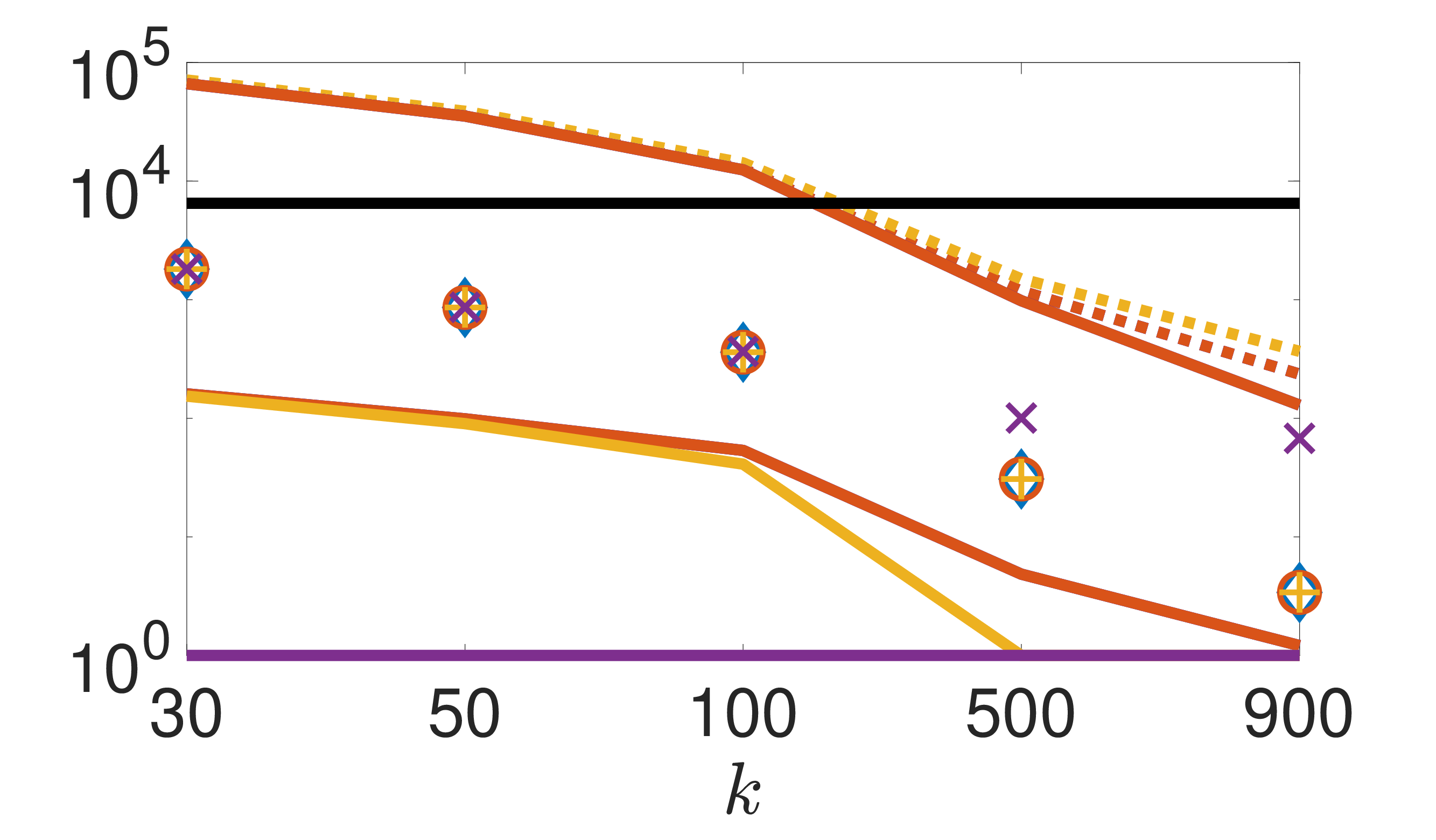}
  \caption{Condition number}
\end{subfigure}
\begin{subfigure}[b]{0.49\linewidth}
  \centering
 \includegraphics[width=\linewidth]{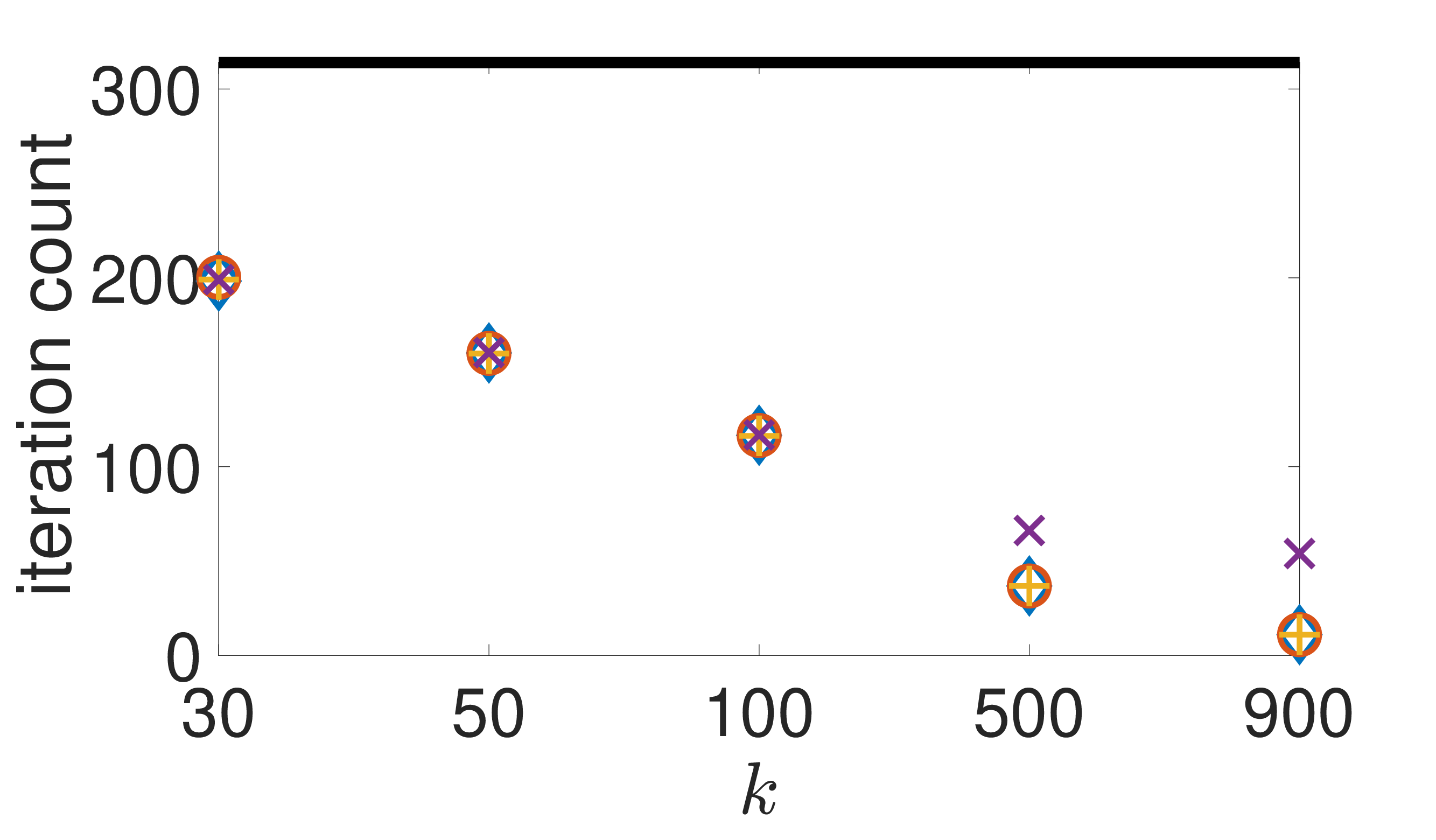}
  \caption{Iteration count}
\end{subfigure}
\begin{subfigure}[b]{0.49\linewidth}
  \centering
 \includegraphics[width=\linewidth]{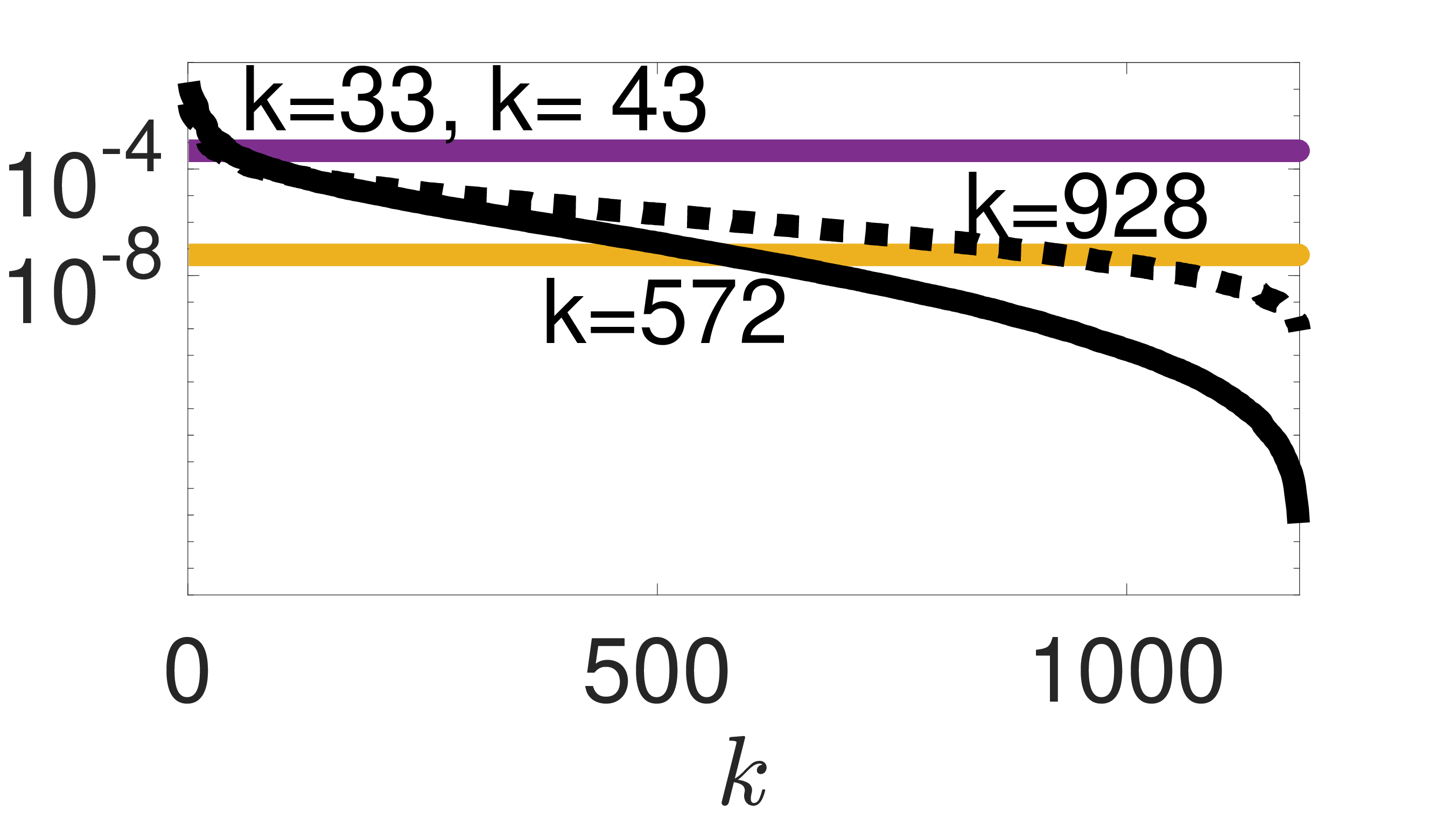}
  \caption{Heuristic}
\end{subfigure}
\begin{subfigure}[b]{0.49\linewidth}
  \centering
 \includegraphics[width=\linewidth]{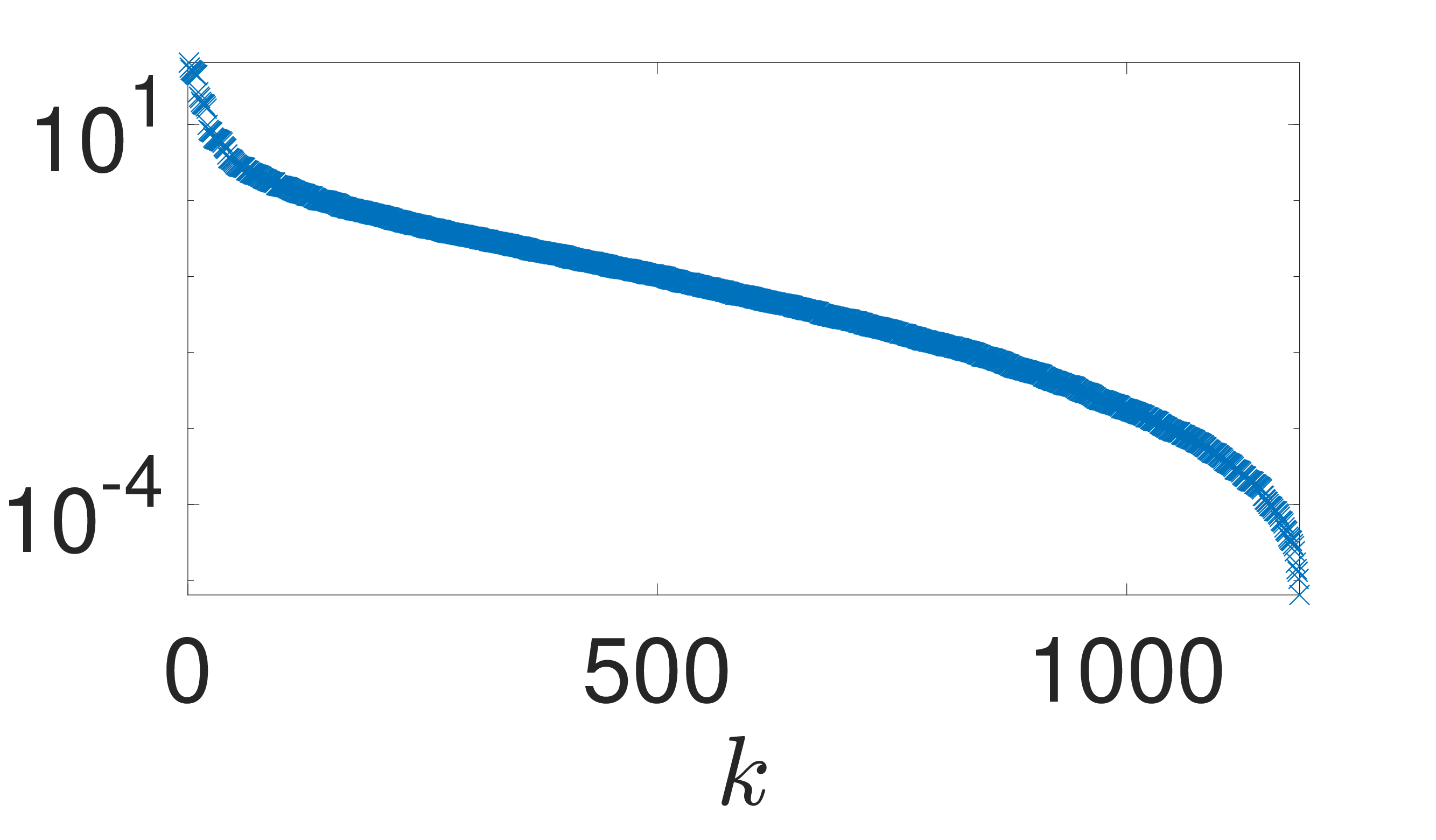}
  \caption{Spectrum}
\end{subfigure}
\end{center}
    \caption{ijcnn1 problem. The top left panel shows Frobenius norm of the mean total error $\|A - \hA_N\|_F$ (markers) versus the rank of approximation $k$. The top right panel shows the mean finite precision error $\|A_N - \hA_N\|_F$ (markers), the mean exact approximation error $\|A-A_N\|_F$ (blue diamonds), and estimates of the finite precision error (dashed lines). 
    The middle left panel shows the condition number of $A +10^{-2} I$ (black solid line) and the mean condition number of $\hPisqr (A + 10^{-2} I) \hPisqr$ (markers) with preconditioner \eqref{eq:prec:LMP_nystrom_shifted_A+muI_inverse} constructed using a rank-$k$ Nystr\"{o}m approximation. The mean estimates of the bounds $b_\text{low}$ (solid), $b_\text{upp}$ (solid; double only) and $b_{uppspd}$ (dotted). The bounds $b_{uppspd}$ with half precision are $\mathcal{O}(10^{8})-\mathcal{O}(10^{9})$ (not shown). The middle right panel shows the mean PCG iteration count.
    The lower left panel shows the right hand-sides of heuristics \eqref{eq:frob_heuristic_full} (black solid) and \eqref{eq:frob_heuristic_no-eig-sum} (black dotted) versus $k$, and $u_p$ (yellow and purple lines). The lower right panel shows the eigenvalues of $A$.
    In all panels, $u_p$ is indicated by the colour of the markers: purple denotes half, yellow denotes single, and red denotes double.}
    \label{fig:ijcnn1_approx}
\end{figure}

\section{Conclusions}\label{sec:conclusions}

In this paper, we considered a mixed precision variant of the single-pass Nystr\"{o}m method for approximating an SPD matrix $A$, where the expensive product with $A$ can be performed in lower precision than the other computations. We bound the total approximation error taking into account the finite precision error. A good quality approximation is obtained when the finite precision error is smaller than the error of the exact Nystr\"{o}m approximation itself. This corresponds to the case when only relatively large eigenvalues are approximated, as suggested by the practical heuristic developed in this work. Numerical examples with both synthetic problems and application problems confirm this observation, and indicate that the standard of using double precision throughout the algorithm may miss opportunities for improving performance.  

We also analysed a randomised limited memory preconditioner constructed with the mixed precision Nystr\"{o}m approximation. We proved bounds on the condition number of the preconditioned coefficient matrix which take into account the finite precision error. The preconditioner is of most interest when it is constructed with a relatively low rank approximation, and in this case using low precision for the product with $A$ does not diminish the quality of the preconditioner in terms of resulting iterations of the conjugate gradient method. 

The limited range of low precision restricts the set of problems that can be considered in the mixed precision framework. This is problematic if setting the low precision to half or even quarter precision is desired. An algorithm that maps a matrix to half precision and preserves symmetry is proposed in \cite{HighamSqueez}. 
Future work on developing scaling algorithms that preserve the positive semidefiniteness property and the structure of the spectrum while also not requiring additionally accesses to the matrix $A$ is of high interest.  

\bibliographystyle{siamplain}
\bibliography{bibl}

\appendix
\section{Details of finite precision analysis}
\label{sec:app_fp}
We provide the details for obtaining the bounds for $\Vert \hC \Vert_F$, $\kappa_2(\hC)$, the step-by-step construction of the computed Nystr\"{o}m approximation taking into account the finite precision error incurred in every operation, and the simplifications of the bound for the finite precision error $\Vert \E \Vert_F$.

\subsection{Bounding $\Vert \hC \Vert_F$}\label{append:cholesky_factor_bound}
We now investigate $\Vert \hC \Vert_F$ by expressing $\hC$ as a perturbed exact Cholesky factor of $\widetilde{B}$. We use \cite[Theorem 1.4]{sun1991perturbation}, which states that given the exact Cholesky decomposition $\widetilde{B} = C^T C$ and if
\begin{equation}\label{eq:cholesky_condition}
    \frac{\Vert \Delta_{Ch} \Vert_F }{\Vert \widetilde{B} \Vert_2} \kappa_2(\widetilde{B}) < 1,
\end{equation}
then the Cholesky decomposition $\widetilde{B} + \Delta_{Ch} = (C + \Delta C_{Ch})^T(C + \Delta C_{Ch}) = \hC^T \hC$ exists and
\begin{equation}\label{eq:deltaC_bound}
    \Vert \Delta C_{Ch} \Vert_F \leq 2^{-1/2} \kappa_2(\widetilde{B}) \frac{\Vert \Delta_{Ch} \Vert_F }{\Vert \widetilde{B} \Vert_2}  \Vert C \Vert_2.
\end{equation}
From \cite[Eq. (10.7)]{HighamBook} we have
\begin{equation}\label{eq:delta_ch_viaB}
    \Vert \Delta_{Ch} \Vert_2 \leq 4k (k+1)u \Vert \widetilde{B} \Vert_2
\end{equation}
and thus using $\Vert \Delta_{Ch} \Vert_F \leq k^{1/2} \Vert \Delta_{Ch} \Vert_2 $, \eqref{eq:cholesky_condition} is transformed to
\begin{equation*}
    4k^{3/2} (k+1)u \kappa(\widetilde{B}) < 1
\end{equation*}
and ignoring the dimensional constant $k$ we need $\kappa(\widetilde{B}) \ll u^{-1}$, which is satisfied under our assumptions. Thus using \eqref{eq:deltaC_bound}, \eqref{eq:delta_ch_viaB},  $\Vert C \Vert_2 = \Vert \widetilde{B} \Vert_2^{1/2}$, and $\Vert C \Vert_F \leq k^{1/2} \Vert C \Vert_2$ we can bound $\Vert \hC \Vert_F$ as
\begin{align}
    \Vert \hC \Vert_F = & \, \Vert C + \Delta C_{Ch} \Vert_F  \leq \Vert C \Vert_F + \Vert \Delta C_{Ch} \Vert_F \nonumber \\
                    \leq & \, k^{1/2} \left(1  + 2^{3/2} k (k+1)u \kappa(\widetilde{B})   \right) \Vert C \Vert_2 \nonumber \\
                     = & \, k^{1/2} \left(1  + 2^{3/2} k (k+1)u \kappa(\widetilde{B})   \right) \Vert \widetilde{B} \Vert_2^{1/2}. \label{eq:Chat_frob_bound1}
\end{align}
Note that using \eqref{eq:forming_Btld}, \eqref{eq:delta_s}, \eqref{eq:B=QTY}, \eqref{eq:delta_B}, and \eqref{eq:Yhat_nu_bound} we have
\begin{equation*}
    \Vert \widetilde{B} \Vert_F \leq \Vert \hB \Vert_F + \Vert \Delta_s \Vert_F \leq (1 + \gamma_n^{(p)} + 3 \gamma_n + 3 \gamma_n \gamma_n^{(p)}) \Vert A \Vert_F \Vert \Omega \Vert_F^2  + (1+ 3 \gamma_n) \hnu \Vert \Omega \Vert_F^2
\end{equation*}
and thus combining this with \eqref{eq:Chat_frob_bound1} gives \eqref{eq:C_hat_frob_bound}.

\subsection{Backtracking the computed approximation}\label{append:backtracking_ANhat}
We backtrack all the computations and refer to the relevant equations in the square brackets.
\begin{align*}
    \hA_N = & \, U \Sigma^2 U^T - \hnu U U^T - \Delta_r \nonumber \\
    \textrm{[\eqref{eq:F=svd}]}    = & \, \hF \hF^T - \hnu U U^T - \Delta_r \nonumber \\
    = & \, \hF \hC \hC^{-1} \hC^{-T} \hC^T \hF^T - \hnu U U^T - \Delta_r \nonumber \\
    = & \, \left( \hY_{\nu} - \hY_{\nu} + \hF \hC \right) \hC^{-1} \hC^{-T} \left( \hY_{\nu} - \hY_{\nu} + \hF \hC \right)^T - \hnu U U^T - \Delta_r \nonumber \\
    = & \, \left( \hY_{\nu} + \Delta_{FC} \right) \hC^{-1} \hC^{-T} \left( \hY_{\nu} + \Delta_{FC} \right)^T - \hnu U U^T - \Delta_r \nonumber \\
        \textrm{[\eqref{eq:chol}]}   = & \, \Big( \hY_{\nu}  + \Delta_{FC} \Big) \Big( \frac{1}{2} \left( \hB + \hB^T \right) + \Delta_{s} + \Delta_{Ch} \Big)^{-1} \Big( \hY_{\nu}  + \Delta_{FC}  \Big)^T - \hnu U U^T - \Delta_r \nonumber \\
         \textrm{[\eqref{eq:B=QTY}]}   = & \, \resizebox{0.87\hsize}{!}{%
        $ \left( \hY_{\nu}  + \Delta_{FC} \right) \Big( \frac{1}{2} \left(  \Omega^T \hY_{\nu} + \Delta_B + 
           \hY_{\nu}^T \Omega + \Delta_B^T \right) + \Delta_{s} + \Delta_{Ch} \Big)^{-1} \left( \hY_{\nu} + \Delta_{FC}  \right)^T $} \nonumber \\
           & \, - \hnu U U^T  - \Delta_r \nonumber \\
         \textrm{[\eqref{eq:Yshifted_compputed}]}   = & \, \left( \hY + \hnu \Omega + \Delta_{\nu}  + \Delta_{FC}  \right) \nonumber \\
          \times  & \, \resizebox{0.87\hsize}{!}{%
        $  \left( \frac{1}{2} \left(  \Omega^T (\hY + \hnu \Omega + \Delta_{\nu}) + \Delta_B + 
           (\hY + \hnu \Omega + \Delta_{\nu})^T \Omega + \Delta_B^T \right) + \Delta_{s} + \Delta_{Ch} \right)^{-1} $} \nonumber \\
           \times  & \,  \left( \hY + \hnu \Omega + \Delta_{\nu} + \Delta_{FC} \right)^T - \hnu U U^T  - \Delta_r \nonumber \\
         \textrm{[\eqref{eq:Yhat}]}   = & \, \Big( (A+ \hnu I) \Omega + \underbrace{\Delta  + \Delta_{\nu}  + \Delta_{FC}  }_{\Delta_1} \Big) \nonumber \\
           \times   & \, \resizebox{0.87\hsize}{!}{%
        $ \Big( \Omega^T (A+ \hnu I )\Omega  + \underbrace{\frac{1}{2} \left(  \Omega^T (  \Delta  + \Delta_{\nu}) + 
           ( \Delta + \Delta_{\nu})^T \Omega + \Delta_B + \Delta_B^T \right) + \Delta_{s} + \Delta_{Ch} }_{\Delta_2} \Big)^{-1} $} \nonumber \\
          \times  & \, \left( (A+ \hnu I) \Omega + \Delta  + \Delta_{\nu}  + \Delta_{FC}  \right)^T - \hnu U U^T - \Delta_r. 
\end{align*}

\subsection{Bounding $\kappa_2(\hC)$}\label{append:cholesky_factor_condition_number}

We wish to express $\kappa_2(\hC)$ via $A$ and $\Omega$. Notice that 
\begin{equation*}
    \hC^T \hC = \Omega^T \left(A +\hnu I \right) \Omega + \Delta_2
\end{equation*}
and 
\begin{equation*}
    \kappa_2(\Omega^T \left(A +\hnu I \right) \Omega + \Delta_2) = \Vert \hC^T \hC  \Vert_2 \Vert \hC^{-1} \hC^{-T}  \Vert_2 = \Vert \hC  \Vert_2^2  \Vert \hC^{-1}  \Vert_2^2 = \kappa_2(\hC)^2.
\end{equation*}
Using this with \cite[Equation 5.4]{cline1979estimate}, we obtain
\begin{align*}
    \kappa_2(\hC) \leq & \, \left( \frac{1+ \epsilon}{1 - \epsilon \kappa_2\left(\Omega^T \left(A +\hnu I \right) \Omega\right)}  \kappa_2\left(\Omega^T \left(A +\hnu I \right) \Omega\right) \right)^{1/2} \nonumber \\
    \leq & \left( \frac{1+ \epsilon}{1 - \epsilon \kappa_2(\Omega^T \left(A +\hnu I \right) \Omega)} \Vert A + \hnu I \Vert_2  \Vert  \Omega \Vert_2^2 \Vert (A_k + \hnu I_k)^{-1} \Vert_2  \Vert \left( W_1^T \Omega \right)^{\dagger} \Vert_2^2 \right)^{1/2} \nonumber \\
    \leq & \ \left( \frac{1+ \epsilon}{1 - \epsilon \kappa_2(\Omega^T \left(A +\hnu I \right) \Omega)}\right)^{1/2} \kappa_2( A_k + \hnu I_k)^{1/2}  \widetilde{\kappa}(\Omega), 
\end{align*}
where $\epsilon$ is defined in \eqref{eq:epsilon_kappaC} and the second inequality is due to \eqref{eq:bound_inverse_XTAX}. Note that $\epsilon \kappa_2(\Omega^T \left(A +\hnu I \right) \Omega) = \Vert  \Delta_2 \Vert_2 \Vert (\Omega^T \left(A +\hnu I \right) \Omega)^{-1} \Vert_2 < 1$ when the assumption \eqref{eq:assumption_1st_order_is_OK} is satisfied.

\subsection{Simplifying the bound}\label{append:simplifying_bound}
We simplify \eqref{eq:finite_prec_error_full_bound} using
\begin{align*}
 & 
\left( \gamma_n + \gamma_n^{(p)} + \gamma_n \gamma_n^{(p)} +  \left( 1  + \gamma_n^{(p)}\right) \gamma_k  \kappa_F(\hC) \left( 1 +  \frac{\gamma_k k^{1/2} \kappa_2(\hC)}{1 - \gamma_k k^{1/2} \kappa_2(\hC)} \right) \right)^2  \\
   &  \times \Vert (\Lambda_k + \hnu I_k)^{-1} \Vert_2 \widetilde{\kappa}( \Omega)^2 \Vert A \Vert_F^2 \\
  &  = \left( \gamma_n + \gamma_n^{(p)} + \gamma_n \gamma_n^{(p)} +  \left( 1  + \gamma_n^{(p)}\right) \gamma_k  \kappa_F(\hC) \left( 1 +  \frac{\gamma_k k^{1/2} \kappa_2(\hC)}{1 - \gamma_k k^{1/2} \kappa_2(\hC)} \right) \right)^2 \\
    & \times \frac{ \kappa_2 (\Lambda_k + \hnu I_k)}{\Vert \Lambda_k + \hnu I_k \Vert_2} \widetilde{\kappa}( \Omega)^2 \Vert A \Vert_F^2  \\
   &  \leq \left( \gamma_n + \gamma_n^{(p)} + \gamma_n \gamma_n^{(p)} +  \left( 1  + \gamma_n^{(p)}\right) \gamma_k  \kappa_F(\hC) \left( 1 +  \frac{\gamma_k k^{1/2} \kappa_2(\hC)}{1 - \gamma_k k^{1/2} \kappa_2(\hC)} \right) \right) \\
    & \times \kappa_2 (\Lambda_k + \hnu I_k)^{1/2} \widetilde{\kappa}( \Omega) \Vert A \Vert_F,
\end{align*}
because
\begin{align*}
    1\geq & \left( \gamma_n + \gamma_n^{(p)} + \gamma_n \gamma_n^{(p)} +  \left( 1  + \gamma_n^{(p)}\right) \gamma_k  \kappa_F(\hC) \left( 1 +  \frac{\gamma_k k^{1/2} \kappa_2(\hC)}{1 - \gamma_k k^{1/2} \kappa_2(\hC)} \right) \right) \\
    & \times \frac{ \kappa_2 (\Lambda_k + \hnu I_k)^{1/2} \widetilde{\kappa}( \Omega) \Vert A \Vert_F }{ \Vert \Lambda_k + \hnu I_k \Vert_2} 
\end{align*}
when \eqref{cond:u_p_kappaAk} and \eqref{cond:u_kappaAk} hold. Under the latter condition, an equivalent argument is applied to 
\begin{equation*}
   \left(  \gamma_n  + \gamma_k \kappa_F(\hC)  \left( 1 + \frac{\gamma_k k^{1/2} \kappa_2(\hC)}{1 - \gamma_k k^{1/2} \kappa_2(\hC)} \right)  \right)^2 \hnu^2 \Vert (\Lambda_k + \hnu I_k)^{-1} \Vert_2 \widetilde{\kappa}( \Omega)^2
\end{equation*}
and we obtain
\begin{align*}
  &   \Vert \hA_N - (A+ \hnu I)_N \Vert_F
        \leq \\
        & \, 
            6  k^{1/2} \left( \gamma_n + \gamma_n^{(p)} + \gamma_n \gamma_n^{(p)} +  \left( 1  + \gamma_n^{(p)}\right) \gamma_k  \kappa_F(\hC) \left( 1 +  \frac{\gamma_k k^{1/2} \kappa_2(\hC)}{1 - \gamma_k k^{1/2} \kappa_2(\hC)} \right) \right)  \nonumber \\
    & \times \kappa_2(A_k + \hnu I)^{1/2} \widetilde{\kappa}( \Omega) \Vert A \Vert_F   \nonumber \\
						 & \, +  k^{1/2} \left( \gamma_n^{(p)} + 3 \gamma_n + k \gamma_{k+1} + 3 \gamma_n \gamma_n^{(p)} + k \gamma_{k+1} \gamma_n^{(p)}  \right) \kappa_2(A_k + \hnu I) \widetilde{\kappa}( \Omega)^2 \Vert A \Vert_F  \nonumber \\
         & \, + 3 k^{1/2} \left(  \gamma_n  + \gamma_k \kappa_F(\hC)  \left( 1 + \frac{\gamma_k k^{1/2} \kappa_2(\hC)}{1 - \gamma_k k^{1/2} \kappa_2(\hC)} \right)  \right)\hnu \kappa_2(A_k + \hnu I)^{1/2} \widetilde{\kappa}( \Omega) \nonumber \\
             &\, +  k^{1/2} \left( 3 \gamma_n + k \gamma_{k+1} \right) \hnu  \kappa_2(A_k + \hnu I) \widetilde{\kappa}( \Omega)^2  \nonumber \\
						  & \, + 2 k^{1/2} \hnu \\
        \lesssim  & \, 
             6  k^{1/2} \left( \gamma_n + \gamma_n^{(p)} + \gamma_n \gamma_n^{(p)} \right) \kappa_2(A_k + \hnu I)^{1/2} \widetilde{\kappa}( \Omega) \Vert A \Vert_F  \nonumber \\             
           &\, + 6  k^{3/2} \left( 1  + \gamma_n^{(p)}\right) \gamma_k  \left( \frac{1+ \epsilon}{1 - \epsilon \kappa_2(\Omega^T \left(A +\hnu I \right) \Omega)}\right)^{1/2}  \kappa_2( A_k + \hnu I_k)  \widetilde{\kappa}(\Omega)^2    \Vert A \Vert_F  \nonumber \\
						 & \, +  k^{1/2} \left( \gamma_n^{(p)} + 3 \gamma_n + k \gamma_{k+1} + 3 \gamma_n \gamma_n^{(p)} + k \gamma_{k+1} \gamma_n^{(p)}  \right) \kappa_2(A_k + \hnu I) \widetilde{\kappa}( \Omega)^2 \Vert A \Vert_F  \nonumber \\
         & \, +     3  k^{1/2} \left(  \gamma_n  + k \gamma_k \left( \frac{1+ \epsilon}{1 - \epsilon \kappa_2(\Omega^T \left(A +\hnu I \right) \Omega)}\right)^{1/2} \kappa_2( A_k + \hnu I_k)^{1/2}  \widetilde{\kappa}(\Omega)  \right) \nonumber \\
            \,    & \times \hnu \kappa_2(A_k + \hnu I)^{1/2} \widetilde{\kappa}( \Omega)  \nonumber \\
             &\, +  k^{1/2} \left( 3 \gamma_n + k \gamma_{k+1} \right) \hnu  \kappa_2(A_k + \hnu I) \widetilde{\kappa}( \Omega)^2  \nonumber \\
						  & \, + 2 k^{1/2} \hnu,
        \label{eq:fpNystrom_error_bound}
\end{align*}
where the second inequality is due to \eqref{eq:kappahC_bound} and ignoring $u^2$ terms. We further simplify the bound using assumptions $u \leq u_p$, $\hnu \leq c(n,k) u_p \Vert A \Vert_F  \Vert \Omega \Vert_F^2$ and $ \widetilde{\gamma}_n^{(p)}  \geq k \gamma_k  \left( \frac{1+ \epsilon}{1 - \epsilon \kappa_2(\Omega^T \left(A +\hnu I \right) \Omega)}\right)^{1/2}$.

\section{Proof of Theorem~\ref{prop:prec_condition_no_deterministic}}\label{append_condition_number_finite_precision}
The main idea of the proof of Theorem~\ref{prop:prec_condition_no_deterministic} closely follows the proof in \cite[Section A.1.1.]{Frangella2021}. We provide the full proof accounting for the finite precision error. 

We first obtain the upper bounds for the condition number.  
Since $A= \hA_N + E + \E$, Weyl's inequality gives the bound 
\begin{align*}
    \lambmax (\hPisqr (A + \mu I) \hPisqr) &\leq 
    \lambmax(\hPisqr (\hA_N + \mu I) \hPisqr) 
    \\&\phantom{\leq}+\lambmax(\hPisqr E \hPisqr) + \lambmax(\hPisqr \E \hPisqr).
\end{align*}
The eigenvalues $\lambmax(\hPisqr (\hA_N + \mu I) \hPisqr)$ and $\lambmax(\hPisqr E \hPisqr)$ can hence be bounded as in \cite{Frangella2021}, because $\hPisqr$ is constructed with eigenpairs of $\hA_N$ and $E$ is positive semidefinite and thus $\hPisqr E \hPisqr$ is also positive semidefinite. We thus have 
\begin{equation}\label{eq:bound_prec_Amu+E}
    \lambmax(\hPisqr (\hA_N + \mu I) \hPisqr) + \lambmax(\hPisqr E \hPisqr) \leq \hlambda_k + \mu + \| E \|_2.
\end{equation}
Since we assume that $\hU$ has orthogonal columns and thus $\sigmax(\hPi) = 1$ when $k<n$, we have
\begin{align}
  \lambmax(\hPisqr \E \hPisqr) & = \, \lambmax (\hPi \E) \nonumber \\
  &  \leq \, \max_i | \lambda_i (\hPi \E)| \nonumber\\
  &  \leq \, \sigmax (\hPi \E) \nonumber\\
  &  \leq  \, \sigmax(\hPi) \sigmax(\E) \nonumber\\
  &  = \, \sigmax(\E) \nonumber\\
  &  = \, \| \E \|_2. \label{eq:bound_prec_Epsi}
\end{align}
Combining \eqref{eq:bound_prec_Amu+E} and \eqref{eq:bound_prec_Epsi} gives
\begin{equation}\label{eq:lambdamax_PAmuP_upper_bound}
    \lambmax (\hPisqr (A + \mu I) \hPisqr) \leq \hlambda_k + \mu + \| E \|_2 + \| \E \|_2. 
\end{equation}
We bound $\lambmin (\hPisqr (A + \mu I) \hPisqr)$ using Weyl's inequality and the facts that we have $\lambmin(\hPisqr (\hA_N + \mu I ) \hPisqr) \geq \mu$ and $\lambmin(\hPisqr E \hPisqr) \geq 0$ as follows:
\begin{align}
     \lambmin(\hPisqr (A + \mu I) \hPisqr)  &  \geq \,
     \lambmin(\hPisqr (\hA_N + \mu I) \hPisqr)  \nonumber\\
     & \quad + \lambmin(\hPisqr E \hPisqr) + \lambmin(\hPisqr \E \hPisqr) \nonumber\\
     &  \geq \, \mu + \lambmin(\hPisqr \E \hPisqr) \nonumber\\
      &  \geq \, \mu - \| \E \|_2.  \label{eq:lambdamin_PAmuP_lower_bound}
\end{align}
Assuming that $ \mu > \| \E \|_2$, from \eqref{eq:lambdamax_PAmuP_upper_bound} and \eqref{eq:lambdamin_PAmuP_lower_bound} we have
\begin{equation*}
    \kappa(\hPisqr (A + \mu I) \hPisqr) \leq \frac{\hlambda_k + \mu + \| E \|_2 + \| \E \|_2}{\mu - \| \E \|_2} = 1 + \frac{\hlambda_k + \| E \|_2 + 2\| \E \|_2}{\mu - \| \E \|_2 },
\end{equation*}
which proves the upper bound in \eqref{eq:condNobounds-lower-upper}. 

The condition $ \mu > \| \E \|_2$ may not be satisfied for small values of $\mu$. It can be avoided 
following the argument in \cite{Frangella2021}. We consider
\begin{align}
   \lambmin \left(\hPisqr (A + \mu I) \hPisqr \right) & = \, \lambmin \left( (A + \mu I )^{1/2} \hPi (A + \mu I)^{1/2} \right) \nonumber\\
                                                     & = \, \frac{1}{\lambmax \left((A + \mu I)^{-1/2} \hP (A + \mu I)^{-1/2} \right)} \label{eq:lmineq}
\end{align}
and bound $\lambmax \left( (\Amu)^{-1/2} \hP (\Amu)^{-1/2} \right)$ from above. Using Weyl's inequality, we obtain
\begin{align}
    \lambmax &\left((\Amu)^{-1/2} \hP (\Amu)^{-1/2} \right)  \nonumber \\  
  \hspace{-1pt}  &=\lambmax \left((\Amu)^{-1/2} \left( \frac{1}{\hlambda_k + \mu} \left( \hA_N \hspace{-1pt} + \hspace{-1pt} \mu \hU \hU^T \right) + \left( I - \hU \hU^T \right) \right) (\Amu)^{-1/2}\right) \nonumber\\
   &\leq   \frac{1}{\hlambda_k + \mu} \lambmax \left( (\Amu)^{-1/2} \left( \hA_N + \mu \hU \hU^T \right)  (\Amu)^{-1/2} \right) \label{eq:max_eig_prec1}\\
   &\phantom{\leq} +  \lambmax \left( (\Amu)^{-1/2} \left( I - \hU \hU^T \right)  (\Amu)^{-1/2} \right). \label{eq:max_eig_prec}
\end{align}
To bound 
\eqref{eq:max_eig_prec1}, we use Weyl's inequality again, giving 
\begin{align*}
\lambmax \Big( \Big. (&\Amu)^{-1/2} \left( \hA_N + \mu \hU \hU^T \right)  (\Amu)^{-1/2} \Big. \Big) \\
&= \lambmax \left( (\Amu)^{-1/2} \left( \hA_N + \E - \E + \mu \hU \hU^T \right)  (\Amu)^{-1/2} \right)\\
     &\leq   \lambmax \left( (\Amu)^{-1/2} \left( \hA_N + \E + \mu \hU \hU^T \right)  (\Amu)^{-1/2} \right)  \\
     &\phantom{\leq}+  \lambmax \left( - (\Amu)^{-1/2} \E   (\Amu)^{-1/2} \right).          
\end{align*}
Note that $\hA_N + \E = A_N$ is the Nystr\"{o}m approximation of $A$ in infinite precision and thus $A - A_N$ is positive semidefinite. Then 
\begin{equation*}
    (\Amu) - (\hA_N + \E + \mu \hU \hU^T ) = A - A_N + \mu (I - \hU \hU^T )
\end{equation*}
is also positive semidefinite, because $I - \hU \hU^T$ is an orthogonal projector. Hence $I - (\Amu)^{-1} (\hA_N + \E + \mu \hU \hU^T ) $ is positive semidefinite and by \cite[Section 7.7]{HornJohnson},
\begin{equation}\label{eq:max_eig_prec3}
    \lambmax \left( (\Amu)^{-1/2} \left( \hA_N + \E + \mu \hU \hU^T \right)  (\Amu)^{-1/2} \right) \leq 1.
\end{equation}
Further,
\begin{align}
    \lambmax \left( - (\Amu)^{-1/2} \E   (\Amu)^{-1/2} \right) \leq & \ \max_i \left| \lambda_i \left((\Amu)^{-1/2} \E   (\Amu)^{-1/2} \right) \right| \nonumber\\
                                                                    = & \ \max_i \left| \lambda_i \left((\Amu)^{-1} \E \right) \right| \nonumber\\
                                                                    \leq & \ \sigmax \left( (\Amu)^{-1} \E \right)  \nonumber\\
                                                                    \leq & \|(\Amu)^{-1} \|_2 \| \E \|_2 \nonumber\\
                                                                    = & \frac{\| \E \|_2}{\lambmin(A) + \mu}. \label{eq:max_eig_prec4}
\end{align}     
The term in \eqref{eq:max_eig_prec} is bounded as in \cite{Frangella2021}, because $I - \hU \hU^T$ is an orthogonal projector and $\| I - \hU \hU^T \|_2 = 1$, that is,
\begin{equation}\label{eq:max_eig_prec5}
    \lambmax \left( (\Amu)^{-1/2} \left( I - \hU \hU^T \right)  (\Amu)^{-1/2} \right) \leq \frac{1}{\lambmin(A) + \mu}.
\end{equation}
From \eqref{eq:max_eig_prec3}, \eqref{eq:max_eig_prec4}, and \eqref{eq:max_eig_prec5}, we obtain
\begin{equation}
     \lambmax \left( (\Amu)^{-1/2} \hP (\Amu)^{-1/2} \right) \leq \frac{1}{\hlambda_k + \mu} + \frac{\| \E \|_2 +1 }{\lambmin(A) + \mu}, \label{eq:lmaxmin}
\end{equation}
and thus combining \eqref{eq:lmaxmin}, \eqref{eq:lmineq}, and \eqref{eq:lambdamax_PAmuP_upper_bound}, we have  
\begin{equation*}
\kappa(\hPisqr (\Amu) \hPisqr) \leq  \left(  \hlambda_k + \mu + \| E \|_2 + \| \E \|_2 \right) \left( \frac{1}{\hlambda_k + \mu} + \frac{\| \E \|_2 +1 }{\lambmin(A) + \mu} \right), 
\end{equation*}
which proves \eqref{eq:condNobounds-spd}.

We now obtain the lower bound for the condition number in \eqref{eq:condNobounds-lower-upper}. The lower bound for $\lambmax (\hPisqr (\Amu) \hPisqr)$ is acquired using the same ideas as for \linebreak $\lambmin(\hPisqr (\Amu) \hPisqr)$ above, that is,
\begin{align*}
    \lambmax (\hPisqr (\Amu) \hPisqr)   &\geq  \lambmax(\hPisqr (\hAmu + E) \hPisqr) \\
    &\phantom{\geq}+ \lambmin(\hPisqr \E \hPisqr)\\ 
    &\geq  \lambmax(\hPisqr (\hAmu) \hPisqr) + \lambmin(\hPisqr E \hPisqr)\\ 
    &\phantom{\geq} + \lambmin(\hPisqr \E \hPisqr) \\
    &\geq \hlambda_k + \mu - \| \E \|_2.
\end{align*}
The upper bound for $\lambmin(\hPisqr (\Amu) \hPisqr) = \lambmin((\Amu)\hPi)$ is obtained as in \cite{Frangella2021}, namely
\begin{equation*}
    \lambmin((\Amu)\hPi) \leq \lambmin(\Amu) \lambmax (\hPi) = \lambmin(A) + \mu.
\end{equation*}
We thus have
\begin{equation*}
   \kappa(\hPisqr (\Amu) \hPisqr) \geq \max \left\{1, \frac{\hlambda_k + \mu - \| \E \|_2}{\mu + \lambda_{min} (A) } \right\},
\end{equation*}
because the condition number is always at least 1.

\end{document}